\documentclass[reqno]{amsart}

\usepackage{amsfonts}
\usepackage[all]{xy}
\usepackage{amssymb}
\usepackage{amsmath}
\usepackage{epsfig}
\usepackage{amscd}
\usepackage{graphicx, color}
\usepackage{epstopdf}
\usepackage{amsmath,mathtools}
\usepackage{hyperref}
\usepackage{tikz}

\usepackage{amssymb}
\usepackage{amsmath}
\usepackage{mathrsfs}
\usepackage{epsfig}
\usepackage{amscd}
\usepackage{graphicx}
\usepackage{xcolor}

\usepackage{enumerate}
\usepackage{marginnote}
\usepackage[normalem]{ulem}
\newcommand{\msout}[1]{\ifmmode\text{\sout{\ensuremath{#1}}}\else\sout{#1}\fi}
\newcommand{\cancel}[1]{
	\ifmmode
	{\color{red  }
		\msout{#1}}
	\else
	{\color{red}
		\sout{#1}}
	\fi
}

\newcommand{\replace}[2]{
	\ifmmode
	{\color{red}\msout{#1}}{\color{blue}\uwave{#2}}
	\else
	{\color{red}\sout{#1}}{\color{blue}\uwave{#2}}
	\fi
}

\numberwithin{equation}{section}

\theoremstyle{theorem}
\newtheorem{thm}{Theorem}[section]
\newtheorem{cor}[thm]{Corollary}
\newtheorem{lem}[thm]{Lemma}
\newtheorem{prop}[thm]{Proposition}

\newtheorem{conj}[thm]{Conjecture}

\theoremstyle{definition}
\newtheorem{defn}[thm]{Definition}
\newtheorem{rem}[thm]{Remark}

\newcommand{\ben}{\begin{equation*}}
\newcommand{\een}{\end{equation*}}
\newcommand{\be}{\begin{equation}}
\newcommand{\ee}{\end{equation}}
\newcommand{\bea}{\begin{eqnarray}}
\newcommand{\eea}{\end{eqnarray}}
\newcommand{\beastar}{\begin{eqnarray*}}
\newcommand{\eeastar}{\end{eqnarray*}}
\newcommand{\bc}{\begin{center}}
\newcommand{\ec}{\end{center}}

\def\R{\mathbb{R}}
\def\Z{\mathbb{Z}}

\def\D{\mathbb{D}}
\def\T{\mathbb{T}}
\def\H{\mathbb{H}}
\def\C{\mathbb{C}}

\def\R{\mathbb{R}}
\def\Z{\mathbb{Z}}

\def\D{\mathbb{D}}
\def\T{\mathbb{T}}
\def\C{\mathbb{C}}

\def\H{{\mathbb{H}}} %hyperbolic space
\def\Ch{{\widehat{\C}}}

\def\Int{\operatorname{Int}}

\newcommand{\del}{\partial}

\def\CA{{\mathcal A}}

\def\CF{{\mathcal F}}

\def\CM{{\mathcal M}}

\def\CP{{\mathcal P}}

\def\CP{{\mathcal P}}

\def\CW{{\mathcal W}}

\def\opname#1{\mathop{\kern0pt{\rm #1}}\nolimits}

\def\Im{\opname{Im}}

\def\supp{\operatorname{supp}}

\def\leng{\operatorname{leng}}

\def\coker{\operatorname{Coker}}

\def\Chord{\operatorname{Chord}}

\def\Image{\operatorname{Image}}

\def\Gam{{\Gamma}} %discrete subgroup in H^3
\def\pD{{\mathcal D}}

\DeclareMathOperator{\Is}{Isom\raisebox{0.1em}{$\hspace{-0.2em}^+$}\hspace{-0.2em}}

\newcommand{\eqsn}[2][c]{
	\ifmmode{
		\begin{aligned}[#1] #2 \end{aligned}}
	\else{\begin{align*} #2 \end{align*}}
	\fi}

\newcommand{\eqs}[2][c]{
	\ifmmode{
		\begin{aligned}[#1] #2 \end{aligned}}
	\else{\begin{align} #2 \end{align}}
	\fi}
\newcommand{\psl}[1][\C]{\textup{PSL}(2,#1)}
\newcommand{\Hb}[1][3]{\overline{\mathbb{H}^#1}}

\newcommand{\bdH}[1][3]{\partial\overline{\mathbb{H}^#1}}

\newcommand{\hr}{\H^2\times\R}
\newcommand{\w}{\widetilde}
\newcommand{\cord}{{\rm Cord}}

\newcommand{\ba}{\begin{equation}\begin{aligned}}
\newcommand{\ea}{\end{aligned}\end{equation}}
\newcommand{\ban}{\begin{equation*}\begin{aligned}}
\newcommand{\ean}{\end{aligned}\end{equation*}}

\begin{document}

\title[Formality of Floer complex]
{Formality of Floer complex of the ideal boundary of
hyperbolic knot complement}

\author{Youngjin Bae, Seonhwa Kim, Yong-Geun Oh}

\thanks{SK and YO are supported by the IBS project IBS-R003-D1.
YO is also partially supported by the National Science Foundation under Grant No. DMS-1440140 during his residence at the Mathematical Sciences Research Institute in Berkeley, California in the fall of 2018.
YB was partially supported by IBS-R003-D1 and JSPS International Research Fellowship Program.}

\address{Youngjin Bae\\
   Research Institute for Mathematical Sciences, Kyoto University\\
   Kyoto Prefecture, Kyoto, Sakyo Ward, Kitashirakawa Oiwakecho, Japan 606-8317}
\email{ybae@kurims.kyoto-u.ac.jp}
\address{Seonhwa Kim\\
Center for Geometry and Physics, Institute for Basic Sciences (IBS), Pohang, Korea} \email{ryeona17@ibs.re.kr}
\address{Yong-Geun Oh\\
Center for Geometry and Physics, Institute for Basic Sciences (IBS), Pohang, Korea \& Department of Mathematics,
POSTECH, Pohang, Korea} \email{yongoh1@postech.ac.kr}

\date{January 4, 2019; Revised on March 4, 2019}

\begin{abstract} This is a sequel to the authors' article \cite{BKO}.
We consider a hyperbolic knot $K$ in a closed 3-manifold $M$ and the
cotangent bundle of its complement $M \setminus K$.
We equip $M \setminus K$ with a hyperbolic metric $h$ and its cotangent bundle
$T^*(M \setminus K)$ with the induced kinetic energy Hamiltonian $H_h = \frac{1}{2} |p|_h^2$ and Sasakian almost complex structure $J_h$, and associate a wrapped Fukaya category to $T^*(M\setminus K)$
whose wrapping is given by $H_h$. We then consider the conormal $\nu^*T$ of a horo-torus $T$ as its object.
We prove that all non-constant Hamiltonian chords are transversal and of Morse index 0
relative to the horo-torus $T$, and so that the structure maps satisfy
$\widetilde{\mathfrak m}^k = 0$ unless $k \neq 2$ and an $A_\infty$-algebra associated to $\nu^*T$
is reduced to a noncommutative algebra concentrated to degree 0.
We prove that the wrapped Floer cohomology $HW(\nu^*T; H_h)$ with respect to $H_h$ is well-defined
and isomorphic to the Knot Floer cohomology $HW(\del_\infty(M \setminus K))$ that was introduced
in \cite{BKO} for arbitrary knot $K \subset M$. We also define a reduced cohomology,
denoted by $\widetilde{HW}^d(\del_\infty(M \setminus K))$, by modding out constant chords and
prove that if $\widetilde{HW}^d(\del_\infty(M \setminus K))\neq 0$ for some $d \geq 1$,
then $K$ cannot be hyperbolic. On the other hand, we prove that all torus knots
have $\widetilde{HW}^1(\del_\infty(M \setminus K)) \neq 0$.
\end{abstract}



\maketitle

\setcounter{tocdepth}{1}
\tableofcontents

\section{Introduction}

The symplectic idea of constructing knot invariants using the conormal lift of a knot (or link) in $\R^3$
as a Legendrian submanifold in the unit cotangent bundle has been explored in symplectic-contact geometry,
especially exploited by Ekholm-Etnyre-Ng-Sullivan \cite{EENS}
in their construction of knot contact homology who proved that this analytic invariant
recovers Ng's combinatorial invariants of the knot \cite{lenny-ng}. It has been also observed
(see \cite{ENS} for example) that the data of the knot contact homology can be obtained from
a version of wrapped Fukaya category on the ambient space, the symplectization of the unit
cotangent bundle $ST^*\R^3$ or of an open subset thereof.

In \cite{BKO}, the authors considered the knot complement $M \setminus K$
of arbitrary orientable closed 3-manifold $M$ directly, and constructed
its associated Fukaya category on it. We emphasize that the base space
\emph{$N:=M \setminus K$ is non-compact}. We take a tubular neighborhood
$N(K)$ of $K$ and consider its boundary $T: = \del(N(K))$.
We define a cylindrical adjustment $g_0$ of the induced metric $g|_{M \setminus K}$
of a smooth metric $g$ of $M$. (See Section \ref{sec:comparison} for the precise definition thereof.)
Then the construction in \cite{BKO} associates an $A_\infty$ algebra
$$
CW_g(\nu^*T,T^*(M\setminus K)): = CW(\nu^*T,T^*(M\setminus K); H_{g_0}).
$$
We denote the associated cohomology by
$$
HW_g(\nu^*T,T^*(M\setminus K)): = HW(\nu^*T,T^*(M\setminus K); H_{g_0}).
$$
It was shown in \cite{BKO} that this cohomology does not depend on the choices of
smooth metric $g$ on $M$, of the tubular neighborhood $N(K)$ but depends on the
isotopy class of knot $K$. In particular, we defined  the wrapped Floer cohomology
as an invariant of the knot $K$.We denote the resulting common graded group by
$$
 HW(\del_\infty(M\setminus K)) = \bigoplus_{d=0}^\infty HW^d(\del_\infty(M\setminus K))
$$
which is called the \emph{knot Floer algebra }in \cite{BKO}. Since the group is independent of
the choice of tubular neighborhood of $K$, one may regard this group as the Moore homology version
(in the horizontal direction) of the wrapped Floer cohomology of the asymptotic boundary
$
\del_\infty(M \setminus K)
$
of non-compact manifold $M \setminus K$. (See \cite{BKO}.)

\subsection{Formality of Floer complex $CW(\nu^*T;H_h)$} \label{sec:statemainresult}

In this paper, we specialize our focus on the case of \emph{hyperbolic knots}, i.e.,
of the knots $K\subset M$ (or links) such that the complement $N$ admits a complete metric
of constant curvature $-1$. We exploit the presence of hyperbolic metric $h$ on the complement
$M\setminus K$ for the computation of $HW(\del_\infty(M\setminus K))$,
even though the metric $h$ cannot be smoothly extended to $M$ itself.
In other words,
the wrapping we put in the definition of wrapped Floer cohomology is
of different nature from that of \cite{BKO}.

We utilize the special geometry of hyperbolic metric in the calculation of
the associated $A_\infty$ structures, by considering ($CW(\nu^*T),\frak m^k$)
associated to a special choice of the above mentioned tubular neighborhood $N(K)$ so that each component of whose
boundary $\del N(K)$ is given by a horo-torus contained in the cusp-neighborhood of $K$
with respect to the hyperbolic metric.
Although we will mostly restrict ourselves to the case of knots for the simplicity of exposition,
we would like to emphasize that main results of the present paper also apply to the links whose
ramification to the study of links is worthwhile to investigate.

Consider the kinetic energy Hamiltonian $H = \frac12|p|_h^2$ of the hyperbolic metric $h$ on $M \setminus K$.
We first prove the following general properties of the Hamiltonian
chords associated to the conormal  $\nu^*T$  and their associated geodesic cords attached to $T$.
\footnote{We will follow the terminology adopted by Ng \cite{lenny-ng} the term \emph{chord} for the Hamiltonian
trajectory attached to the conormal and the term \emph{cord} for the corresponding geodesic attached to the base of the
conormal.}

It follows from an argument in hyperbolic geometry and using the special geometry of horo-torus that
 for each $e \in \pi_1(M\setminus K, T)$, there exists a unique
geodesic cord $c_e$ of $T$ attached to $T$ for each homotopy class $e \in \pi_1(N,T)$.

\begin{thm}[Theorem \ref{thm:index-nullity}]\label{thm:nondeg}
Let $N$ and $T$ be as above. Then
for any geodesic cord $c \in \text{\rm Cord}(T)$, both
Morse index and nullity of $c$ vanish. In particular, any non-constant Hamiltonian
chord associated to $\nu^*T$ is rigid and nondegenerate.
\end{thm}

This enables us to work with the kinetic energy Hamiltonian, without perturbation,
in the construction of an $A_\infty$ algebra generated by
the set $\text{Chord}(H;\nu^*T)$ of Hamiltonian chords of $H$ attached to $L:=\nu^*T$.
The relevant perturbed pseudo-holomorphic equation is nothing but
\be\label{eq:CRXH}
(du - \beta \otimes X_H)^{(0,1)} = 0
\ee
for a map $u:\Sigma \to T^*N$ satisfying suitable (moving) Lagrangian boundary
condition together with asymptotic conditions converging to Hamiltonian chords of
the kinetic energy Hamiltonian $H_h$ given above.
We especially study the Cauchy-Riemann equation
\eqref{eq:CRXH} with respect to the Sasakian almost complex structure on $T^*N$ of $h$ on $N$: It is given by
\be\label{eq:SasakianJg}
J_h (X) = X^\flat, \quad J_h(\alpha) = - \alpha^\sharp
\ee
under the splitting $T(T^*N) \simeq TN \oplus T^*N$ via the Levi-Civita connection of $h$.

The graded Floer chain complex $CW(L;H_h)$ is a free abelian group generated by the Hamiltonian chords:
$$
CW(L;H_h) : = C^*(T) \oplus \bigoplus_{x\in\Chord^*(L;H_h)}\Z \cdot x,
$$
where $C^*(T)$ is a cochain complex of $T \cong \T^2$, e.g., the de Rham complex of $T$
or the Morse complex of a Morse function of $T$.
Here the grading is given by the grading of the Hamiltonian chords $|x|$. We establish the
$C^0$ estimates, especially the horizontal $C^0$ estimate, in Section \ref{subsec:z-coordinate},
which enables us to directly define the wrapped Floer complex as an $A_\infty$ algebra
for the hyperbolic metric, \emph{without making a cylindrical adjustment} unlike in \cite{BKO}.

We show that $(C^*(T),d)$ forms a sub-complex of $CW(L;H_h)$ and so can define the
reduced complex
$$
\widetilde{CW}^*(L;H_h) = CW(L;H_h)/ C^*(T)
$$
and denote its induced $A_\infty$ operators by $\widetilde{\mathfrak m}^k$ and
its cohomology by $\widetilde{HW}^*(L;H_h)$. We then prove

\begin{thm}[Theorem \ref{thm:comparison}]\label{thm:comparison-intro} Suppose $K$ is a hyperbolic knot on $M$. Then
we have an (algebra) isomorphism
$$
HW^d(\nu^*T;H_h) \cong HW^d(\del_\infty(M \setminus K))
$$
for all integer $d \geq 0$. In particular $\widetilde{HW}^d(\del_\infty(M \setminus K)) = 0$ for all $d > 0$ and
$\widetilde{HW}^0(\del_\infty(M \setminus K))$ is a free abelian group generated by $\mathscr G_{M\setminus K}$.
\end{thm}

One main ingredient of the proof of this theorem is
the following formality of $(\widetilde{CW}^*(L;H_h), \{\widetilde{\mathfrak m}^k\}_{k=1}^\infty)$
we prove in the present paper.

\begin{thm}[Theorem \ref{thm:formal}]\label{thm:formality}  Let
$h$ be the hyperbolic metric of $N$ and $J=J_h$ be the Sasakian almost
complex structure of $h$ on $T^*N$, and let $T$ be a horo-torus as above.
Consider the kinetic energy Hamiltonian associated to the metric $h$
$$
H_h(q,p) = \frac{1}{2} |p|_h^2
$$
and the associated perturbed Cauchy-Riemann equations \eqref{eq:CRXH} equipped with
some boundary condition associated to the conormal $\nu^*T$.
Let $\widetilde{\mathfrak m} = \{\widetilde{\mathfrak m}^k\}_{k=1}^\infty$ be the corresponding $A_\infty$ maps.
Then we have $\widetilde{\mathfrak m}^k = 0$ for all $k \neq 2$.
\end{thm}

This theorem is a consequence of the standard Fredholm theory
combined with the geometric properties of the hyperbolic metric $h$  stated in
Theorem \ref{thm:nondeg} in the study of moduli space of \eqref{eq:CRXH}.
(See Section \ref{sec:index} for details.)

The following is an immediate corollary of Theorem \ref{thm:nondeg}, Theorem \ref{thm:comparison-intro}
and Theorem \ref{thm:formality}.

\begin{cor}\label{cor:d=0} Suppose $K \subset M$ is a hyperbolic knot. Then
$$
\widetilde{HW}^d(\del_\infty(M \setminus K)) \cong
\begin{cases} \Z^{\oplus [(I,\del I),(M \setminus K, T)]} \quad &\text{for } d= 0 \\
0 \quad & \text{for } d \neq 0,
\end{cases}
$$
\end{cor}

The next theorem is the first step towards
making an explicit calculation of the map $\widetilde{\mathfrak m}^2$.
To describe the matrix coefficients of $\widetilde{\mathfrak m}^2$, we will relate solutions of \eqref{eq:CRXH}
to the hyperbolic triangles on $N$ truncated by $T$.
We recall that the projection $\pi \circ \gamma^j$ of non-constant Hamiltonian trajectory of $H$ is
a geodesic cord and conversely each geodesic is uniquely lifted to a Hamiltonian chord.

Let $N$ and $T$ be as above and $u$ be any solution to \eqref{eq:CRXH}
with its asymptotic triple $(\gamma^0, \,  \gamma^1, \, \gamma^2)$ of Hamiltonian chords.
Denote by $c^j = \pi \circ \gamma^j$ the associated geodesic cords for the pair $(N,T)$.
%For each solution $u$ of \eqref{eq:CRXH} in {$T^*N$}, we lift it to the universal covering space $T^*\H^3$
%together with the relevant boundary and asymptotic conditions.
%Then the triple $(c^0,c^1,c^2)$ of geodesic cords determines the triple $(\infty^0,\infty^1,\infty^2)$
%of points in $\del \overline \H^3$ modulo the action of deck transformations. (See Lemma \ref{lem:3d-classify}.)
%We denote by $\pi:T^*N \to N$ the canonical projection.
%
%\begin{thm}\label{thm:classify-3dim}
%Denote by
%$$
%\widetilde \Delta = {\widetilde \Delta_{(\infty^0,\infty^1,\infty^2)}}
%$$
%the totally geodesic triangle in $\H^3$ determined by $(\infty^0,\infty^1,\infty^2)$ and
%let $\widetilde u$ as the unique lift of $u$ satisfying the asymptotic conditions
%given by the triple.
%Then the map $f: \Sigma \to \H^3$ defined by
%$f(\zeta) = \pi \circ \widetilde u(\zeta)$ has its image contained in $\widetilde\Delta$.
%\end{thm}
%The statements in Theorem \ref{thm:classify-3dim} can be written in terms of $N = M \setminus K$
%as follows.
\begin{thm}[Theorem \ref{thm:hexagon}]\label{thm:classify-3dim}
Let $u$ be any solution to \eqref{eq:CRXH} with its asymptotic triple $(\gamma^0, \,  \gamma^1, \, \gamma^2)$ of
Hamiltonian chords. Denote by $c^j = \pi \circ \gamma^j$ the associated geodesics on $N$.
There exists a totally geodesic immersed ideal triangle $\Delta$ whose ideal edges contain
each geodesic triples $(c^0, c^1, c^2)$ such that  the map $f: \Sigma \to N$ defined by
$$
f(\zeta) = \pi \circ u(\zeta)
$$
has its image contained in  $\Delta$.	
\end{thm}

The following conjecture is an important one to resolve.

\begin{conj}\label{conj:main} The map $u \mapsto \pi \circ u$ induces a one-one
correspondence between the set of
Floer triangles associated to the triple $(\gamma^0, \gamma^1, \gamma^2)$ and
that of geodesic triangle associated to the triple $(c^0,c^1,c^2)$ with
$c^i = \pi \circ \gamma^i$.
\end{conj}
One outcome of Conjecture \ref{conj:main} would be that calculation of structure constants
of the algebra $HW(\del_\infty(M\setminus K))$
is reduced to a counting problem of geodesic triangles. We hope
to come back to investigate validity of this conjecture elsewhere.

Finally, we study the torus knots and prove the following as a consequence of
Theorem \ref{thm:torusKnot}.

\begin{thm}\label{thm:torus} Let $K \subset S^3$ be any torus knot. Then we have
$\widetilde{HW}^d(\partial_\infty(S^3\setminus K))$ is non-zero for $d=0, 1$ and zero otherwise.
\end{thm}
We refer to Theorem \ref{thm:torusKnot}] for the precise statement of this non-triviality
result.

Combining Theorem \ref{thm:comparison-intro} and Corollary \ref{cor:d=0}, we have derived
\begin{cor} The knot Floer cohomology $\widetilde{HW}^*(\partial_\infty(S^3\setminus K))$
differentiates hyperbolic knots from torus knots.
\end{cor}

\begin{rem}\label{rem:alternative} There is an alternative approach one could take towards
proving Theorem \ref{thm:comparison-intro} following the scheme \cite{ASc}, \cite{APSc}. It goes as follows.
Compare the Morse complex of the kinetic energy functional $E_g: \CP(M \setminus K, T) \to \R$ and
the associated Floer complex $\Omega(\nu^*T; T^*(M\setminus K))$. If $N = M \setminus K$ were a closed
manifold, the main theorem of \cite{APSc} would imply that $HF(\nu^*K, T^*N; H_g)$ is isomorphic
to the Morse homology of $E(g)$ and so to the singular homology of $\CP(N, T)$ \emph{for any metric} $g$.
To perform similar scheme for $\CP(N,T)$ on the open manifold $N = M \setminus K$, one
needs to first identify the set of \emph{admissible} metrics for which relevant analysis entering in the
Morse theory, Floer theory and the comparison of the two theory can be performed. After these preparations, one
expects that the same isomorphism holds for \emph{such a metric}.

In our situation, we consider two different metrics, a hyperbolic metric $h$ and its cylindrical
adjustment $h_0$, on the \emph{open} manifold $M \setminus K$ and directly compare them
in the Floer theory level. Note that the hyperbolic metric $h$ of finite volume
in $M \setminus K$ has injectivity radius zero, and so the global Sobolev inequality fails to hold and
 cannot be used \emph{unless relevant $C^0$-estimates
is preceded.} This makes the scheme used in \cite{APSc} not directly applicable for
such a metric on the open manifold $M \setminus K$. In this regard, a large part of the present paper
is devoted to verifying that the hyperbolic metric and the cylindrical adjustment thereof
are such admissible metrics by overcoming such analytical difficulties that are anticipated
to meet even in this alternative route.
\end{rem}

\subsection{Hyperbolic geometry and the Bochner techniques}

The main ingredients of the proofs of both theorems mentioned above are various applications of
the Bochner-type techniques both in the pointwise version and in the integral version.
A brief outline of how we apply these techniques is now in order.

Theorem \ref{thm:formality} is a consequence of Theorem \ref{thm:nondeg} and
a degree counting argument. Therefore we
will focus on  Theorem \ref{thm:nondeg} and Theorem \ref{thm:classify-3dim}.
The proof of Theorem \ref{thm:nondeg} is an explicit calculation of the second
variation of the energy functional for the paths satisfying the free boundary
condition associated to the horo-torus $T$. Then the explicit formula
we obtain manifestly establishes the positivity and nondegeneracy of the second variation,
which is thanks to the special geometric property of the horo-torus
(See Section \ref{sec:index}.):
it has a constant positive mean curvature relative to the outward unit normal to $T$
pointing to the cusp direction, which we also compute.

For the proof of Theorem \ref{thm:classify-3dim}, we apply some isometry element $g \in \psl$,
and first reduce the classification problem to that of $\H^2 \simeq \{x=0\}$.

For these purposes, we exploit
\begin{enumerate}
\item the \emph{negative} constant curvature property of hyperbolic
metric on $\H^3$,
\item the constant mean-curvature property of the horo-torus $T$ with the correct sign,
\item a usage of (strong) maximum principle based on rather delicate
calculation of the Laplacian of an indicator function and subtle rearrangement of the
terms appearing in the computed Laplacian. (See Appendix \ref{sec:Deltaxz}.)
\end{enumerate}

\subsection{Conventions}
\label{subsec:sasakian}

In the literature on symplectic geometry, Hamiltonian dynamics, contact geometry and
the physics literature, there are various conventions used which are different from
one another one way or the other. In the mathematics literature, there are two
conventions that have been dominantly appeared, which are
summarized in the preface of the book \cite{oh:book1}: one is the convention that has been
consistently used by the third named author and the other is the one that is
called Entov-Polterovich's convention in \cite{oh:book1}.

The major differences between the two conventions lie in the choice of the following three definitions:
\begin{itemize}
\item {\bf Definition of Hamiltonian vector field:} On a symplectic manifold $(P,\omega)$,
the Hamiltonian vector field associated to a function $H$ is given by the formula
$$
\omega(X_H,\cdot) = dH \, (\text{ resp. } \, \omega(X_H,\cdot ) = - dH),
$$
\item{\bf Compatible almost complex structure:} In both conventions, $J$ is compatible to $\omega$ if
the bilinear form $\omega(\cdot, J \cdot)$ is positive definite.
\item{\bf Canonical symplectic form:} On the cotangent bundle $T^*N$, the canonical symplectic form
is given by
$$
\omega_0 = \sum_{i=1}^n dq^i \wedge dp_i, (\text{resp. }\,   \sum_{i=1}^n  dp_i \wedge dq^i),
$$
\end{itemize}
It appears that in the physics literature (e.g., \cite{AENV} and others) the canonical symplectic form
is taken as $dq \wedge dp$ as well as in \cite{klingenberg,EENS} and \cite{oh:jdg}-\cite{oh:book2}.
For the convenience of designating the conventions, let us call the first Convention I and
the second Convention II in the paragraph below. Our current convention is consistent with that of \cite{fooo:book}.

We will utilize various forms of the (strong) maximum principle for the equation
$$
(du - \beta \otimes X_{H_h})_{J_h}^{(0,1)} = 0
$$
with the \emph{negative} sign in front of $\beta \otimes X_{H_h}$.
In the strip coordinate $(\tau,t)$, the equation becomes
\be\label{eq:CRJH}
\frac{\del u}{\del \tau} + J\left(\frac{\del u}{\del t} - X_{H_h}(u)\right) = 0.
\ee
Applicability of the maximum principle is very sensitive to the choice of conventions and the signs
in the relevant equations in general such as \eqref{eq:CRXH} in the present study.
We study the Cauchy-Riemann equation \eqref{eq:CRXH} on $T^*N$ associated to the triple
$$
(\omega_0,J_h, H_h)
$$
where $\omega_0$, $H_h$ and $J_h$ are defined on $T^*N$ following Convention I.
Under these circumstances, it turns out that it is essential to adopt Convention I
to be able to apply the various maximum principles we need for the equation \eqref{eq:CRXH}.
(See the calculations provided in Appendix and Subsection \ref{subsec:z-coordinate},
especially Lemma \ref{lem:positive}, to see how these arise.)

In addition to these, the Floer continuation map is defined over the homotopy of Hamiltonian
\emph{in the increasing direction.} To be able to obtain the necessary energy estimates
in the wrapped setting, we consider the action functional associated to Hamiltonian $H$ on $T^*N$ by
$$
\CA_H(\gamma) = - \int \gamma^*\theta + \int_0^1 H(t, \gamma(t))\, dt
$$
which is the negative of the classical action functional. For the kinetic energy Hamiltonian $H = H_g(x)$, we have
\be\label{eq:CA-E}
\CA_H(\gamma_c) = - E_g(c)
\ee
where $\gamma_c$ is the Hamiltonian chord associated to the geodesic $c$ and $E_g(c)$ is the energy of
$c$ with respect to the metric $g$.

While this paper is written as a sequel to \cite{BKO}, its content is largely
independent of that of \cite{BKO} except that we adopt the same convention as thereof.
Except in Section \ref{sec:knot-algebra} and \ref{sec:comparison}, we directly work with
the given hyperbolic metric for the study of perturbed Cauchy-Riemann equation above
without taking the cylindrical adjustment. This forces us to establish a new form of horizontal
$C^0$ estimates (see Theorem \ref{thm:z-coord}) directly applicable to the hyperbolic metric
without taking a cylindrical adjustment. One important difference of the hyperbolic metric
from that of cylindrical metric is that a geodesic issued even inward from
$\del N^{\text{\rm cpt}}$ for $N^{\text{\rm cpt}} = M \setminus N(K) \subset M\setminus K$
may go out of the domain $N^{\text{\rm cpt}}$, get closer to the knot $K$ and then
come back to the domain $N^{\text{\rm cpt}}$.
In Section \ref{sec:comparison},
we compare the wrapped Floer cohomology associated to $(\nu^*T, H_h)$ with
the Knot Floer cohomology defined for a general knot,
not necessarily a hyperbolic knot, via a cylindrical adjustment
of a smooth metric $g$ on $M$ restricted to $M \setminus K$.

\bigskip

\noindent{\bf Acknowledgement:} Y. Bae thanks Research Institute for Mathematical Sciences, Kyoto
University for its warm hospitality. {The authors also thank C. Viterbo for pointing out
that the isomorphism constructed in Theorem \ref{thm:comparison-intro} may be a consequence of
classical topology of the path space $\CP(N, T)$. The alternative route we outline in Remark \ref{rem:alternative}
is the result of our afterthought.

\section{Definition of Knot Floer algebra in \cite{BKO}}
\label{sec:knot-algebra}

We first provide the construction of Knot Floer algebra introduced in \cite{BKO}
without the details of its construction.

Let $g$ be a smooth Riemannian metric on $M$.
Consider a tubular neighborhood $N(K)$ of $K$. We denote its boundary by
$T = \del(N(K))$ and $L = \nu^*T$, the conormal bundle of the torus $T$.

We define a cylindrical adjustment $g_0$ of the metric $g$ on $M$ by
$$
g_0 = \begin{cases} g \quad & \text{on } M\setminus N'(K) \\
da^2 \oplus g|_{\del N(K)} & \text{on } N(K) \setminus K
\end{cases}
$$
which is suitably interpolated on $N'(K) \setminus N(K)$ and fixed.
Then we denote
\[
W(K) = T^*N(K) \subset T^*(M\setminus K).
\]

We denote by $\mathfrak X(L;H_{g_0})= \mathfrak X(L,L;H_{g_0})$ the set of Hamiltonian chords of $H_{g_0}$
attached to a Lagrangian submanifold $L$ in general.
We have
$$
\mathfrak X(L;H_{g_0}) = \mathfrak X_0(L;H_{g_0}) \coprod \mathfrak X_{ < 0}(L;H_{g_0})
$$
where the subindex of $\mathfrak X$ in the right hand side denotes the action of the Hamiltonian chords of $H_{g_0}$.
We also define
\be\label{eq:spectrum}
\text{\rm Spec}(L;H_{g_0}) = \{\CA_{H_{g_0}}(\gamma) \in \R \mid \gamma \in \frak X(L;H_{g_0})\}
\ee
and call the action spectrum of the pair $(L;H_{g_0})$. By definition of the kinetic energy Hamiltonian

We note that $\mathfrak X_0(L;H_{g_0}) \cong \T^2$ and the component is clean in the sense of Bott
as follows. The following general proposition seems to be interesting of its own.

\begin{prop}\label{prop:Bott-clean} Consider an arbitrary Riemannian manifold $(N,g)$
of any dimension $n$. Let $T \subset N$ be any compact submanifold of dimension $0 \leq k \leq n$. Then
\begin{enumerate}
\item The set $\widehat T$ of constant Hamiltonian chords of $H_{g}$
attached to $\nu^*T$ consist of constant paths valued in $o_{\nu^*T}$. Hence
$\widehat T$ is in one-one correspondence with $\T^2$.
\item
The set $\widehat T$ is normally nondegenerate in the path space
$$
\Omega_{[0,1]}(\nu^*T;T^*N) = \{ \gamma: [0,1] \to T^*N \mid \gamma(0), \, \gamma(1) \in \nu^*T\}
$$
and is diffeomorphic to $T \cong \T^2$.
\end{enumerate}
\end{prop}
\begin{proof}
The statement (1) is a direct
consequence of  the boundary condition $x(0), \, x(1) \in \nu^*T$, since
any constant solution of $\dot x = X_{H_{g}}(x)$ has zero momentum, i.e., $p = 0$.

The remaining proof will be occupied by the proof of Statement (2).
Let $\gamma_q: [0,1] \to T^*N$ with $q \in T$ be a constant Hamiltonian cord
valued at $(q,0) \in \nu^*T$.

We decompose vector field $\xi$ along $\gamma_q$
into $\xi = \xi^\parallel + \xi^\perp$
under the decomposition $T_{(q,0)}(T^*N) \cong T_q N \oplus T_q^*N$.
For each pair of vector fields $\xi_1, \, \xi_2$ along $\gamma_q$ satisfying
$$
\xi_i(0), \, \xi_i(1) \in T_{(q,0)}(\nu^*T)
$$
a straightforward calculation by integration by parts give rise to
the formula for the Hessian of $\CA_{H_g}$ at $\gamma_q$
\beastar
d^2\CA_{H_g}(\gamma_q)(\xi_1,\xi_2) & = & \int_0^1 \omega \left(\frac{D \xi_2}{dt}(t), \xi_1(t)\right)
+ g^\flat(\xi_1^\perp(t),\xi_2^\perp(t)))\, dt \\
& = & \omega(\xi_2(1), \xi_1(1)) - \omega(\xi_2(0),\xi_1(0)) \\
&{}&
- \int_0^1 \omega\left(\xi_2(t),\frac{D \xi_1}{dt}(t)\right)+ g^\flat(\xi_1^\perp(t),\xi_2^\perp(t)))\, dt\\
& = & - \int_0^1 \omega\left(\xi_2(t),\frac{D \xi_1}{dt}(t)\right)+ g^\flat(\xi_1(t)^\perp,\xi_2^\perp(t)))\, dt
\eeastar
where the last equality follows since $\nu^*T$ is Lagrangian.
(We refer readers to the proof of \cite[Proposition 18.2.8]{oh:book2} for the calculation of the
first term in the first line, the Hessian of the functional $\gamma \to \int_ \gamma^*\theta$.)
Using the decomposition $\xi = \xi^\parallel + \xi^\perp$, we have
\beastar
\omega\left(\xi_2,\frac{D \xi_1}{dt}(t)\right) &  = &
\left\langle \xi_2^\parallel(t), \frac{D \xi_1^\perp}{dt}(t)\right\rangle -
\left\langle \xi_2^\perp(t), \frac{D \xi_1^\parallel}{dt}(t)\right\rangle\\
& = & g^\flat\left((\xi_2^\parallel)^\flat(t), \frac{D \xi_1^\perp}{dt}(t)\right)
-  g^\flat\left((\xi_2^\perp(t), \frac{D (\xi_1^\parallel)^\flat}{dt}(t)\right)\\
\eeastar
where the pairing in the first line is the canonical pairing between $TM$ and $T^*M$.
Therefore substituting this into above, we obtain
\beastar
d^2\CA_{H_g}(\gamma_q)(\xi_1,\xi_2) & = &
\int_0 ^1 g^\flat\left((\xi_2^\parallel)^\flat(t), \frac{D \xi_1^\perp}{dt}(t)\right)\, dt\\
&{}& - \int_0^1 g^\flat\left((\xi_2^\perp(t), \frac{D (\xi_1^\parallel)^\flat}{dt}(t) - \xi_1^\perp(t)\right)
\, dt
\eeastar

Therefore a kernel element of $d^2\CA_{H_g}(\gamma_q)$ is given by
the vector field $\xi \in T_{\gamma_q}\Omega(\nu^*T, T^*N)$ satisfying
\be\label{eq:kernel}
\begin{cases}
\frac{D (\xi^\parallel)^\flat}{dt}(t) - \xi^\perp(t) = 0, \quad \frac{D \xi^\perp}{dt}(t) = 0 \\
\xi^\parallel(0), \, \xi^\parallel(1) \in T_qT, \quad  \xi^\perp(0), \, \xi^\perp(1) \in \nu^*_qT.
\end{cases}
\ee
We solve the equation and obtain general solution
$$
\xi^\perp (t) \equiv \alpha, \quad \xi^\parallel(t) = v + (\alpha)^\sharp t
$$
for some vectors $\alpha \in \nu^*_qT, \, v \in T_qM$ that satisfy
\be\label{eq:xiparallelat1}
\alpha^\sharp = \xi^\parallel(0), \, v + \alpha^\sharp = \xi^\parallel(1)\in T_qT.
\ee
Since $\alpha^\sharp \in \nu_qT$ already, the first equation implies $\alpha^\sharp \in T_qT \cap \nu_qT=\{0\}$ and hence $\alpha = 0$.
On the other hand, the second equation then implies $v$ can be chosen arbitrarily from
$T_qT$. This proves
\be\label{eq:kernel-set}
\ker d^2\CA_{H_g}(\gamma_q) = \left\{\xi: [0,1] \to T_q M \oplus T_q^*M
\mid \xi(t)\equiv (v, 0), \, v \in T_qT \right\}
\ee
and hence $\ker d^2\CA_{H_g}(\gamma_q) \cong T_{\gamma_q}\widehat T \cong T_qT \cong \R^k$.

Taking the inner product with a test function $\eta$ of the equation \ref{eq:xiparallelat1} and
performing integration by parts, we obtain
the $L^2$-adjoint equation of \eqref{eq:kernel}, which is
\be\label{eq:cokernel}
\begin{cases}
\frac{D (\eta^\perp)}{dt}(t) + (\eta^\parallel)^\flat(t) = 0, \quad \frac{D \eta^\parallel}{dt}(t) = 0 \\
\eta^\parallel(0), \, \eta^\parallel(1) \in \nu_qT, \quad  \eta^\perp(0), \, \eta^\perp(1) \in (T_qT)^\flat.
\end{cases}
\ee
Solving this, we obtain
$$
\eta^\parallel(t) = w, \quad \eta^\perp (t) = \beta - t w^\flat
$$
for some $\beta \in (T_qT)^\flat$ and $w \in \nu_q T$ satisfying
$$
\beta = \eta^\perp (0), \,  \beta - w^\flat = \eta^\perp(1) \in (T_q T)^\flat.
$$
Since  $\beta^\sharp \in T_qT$ already, we have $w = 0$ and so $\eta^\parallel \equiv 0$
and $\eta^\perp(t) \equiv \beta$. This proves that the $L^2$-cokernel becomes
$$
\coker d^2\CA_g(\gamma_q)= \left\{ \eta: [0,1] \to T_qM \oplus T_q^*M \mid \eta(t) \equiv (0,\beta), \, \beta \in (T_qT)^\flat \right\}
$$
From this explicit expression of $\coker d^2\CA_g(\gamma_q)$, we obtain
$$
\coker d^2\CA_g(\gamma_q) \cap (T_{\gamma_q} \widehat T)^\circ = \{0\}
$$
where  $(T_{\gamma_q} \widehat T)^\circ$ is the annihilator of $T_{\gamma_q} \widehat T$
which is the dual of $L^2$-orthogonal complement of $T_{\gamma_q} \widehat T$ in
$
L^2(\gamma_q^*(T(T^*M)), (\del \gamma_q)^*T(\nu^*T)))
$
which is the set of $L^2$-sections of $\gamma_q^*(T(T^*M))$ with boundary values at $(\del \gamma_q)^*T(\nu^*T))$.

Combining the two, we have finished the proof of (2).
\end{proof}

We apply this proposition to the pair of Riemannian manifold $(M\setminus K, g_0)$ and
the torus $T \subset M \setminus K$ in our current context of hyperbolic knot.
Take
\be\label{eq:CW-morsebott}
CW(\nu^*T,\nu^*T; T^*(M\setminus K); H_{g_0}) := C^*(T) \oplus \Z\langle \mathfrak X_{< 0}(L;H_{g_0})\rangle
\ee
where $C^*(T)$ is a cochain complex of $T$, e.g., $C^*(T) = \Omega^*(T)$ the de Rham complex
and associate an $A_\infty$ algebra following the construction from \cite{fooo:book}.
It was shown in \cite{BKO} that $HW_g(T,M\setminus K)$ does not depend on the choice of
smooth metric $g$ on $M$ and of the tubular neighborhood $N(K)$ but depends only on the
isotopy type of the knot $K$.

\begin{defn}[Knot Floer algebra \cite{BKO}]
We denote by
$$
HW(\del_\infty(M\setminus K)) = HW_g(T,M \setminus K)
$$
the resulting common (isomorphism class of the) group
and call it the knot Floer algebra of $K$ in $M$.
\end{defn}

To facilitate our calculation of this algebra and its comparison with Knot Floer algebra $HW(\del_\infty(M\setminus K)$,
we now take the Morse complex model for $C^*(T)$, and realize the model \eqref{eq:CW-morsebott} as
a nondegenerate wrapped Floer complex of a perturbed $\nu^*T$ as follows.

Take a compactly supported smooth function $k: M\setminus K \to \R$ such that
$\nu^*T \pitchfork \Image dk $.
We then consider the translated conormal $\nu^*_k T \to T$ whose fiber is given by
\be\label{eq:nufT}
(\nu^*_{k} T)_q: = \{\alpha + dk(q) \in T_q^*N \mid \alpha  \in \nu_q^*T\}.
\ee
Then it is easy to check that we have one-one correspondence between
$\nu^*T \cap \nu_k^*T $ and $\nu^*T \cap \Image dk$ and
the intersection is transversal by the hypothesis $\nu^*T \pitchfork \text{\rm Image } dk$.

We then take a radially cut-off function $\rho: T^*N \to \R$ satisfying $\rho(q,p) = \rho_q(|p|)$ where
$\rho_q: \R_+ \to [0,1]$ is a monotonically decreasing function satisfying
$$
\rho_q(r) = \begin{cases} 0 \quad & \text{for } r \geq 3\|dk\|_{C^0} \\
1 \quad & \text{for } r \leq 2\|dk\|_{C^0}
\end{cases}
$$
and consider the function
$
f: \nu^* T \to \R
$
defined by
$$
f(\alpha) = \rho(\alpha) k (\pi(\alpha)).
$$
Then we take a Darboux-Weinstein chart $\Phi: V \subset T^*(\nu^*T) \to U \subset T^*N$ of $\nu^*T$ and then consider the exact Lagrangian submanifold
\be\label{eq:nu-krho-T}
\nu^*_{k,\rho} T : =  \Phi(\Image df) \subset T^*N.
\ee
In other words, $\nu^*_{k,\rho} T = \Image \iota_{k,\rho}$ for the Lagrangian embedding
$
\iota_{k,\rho}=\iota_{k,\rho}^\Phi: \nu^*T  \to U \subset T^*N
$
defined by
$$
\iota_{k,\rho}^\Phi(\alpha) = \Phi(d(\rho k\circ \pi)|_\alpha).
$$
As usual, we require $\Phi$ to satisfy
\be\label{eq:dPhi}
\Phi|_{o_{T^*(\nu^*T)}} = id|_{\nu^*T}, \quad d\Phi|_{o_{T(T^*(\nu^*T))}} = id|_{T(\nu^*T)}
\ee
under the canonical identifications of $o_{T^*(\nu^*T)}\cong \nu^*T$ and
$$
T(o_{T^*(\nu^*T)}) \cong T(\nu^*T).
$$

It is easy to check that $f$
satisfies $\iota_{k,\rho}^*\theta = df$ and so $\nu^*_{k,\rho} T$ is an exact Lagrangian submanifold.
We also note that
$$
\nu^*_{k,\rho} T = \begin{cases} \nu^*T \quad & \text{\rm for } |p| \geq 3\|dk\|_{C^0}\\
 \nu^*_{k} T \quad & \text{\rm for } |p| \leq 2\|dk\|_{C^0}
 \end{cases}
$$
and $f(\beta) = 0$ for $|\beta| \geq 3\|dk\|_{C^0}$ and $f(\beta) = k(\beta)$ for
$|\beta| \leq 2\|dk\|_{C^0}$.

Next we denote by $\mathfrak G_{g_0}(T)$ the energy of the shortest geodesic cord of $T$
relative to the metric $g_0$. Then the following lemma is an immediate
consequence of the implicit function theorem.

\begin{lem}\label{lem:perturbed-cord} Let $k$ be the function given above such that
$\nu^*T \pitchfork \text{Image } dk$. Then
there exists some $0 < \epsilon_0 < \frac{\mathfrak G_{g_0}(T)}2$ such that
$$
\frak X_{<- \epsilon_0}(\nu^*_{k,\rho} T, \nu^*T) = \frak X_{< - \epsilon_0}(\nu^* T, \nu^*T)
$$
and
$$
\frak X_{\geq -\epsilon_0}(\nu^*_{k,\rho} T, \nu^*T) \cong \text{\rm Image } dk \cap \nu^*T
$$
provided $\|k\|_{C^2} \geq - \epsilon_0$. Here $\cong$ means one-one correspondence.
\end{lem}
\begin{proof}  Both identities then are immediate consequences of Sard-Smale implicit function theorem via
the definition of the cut-off function $\rho$ and $C^1$-smallness of $df$, and
the requirement \eqref{eq:dPhi}.
\end{proof}

By the generic transversality proof under the perturbation of Lagrangian
boundary from \cite{oh:fredholm}, we can choose $k$ so that
$\frak X(\nu^*_{k,\rho} T, \nu^*T)$ is nondegenerate.
We have one-one correspondence
$$
\frak X_{< - \epsilon_0}(\nu^* T, \nu^*T) \cong \text{\rm Crit } k
$$
and hence
$$
\Z \langle  \frak X_{< - \epsilon_0}(\nu^* T, \nu^*T) \rangle \cong
\Z \langle \text{\rm Crit } k \rangle.
$$
We denote $L = \nu_{k,\rho}^*T$.
Then we define a Floer chain complex
$$
CW_{g}^d(T, M\setminus K): = CW^d(L,L; T^*(M\setminus K); H_{g_0}).
$$
Here the grading $d$ is given by the grading of the Hamiltonian chords $|x|$.
Then the construction in \cite{BKO} associates an $A_\infty$ algebra to
$CW_{g_0}(T, M\setminus K)$. We denote the associated cohomology by
\be\label{eq:HWHh}
HW_g(T,M \setminus K) =  HW(L,L; T^*(M\setminus K); H_{g_0}).
\ee

It follows from the Bott-Morse property of $(\nu^*T,H_{g_0})$ that
the complex $CW_{g}^d(T, M\setminus K)$ has such a decomposition
\be\label{eq:decompose}
CW_g(T,M\setminus K): =
\Z\langle \mathfrak X_{\geq - \epsilon_0}(H_{g_0};L,L)\rangle\oplus
\Z\langle \mathfrak X_{< - \mathfrak G_{g_0}(T)/2}(H_{g_0};L,L)\rangle.
\ee

Moreover $\Z\langle\mathfrak X_{\geq - \epsilon_0}(H_{g_0};L,L)\rangle$ is a subcomplex
of $CW_g(T,M\setminus K)$ which is isomorphic to
the Morse complex $(C^*(k|_T),d)$ of the Morse function $k|_T : T \to \R$.

\begin{prop}
The operator $\mathfrak m^1$ has  the matrix form
$$
\mathfrak m^1 = \left(\begin{matrix} \pm d & 0 \\
* & \mathfrak m^1_{< 0}
\end{matrix}\right)
$$
with respect to the above decomposition.
\end{prop}
\begin{proof}
Suppose $\gamma_- \in \mathfrak X_{\geq - \epsilon_0}(H_{g_0};L,L)$
and $\gamma_+ \in \mathfrak X_{\geq + \epsilon_0}(H_{g_0};L,L)$. Then we have
$$
\CA_{H_{g_0}}(\gamma_-) \geq - \epsilon_0, \quad \CA_{H_{g_0}}(\gamma_+) \leq - \mathfrak G_{g_0}(T)/2
$$
and hence
\be\label{eq:<0}
\CA_{H_{g_0}}(\gamma_+) - \CA_{H_{g_0}}(\gamma_-) \leq - \mathfrak G_{g_0}(T)/2 + \epsilon_0 < 0.
\ee
On the other hand, \emph{if there exists a solution $u$ for \eqref{eq:CRJH}
satisfying $u(\pm \infty)) =\gamma_\pm$}, then
$$
0 \leq \int \left|\frac{\del u}{\del \tau}\right|^2_{J_{g_0}} = \CA_{H_{g_0}}(u(+\infty)) - \CA_{H_{g_0}}(u(- \infty))
= \CA_{H_{g_0}}(\gamma_+) - \CA_{H_{g_0}}(\gamma_-)
$$
which contradicts to \eqref{eq:<0}. This finishes the proof.
\end{proof}

\begin{defn}[Reduced Knot Floer complex] Denote by $(\widetilde{CW}_g(T,M\setminus K), \widetilde{\mathfrak m}^1)$ the quotient complex
$$
\left(CW_g(T,M\setminus K)/  C^*(T), [ \mathfrak m^1_{< 0}]\right).
$$
We call this complex by the reduced Knot Floer complex.
\end{defn}

We also note $ \mathfrak m^1_{< 0} \circ  \mathfrak m^1_{< 0} = 0$ and so
$$
\left(\Z\left\langle\mathfrak X_{< - \mathfrak G_{g_0}(T)/2}(H_{g_0};L,L)\right\rangle,
\mathfrak m^1_{< 0}\right)
$$
is naturally isomorphic to the reduced complex $(\widetilde{CW}_g(T,M\setminus K), \widetilde{\mathfrak m}^1)$.

Therefore the reduced complex is nothing but the complex generated by non-constant
Hamiltonian chords.

We emphasize that in the context of hyperbolic knots which is the case of our main interest in the
present paper the given hyperbolic metric $h$ on $M\setminus K$ is neither cylindrical at
infinity nor smoothly extends to the whole space $M$. Because of this,
we cannot directly use the hyperbolic metric $h$ defined on $M \setminus K$
for the calculation of $HW(\del_\infty(M\setminus K))$. The rest of the paper is occupied by
the construction of this wrapped Floer complex associated to the kinetic energy
Hamiltonian $H_h$ of the hyperbolic metric on $M \setminus K$ whose injectivity radius is zero.

\section{Preliminary on hyperbolic 3-manifold of finite volume} \label{sec:prelim}
In this section, we briefly review several well-known facts. For a reference, Martelli's book \cite{Mart} is readable and enough to know some basics about hyperbolic geometry and 3-manifold theory used in this article.

Let $N$ be a complete hyperbolic 3-manifold.
The universal cover  $\widetilde N$ is identified with the hyperbolic 3-space
\[
\H^3=\{(x,y,z)\in\R^3 \mid x,y\in \R, z\in\R^{+}\}
\]
and $N$ is isometric to $\H^3 /\, \Gam$ with a discrete group $\Gam \cong \pi_1(N)$. If $N$ is orientable, $\Gam$ consists of orientation preserving isometries. Hence, from now on we identify the  hyperbolic space $\H^3$ and the group of orientation preserving isometries $\Is(\H^3)$ with a upper half space and  $\psl$ respectively. The \emph{ideal boundary} $\bdH$ of $\H^3$ is identified with $\Ch := \C \cup \{\infty\}$ and $\psl$ acts on $\Ch$ and $\H^3$ as M\"obius transformations, and the Poincar\'e extensions, respectively.

Let $N$ be a knot complement of a orientable closed 3-manifold $M$, i.e., $N=M\setminus K$. Then $N$ is homeomorphic to the interior of the knot exterior, a compact 3-manifold denoted by $\overline N$,
that has a torus boundary of $N$ and the complete hyperbolic structure should be of finite volume by the torus boundary condition.

\subsection{$\varepsilon$-thick-thin decomposition by the Busemann function}
\smallskip

The first important fact in this situation would be the uniqueness of the hyperbolic metric $h$  by Mostow-Prasad rigidity, i.e., $h$ is unique up to isometry. Therefore any hyperbolic metric invariant can be regarded as a topological invariant, as we have a canonical Riemannian metric.

Secondly, we can nicely separate compact part and the non-compact end part of $N$ as follow.
	\begin{prop}
		There is a constant $\varepsilon_0>0$ such that the $\varepsilon_0$-thin part,
		\be
		N_{(0,\varepsilon_0]} := \{x\in N \mid \text{\rm inj}_x(N)\leq \varepsilon_0\} ,
		\ee
		is homeomorphic to the end of $N$, which is $\T^2 \times [0,+\infty)$.
		Thus the $\varepsilon$-thick part $N_\varepsilon$ is compact and the interior is homeomorphic to $N$ itself once we take $\varepsilon \leq \varepsilon_0$,
		\be
		N_\varepsilon := N_{[\varepsilon,+\infty)} \approx N.
		\ee	
	\end{prop}
\begin{proof}
	By the thick-thin decomposition using the Margulis constant $\varepsilon_0$, we have two kinds of thin parts,
	 thin-tubes $S^1\times D^2$ and truncated cusp $\T^2 \times [0,1)$.
When we take $\varepsilon_0$ less than smallest injective radius among thin-tubes, then the only possible $\varepsilon_0$-then part is the truncated cusp of the boundary.
\end{proof}

By the structural property of the thick-thin decomposition \cite[ Section 4]{Mart}, we know that $N_{(0,\varepsilon]}$ is a \emph{truncated cusp} and thus $N_{(0,\varepsilon]} \cap N_\varepsilon = \partial N_{(0,\varepsilon]} = \partial N_\varepsilon$ is a Euclidean torus.
We describe $N_\varepsilon$ by using the Busemann function instead of injective radius.

\begin{defn}
	Let  $\delta:[0,+\infty)\to N$ be a geodesic ray satisfying \[d_h(\delta(t),\delta(t'))=|t-t'|.\]
	The {\em Busemann function} $b_\delta:W\to\R$ is defined by
	\begin{align*}
	b_\delta(q):=\lim_{t\to\infty}(d(q,\delta(t))-t).
	\end{align*}
\end{defn}

	Without loss of generality, we can take a lifting $\widetilde \delta$ in $\H^3$ of $\delta$,
	\begin{align*}
	\widetilde \delta (t) := (0,0,e^t) \in \H^3,
	\end{align*}
such that the lifted Busemann function $b_{\widetilde \delta} $  for $\H^3$ is given by
	\begin{align*}
	b_{\widetilde \delta}((x,y,z))=\lim_{t\to\infty}\big(d((x,y,z),\widetilde\delta(t))-t \big)
	=-\log z,
	\end{align*}
	and the level set of $b_{\widetilde\delta}$ is a horosphere  centered at $\{z=+\infty\}$,
	\be\label{eq:bdelta}
	b_{\widetilde \delta}^{-1}(t) = \{(x,y,z) \in \H^3 \mid z= e^{-t}\}.
	\ee
These are direct consequences of a hyperbolic distance formula in \cite[ Section  III.4]{Fen}.
We remark that there are many  other lifts of $\delta$ which may tend to the other  ideal points of $\partial \Hb$.

\begin{prop}\label{prop:thinthickdecomp}
	We have a decomposition $N=N_\varepsilon \cup N_{(0,\varepsilon]}$  by a Busemann function  $b_\delta:N\to\R$ where
	\begin{align*}
	N_\varepsilon  &=  b_{\delta}^{-1}([t_0, \infty)) \\
	N_{(0,\varepsilon]} & = b_{\delta}^{-1}((-\infty, t_0]) \approx b_{\delta}^{-1}(t_0) \times [0,\infty).
	\end{align*}
We call $N_\varepsilon$ a $\varepsilon$-thick compact part and  $N_{(0,\varepsilon]}$ a $\varepsilon$-thin cusp part, respectively.
\end{prop}
\begin{proof}
	It is succinct to use Epstein-Penner decomposition \cite{EpPe}. Let $N$ be obtained from a finite union of  convex hyperbolic ideal polyhedra with face gluings in $\Is(\H^3)$.
Hence we have a fundamental domain $\pD$ whose boundary is made up of  totally geodesic convex ideal polygons such that
\[ \H^3 = \bigcup\limits_{g\in\Gam}  g \cdot \pD,\]
Without loss of generality, we can assume $ \widetilde\delta(t)  \in \pD$ and $\widetilde\delta{(\infty)} =\infty \in \widetilde V_{\pD}\subset \Ch$,
where
\begin{align*}
\widetilde V_{\pD}=\bigcup\limits_{g\in\Gam}  g \cdot V_{\pD}
\end{align*}
and $V_{\pD}$ is the finite set of ideal vertices of $\pD$.
We now take a sufficiently large negative number $t_0$ such that the  horospheres centered at $\widetilde V_{\pD}$ are mutually disjoint, i.e.,
\begin{align*}
g \cdot b_{\widetilde \delta}^{-1}(t_0) \cap g' \cdot b_{\widetilde \delta}^{-1}(t_0) = \varnothing,
\end{align*}
for any distinct pair of $g,\, g'\in \Gamma$.
Then the truncated fundamental domain,
\be
\pD  \setminus   \bigcup\limits_{i=1,\dots,n} g_i\cdot b_{\widetilde \delta}^{-1}(t_0),
\ee
produces $N_\varepsilon$ by face gluing isometries originally used to make  $N$. The injective radius $\varepsilon$ is taken at any point in $b_{\delta}^{-1}(t_0)$.
\end{proof}

\subsection{Infinite tame geodesics and geodesic cords of horo-torus}
\smallskip

Denote by $T = \del N_\varepsilon$.
We call the boundary torus of $N_\varepsilon$ as a \emph{horo-torus} $T$, which is given by a level set  $b_\delta^{-1}(t_0)$ of the Busemann function.

\begin{defn}\label{defn:cord} We define
$$
{\rm Cord}(T) = \{ c:[0,1] \to N \mid \nabla_t \dot c = 0,\,c(i) \in T, \,\, \dot c(i) \perp T, \,\text{for } i=0, 1\}.
$$
We call an element of  $\text{Cord}(T)$ a \emph{geodesic cord} of $T$.
\end{defn}

Next we consider an \emph{infinite tame geodesic} $\delta: (-\infty,\infty) \to N$.

\begin{defn}\label{defn:tame-geodesic}
	An infinite geodesic $\delta : \R \to N$ is called \emph{tame} if
	if there is a continuous map $\alpha : [0,1] \to $ with $\alpha(\{0,1\})\in \overline N \setminus N$ such that $\delta$ and $\alpha |_{(0,1)}$ has the same image in $\overline N$. We denote by
$\mathscr{G}=\mathscr{G}_N$ the set of images of all infinite tame geodesics in $N$.
\end{defn}

Now we would like to emphasize the following proposition which gives a crucial intuition for our purpose.

\begin{prop}\label{prop:cordG}
	Let $T$ be a horo-torus in hyperbolic knot complement $N$. Then there is a one-one correspondence
 between $\text{\rm Cord}(T)$ and $\mathscr{G}$.
\end{prop}
\begin{proof} Denote by $c$ and $\delta$ a geodesic cord of $T$ and an infinite tame geodesic respectively.

Let $\delta$ be an infinite tame geodesic on $N$.
We take a lift $\widetilde{\delta}$ to the universal cover $\H^3$ and denote its $\alpha$-limit
point by $q \in \del \overline \H^3$, i.e., $\lim_{t \to -\infty}\widetilde \delta(t) = q$.
Then we take any lift $\widetilde{T}$. We can
pick an element $g \in \text{PSL}(2,\C)$ that maps the center of the horo-sphere $\widetilde T$ to $q$.
By applying the action by another $\text{PSL}(2,\C)$ element, we may assume that $g \cdot \widetilde T$
has its center at $q = \infty$ so that $g \cdot \widetilde{T}= \{z=z_0\}$ for some $z_0 \in \R_+$.
Since any horo-sphere intersects perpendicularly to the geodesic on $\H^3$ issued at its center,
$g \cdot \widetilde \delta$ is perpendicular to $\{z=z_0\}$ and so a vertical infinite geodesic in $\H^3$.
This implies that $g \cdot\widetilde \delta$ is perpendicular to $\{z = z_0\}$ for all $z_0 \in \R_+$.
In particular $\delta$ is perpendicular to all the horo-spheres with centers at either of $q^\delta_\pm$,
its asymptotic limits at $\pm \infty$ respectively. Therefore the restriction of
$\delta \in \mathscr{G}$ on $[0,1]$ is perpendicular to $T$ at its end points at $t =0, \, 1$.
In particular $c: = \delta_{[0,1]}$ is contained in $\text{\rm Cord}(T)$.
Denote the corresponding cord by $c_\delta$.

Conversely, let $c \in \text{Cord}(T)$. Take lifts $\widetilde T_\pm$ of $T$ and $\widetilde{c}$ so that
$$
\widetilde c(0) \in \widetilde T_-, \, \widetilde c(1) \in \widetilde T_+.
$$
We denote by $q_\pm \in \del \overline \H^3$ the centers of the horo-spheres $\widetilde T_\pm$
respectively. Again choose $g \in \text{PSL}(2,\C)$ as above so that $\widetilde T_-$
has its center at $\infty$ so that $g \cdot \widetilde T_- = \{z=z_0\}$ for some $z_0$.
Since the geodesic $g\cdot \widetilde c$ is perpendicular to $\{z=z_0\}$ at $t=0$, it
must be a vertical geodesic segment. Moreover $g \cdot \widetilde T_+$ is perpendicular to
$ g \cdot \widetilde c$ at $t=1$
which follows since $\dot c(1) \perp T$ by definition of $\text{Cord}(T)$.
Therefore $g \cdot\widetilde c$ is a part of infinite
vertical geodesic $\delta$ issued at an ideal point $q_+ \in \bdH$.
This proves that $\widetilde c$ is a part of the infinite hyperbolic geodesic
$g^{-1}\cdot \widetilde \delta$. Since the covering projection is local isometry,
its projection is an infinite geodesic which is an extension of $c$ and tame to $T$.
 Denote by $\delta_c$ the corresponding geodesic.

By construction, it follows $c = c_{\delta_c}$ and $\delta = \delta_{c_\delta}$ and so
finishes the proof of one-one correspondence.
\end{proof}
We remark that some interior points in $c$ or $\gamma$ may intersect $T$ non-perpendicularly and some geodesic cords may not be contained in a fixed $N_i$ in an exhaustion sequence $\{N_i\}$. In fact, only a finite number of geodesic cords can be contained in a fixed $N_i$.

	Finally, we recall the following result which follows from the well-known argument
in hyperbolic geometry. 	
\begin{prop}
	For each non-trivial element $e \in \pi_1(N, T)$,
	there exists a unique geodesic cord $c_e$ attached to $T$.
\end{prop}
\begin{proof} We first prove uniqueness result by contradiction.
	Suppose that there are two different $c$ and $c'$ in $\text{\rm Cord}(N,T)$ in the same homotopy class of $\pi_1(N,T)$.
  We consider their liftings denoted by $\widetilde c$ and $\widetilde c'$ respectively. We can
   choose the same lifted horo-tori (i.e., horo-spheres) $\widetilde T_0$ containing both $\widetilde c(0)$ and $\widetilde c'(0)$ since $c(0), \, c'(0)$ are in $T$ and $T$ is connected.
    Furthermore the lifted horo-tori of $\widetilde T_1$ and $\widetilde T'_1$ corresponding to
    the horo-tori $T$ respectively containing $\widetilde c(1)$ and $\widetilde c'(1)$ also coincide since $c, \, c'$ are assumed to be in the same homotopy class.
    Moreover the centers of the horo-spheres $\widetilde T_0$ and $\widetilde T_1$ are distinct because $[c], \, [c']$ are non-trivial in $\pi_1(N,T)$. But both $\widetilde c$ and $\widetilde c'$ connect the same pair of
    tori $\widetilde T_0$ and $\widetilde T_1$ and intersect them perpendicularly thanks to the Neumann
    boundary condition put on the geodesic cords. Recall that the geodesic perpendicularly meeting
    two horo-spheres with distinct centers is unique in $\H^3$. Therefore their projections $c, \, c'$
    also coincide, which contradicts to the assumption that
    $c$, $c'$ are distinct. This completes the proof of uniqueness.
    For the existence, this follows from the standard result in the calculus of variations
    concerning the existence of an energy minimizing geodesic path
    in each given non-trivial homotopy class $\pi_1(M\setminus K,T)$,
    noting that $T$ is a compact submanifold
    and so a minimizing sequence can be put into a compact region of $M \setminus K$.
   (Alternatively one could also employ hyperbolic geometry argument similarly as in the above
    uniqueness proof.) This finishes the proof.
\end{proof}
We remark that for a cord $e$ is trivial in $\pi_1(N,T)$, there are $T$-many constant geodesic cords in the same trivial homotopy class in $\pi_1(N,T)$.

\section{Cotangent bundle of hyperbolic knot complement}

In this section, we summarize some basic facts on the Riemannian geometry of
the cotangent bundle of hyperbolic knot complement.

\subsection{The Sasakian almost complex structure on the cotangent bundle}\label{sec:metric_cotangent_bundle}
For the later purpose we discuss about an induced metric on the cotangent bundle $\widetilde \pi:T^*\H^3\to\H^3$ of
$$
\H^3=\{(x,y,z)\in\R^3 \mid z > 0\}
$$
with a complete hyperbolic metric $h=h_{\H^3}$ given by
$$
(h_{ij})_{1\leq i,j\leq 3}=\big(\frac{1}{z^2}\delta_{ij}\big)_{1\leq i,j\leq 3}.
$$
Let us start with Levi-Civita connection $\nabla=\nabla^{h}$ and an induced (co-)frame fields $H_i,V_i$ (and $H^i,V^i$) on $T^*\H^3$ which are given as follows:
\begin{align*}
H_i&= \partial_{q^i}+p_a\Gamma^a_{ij}\partial_{p_j},& V_i&=\partial_{p_i};\\
H^i&= dq^i,& V^i&=dp_i-p_a\Gamma^a_{ij}dq^j.
\end{align*}
Here $\Gamma^a_{ij}$ are Christoffel symbols for the connection $\nabla$ and we used the Einstein summation convention.

A direct calculation using the definition of Christoffel symbols for the hyperbolic metric
$$
h = \frac{dx^2 + dy^2 + dz^2}{z^2}
$$
on $\H^3$ for the standard coordinate $(x,y,z)$ with $ z \geq 0$ gives rise to
\begin{lem}\label{lem:Gammaijk}
\bea\label{eq:R-chris}
\Gamma_{11}^3& = &1/z, \quad  \Gamma_{13}^1=\Gamma_{31}^1=-1/z,\nonumber\\
\Gamma_{22}^3 &= &1/z, \quad  \Gamma_{23}^2=\Gamma_{32}^2=-1/z, \nonumber\\
\Gamma_{33}^3 &= & -1/z, \quad \text{ and all other symbols are zero.}
\eea
\end{lem}
Using this calculation, we explicitly express
\begin{align*}
H_1&=\partial_x+\frac{p_z}{z}\partial_{p_x}-\frac{p_x}{z}\partial_{p_z};\\
H_2&=\partial_y+\frac{p_z}{z}\partial_{p_y}-\frac{p_y}{z}\partial_{p_z};\\
H_3&=\partial_z-\frac{p_x}{z}\partial_{p_x}-\frac{p_y}{z}\partial_{p_y}-\frac{p_z}{z}\partial_{p_z}\\
\end{align*}
and
\begin{align*}
V^1&= dp_x-\frac{p_z}{z}dx+\frac{p_x}{z}dz;\\
V^2&= dp_y-\frac{p_z}{z}dy+\frac{p_y}{z}dz;\\
V^3&= dp_z+\frac{p_x}{z}dx+\frac{p_y}{z}dy+\frac{p_z}{z}dz.\\
\end{align*}

An induced Riemannian metric $\widetilde h$ on $T^*\H^3$ with respect to the (co-)frame fields is given by
$$
h_{ij}dq^idq^j+h^{ij}\delta p_i\delta p_j
$$
where $(h^{ij})_{i,j}$ is the inverse matrix of $(h_{ij})_{i,j}$ and $\delta p_i=V^i$.
In a matrix form we have
\begin{displaymath}
\left(\begin{array}{c|c}
h_{ij} & 0 \\
\hline
0& h^{ij}
\end{array}\right).
\end{displaymath}

Now we equip $(T^*\H^3, \widetilde h)$ with the canonical symplectic 2-form
$$
\omega=\sum_{i=1}^3 dq^i\wedge dp_i= \sum_{i=1}^3 H^i\wedge  V^i
$$
and an almost complex structure on $T^*\H^3$ the so called Sasakian almost
complex structure $J_h$ associated to the Levi-Civita connections of $h$.
First the Levi-Civita connection induces the splitting
$$
T_{(q,p)}(T^*\H^3) \simeq T_q \H^3 \oplus T_q^*\H^3
$$
at each point $(q,p) \in T^*\H^3$, where the isomorphism is obtained by
\be\label{eq:split-iso}
H_j \mapsto \frac{\del}{\del q^j}, \qquad V_j \mapsto dq^j.
\ee
In terms of the splitting, the Sasakian almost complex Structure $J_h$
on $T^*\H^3$ are defined by
$$
J_h(X) = X^\flat, \qquad J_h(\alpha) = - \alpha^\sharp
$$
for $X \in T\H^3$ and $\alpha \in T^*\H^3$ respectively.

In the canonical coordinates, the almost complex structure $J=J_h$ is given by the formulae
\begin{align*}
J_h:T(T^*\H^3)&\to T(T^*\H^3);\\
H_i&\mapsto h_{ij}V_j;\\
V_i&\mapsto -h^{ij}H_j,
\end{align*}
which can be expressed in the following matrix
\begin{displaymath}
\left(\begin{array}{c|c}
0 & -h^{ij} \\
\hline h_{ij} & 0
\end{array}\right).
\end{displaymath}
with respect to the above frame fields.
Then the compatibility condition
$$
\widetilde h(\cdot,\cdot)=\omega(\cdot,J_h\cdot)
$$
between the triple $(\widetilde h,\omega,J_h)$ can be guaranteed by the following matrix multiplication:
\begin{displaymath}
\left(\begin{array}{c|c}
h_{ij} & 0 \\
\hline
0& h^{ij}
\end{array}\right)
=
\left(\begin{array}{c|c}
0 & \delta_{ij} \\
\hline
-\delta_{ij}& 0
\end{array}\right)\cdot
\left(\begin{array}{c|c}
0 & -h^{ij} \\
\hline
h_{ij} & 0
\end{array}\right)
\end{displaymath}

The following proposition enables us to anticipate something good about the study of
perturbed Cauchy-Riemann equation we will carry out.

\begin{prop}\label{prop:Jhconvex} Let $b= b_\delta: N \to \R$ be a Busemann function
such that its lift $\widetilde b = b_\delta \circ p: \H^3 \to \R$ satisfies
$e^b \circ p = \frac1z$ near the ideal boundary. Denote by $\pi: T^*N \to N$ the canonical projection.
Then the function $f: T^*N \to \R_+$ defined by
$$
f = H_h + e^b \circ \pi,
$$
is a strictly $J_h$-pluri-subharmonic exhaustion function `at infinity'.
\end{prop}
\begin{proof}
Let $(x,y,z,p_x,p_y,p_z)$ be the canonical coordinates for the $T^*N$ with $N = M \setminus K$.
With slight abuse of notations, we also denote the canonical coordinates of $T^*\H^3$ by
the same.
Then by the hypothesis, we have the formula for the lift $\widetilde f: \H^3 \to \R$ of $f$
$$
\widetilde f = H_{\widetilde h} + \frac1z = \frac{1}{2 z^2}(p_x^2 + p_y^2 + p_z^2) + \frac1z.
$$
A straightforward computation leads to
\beastar
-d(dH_{\widetilde h}\circ J_h) & = & \frac{1}{z^2}dz \wedge \theta + \frac{1}{z} \omega_0,\\
-d(d(1/z)\circ J_h) & = & \omega_0
\eeastar
and hence
$$
-d(d \widetilde f \circ J_h) = \frac{1}{z^2}dz \wedge \theta + \left(1+\frac{1}{z}\right) \omega_0.
$$
It is easy to check that
$$
-d(d \widetilde f \circ J_h)(V_i, J V_i) = \frac{1+z}{z^3} > 0
$$
for each $V_i=\partial_{p_i}$ with $i= x,y,z$. Since $p: \H^3 \to N$ is a local isometry,
this implies $-d(d f \circ J)(X,JX) \geq 0$ and
equality holds only when $X = 0$ for any $X \in T(T^*N)$. The exhaustion property
of $f$ follows immediately from the property that $e^b$ is an
exhaustion function of the base $N$.
This proves that the function $f$ is $J_h$-pluri-subharmonic.
\end{proof}

\subsection{Horo-tori and conormal Lagrangians}\label{sec:Lagrangian}
We consider the universal cover $p:\H^3\to N$
and then the lifted metric $p^*g_N$ on $\H^3$ coincide with $g_{\H^3}$.

Let us denote the deck-transformation group of the covering $p$ by $\Gamma$.

For a given horo-torus $T$ in $N$, we consider its lifting
$\widetilde T:=p^{-1}(T)$ to $\H^3$. It is a disjoint union
\[
\widetilde T=\bigsqcup_{g\in G}g\cdot\{z=a\}
\]
where $a$ is a sufficiently large positive real number satisfying $\{z=a\}\cap g\cdot\{z=a\}=\emptyset$ for any $g\in \Gamma\setminus\{{\rm id}\}$.
For a given submanifold $S$ in $N$, we consider its conormal as a (exact) Lagrangian submanifold in $T^*N$
\[
\nu^*S:=\{(q,p)\in T^*N\,|\,q\in S,\ \text{and}\ \langle p,v\rangle=0, \forall v\in T_q S \}.
\]

Next we derive some geometric property of horospheres in $\H^3$. By applying the action of $\psl$,
we are reduced to the study of hyperplane
\[\{(x,y,z) \in \H^3 \mid z = z_0\}\]
for given constant $z_0 > 0$.
The following geometric property, constant mean curvature property with positivity is a crucial
ingredient entering in the proof of rigidity and nondegeneracy of Hamiltonian chords attached to
the conormal $\nu^*T$ (Theorem \ref{thm:nondeg}).

\begin{prop}\label{prop:meancurvature} Denote by ${\bf N}$ be the outward unit normal to the plane
$\{z = z_0\}$ in $\H^3$. Then the mean-curvature vector $\vec H$ relative to ${\bf N}$ is given by
$$
\vec H = {\bf N}.
$$
In particular the mean-curvature is 1 for all $z_0 > 0$.
\end{prop}
\begin{proof}
Recall the hyperbolic metric $h_{ij}=\frac{1}{z^2} \delta_{ij}$ on the upper half space model
of $\H^3$. Then we have covariant derivative  on $\H^3$ as follows,
\be\label{eq:cov-del}
\begin{aligned}
\nabla_{\partial_*}\partial_* &= \frac{1}{z}\partial_z &&\text{ for }~ *=x,y; &
\nabla _{\partial_x}\partial_y & =\nabla _{\partial_y}\partial_x=0; \nonumber \\
\nabla_{\partial_*}\partial_z &= -\frac{1}{z}\partial_* &&\text{ for }~ *=x,y,z. & &&
\end{aligned}
\ee
Let us consider a horo-sphere $\widetilde N$ centered at $\infty$ with $\{(x,y,z)\in \H^3 : z=t_0\}$.
A unit normal vector $\mathbf{N}$ toward $\infty$ at $q \in \widetilde N$, i.e., outward normal, is given by
$\mathbf{N} = t_0 \partial_z$ and an orthonormal basis of $T_pH$ is $\{ t_0 \partial_x,  t_0 \partial_y \}$.
Let us compute the shape operator $S_\mathbf{N}$ as follows.
\be
S_\mathbf{N}(\partial_*) = - \nabla_{\partial_*}(\mathbf{N})=\partial_* ~~~ \text{ for } *=x,y.
\ee
This proves the first statement.

A mean curvature $H$ is computed by the second fundamental form as follows,
\be
H= \frac{1}{2} \left\{ \langle S_\mathbf{N}(t_0 \partial_x) , t_0 \partial_x \rangle_h
+ \langle S_\mathbf{N}(t_0 \partial_y) , t_0 \partial_y \rangle_h \right\} = 1.
\ee
This finishes the proof.
\end{proof}

\subsection{$C^0$ bounds of Neumann geodesic cords of horo-torus}

In this subsection, we provide the following classification result on
the Neumann geodesic cords of  a horo-torus $T= \del N_i$.

\begin{lem}\label{lem:length}
Consider a geodesic $c:[0,\ell] \to M \setminus K$ with
$c(0), \, c(1) \in T$. Lift $\widetilde T$ to a lift to a horo-sphere $S_0 = \{z = a_0 \}$
in $\H^3$ and $c$ to a geodesic $\widetilde c$ in $\H^3$ so that $\widetilde c(0) \in S_0$
that is perpendicular to $S_0$.
Let $z_\ell = z(\widetilde c(\ell))$. Consider the function $f(t) = \frac{1}{z(\widetilde c(t))}$. Then
$$
f(t) =  a_0 \cosh(\ell t) + b_0 \sinh(\ell t)
$$
for which $|b_0|$ is bounded by a constant depending only on $S_0$ and $\ell$.
In particular
\be\label{eq:bound-|f|}
\max_{t \in [0,1]} |f(t)| \leq a_0 \cosh \ell + |b_0| \sinh \ell.
\ee
\end{lem}
\begin{proof}
We denote $c(t) = (c_x(t), c_y(t), c_z(t))$ in the standard coordinates $(x,y,z)$ of
$\H^3 \subset \R^3$ with the identification $\H^3 = \{(x,y,z) \mid z > 0\}$.

We have
$$
f'(t) =  - \frac1{z^2(c(t))} \dot c_z(t)
$$
and
$$
f''(t) = \frac{2}{z(c(t))^3} |\dot c_z(t)|^2 - \frac1{z^2(c(t))} \frac{d}{dt} (dz(\dot c(t))).
$$
Since $c$ is a geodesic, we also have
$$
\frac{d}{dt} (dz(\dot c))  =  \nabla_t(dz)(\dot c).
$$
A straightforward computation using \eqref{eq:R-chris} gives rise to
$$
\nabla_t (dz)  =  \frac1z( - \dot c_x dx -\dot c_y dy + \dot c_z dz)
$$
and hence
$$
\frac1{z^2(c(t))} \frac{d}{dt} (dz(\dot c)) = \frac1{z^3(c(t))}(-|\dot c_x|^2 - |\dot c_y|^2 + |\dot c_z|^2).
$$
Combining the above calculations, we obtain
\be\label{eq:f''}
f''(t) = \frac{1}{z(c(t))^3}(|\dot c_x(t)|^2 + |\dot c_y(t)|^2 + |\dot c_z(t)|^2) = \frac1{z(c(t))}|
\dot c(t)|_h^2= \ell^2 f(t)
\ee
where $\ell \equiv |\dot c(t)|_h$ is the length of the geodesic cord $c$.
(Recall that geodesic has constant speed.)
In other words, $f$ satisfies
$$
f''(t) - \ell^2 f(t) = 0.
$$
Its general solution is given by
$$
f(t) = a \cosh(\ell t) + b \sinh(\ell t).
$$
By $f(0) = a_0$, we obtain $a = a_0$.  Therefore we obtain
$$
|f(t)| \leq |a_0| \cosh(\ell) + |b_0| \sinh(\ell)
$$
for all $t \in [0,1]$ noting that both $\cosh$ and $\sinh$ are increasing for $t > 0$.
Finally we note that
$$
f'(0) = \ell b_0
$$
from the above formula. On the other hand, by definition $f(t) = \frac{1}{z(c(t))}$, we also have
$$
f'(0) = -\frac{1}{z(c(t))^2} c_z'(0).
$$
From this we also have
$$
|f'(0)| = \frac{|c_z'(0)|}{z(c(0))^2} \leq \frac{1}{z(c(0))} \|c'(0)\|_h = \frac{\ell}{a_0}.
$$
Comparing the two, we obtain
$$
|b_0| \leq  \frac{\ell}{a_0}.
$$
This finishes the proof.
\end{proof}

This lemma completely determines the $z$-coordinate of any geodesic
with the initial data on $c(0), \, c'(0)$.

\begin{rem}\label{rem:length} Examination of the proof of Lemma \ref{lem:length} shows that
the same kind of estimate holds for arbitrary admissible Lagrangian submanifolds, not just for
the horo-spheres. See \cite[Definition 4.1]{BKO} for the definition of admissible Lagrangian
submanifolds.
\end{rem}

\section{The kinetic energy Hamiltonian and hyperbolic geodesics}
\label{sec:hamiltonian}

We endow $T^*N$, $N=M\setminus K$ with the canonical symplectic structure
$$
\omega_0= dq \wedge dp
$$
and the kinetic energy Hamiltonian function $H=H_h:T^*N\to \R$ defined by
\[
H_h(q,p)=\frac{1}{2}|p|^2_{h^\flat},
\]
where $h^\flat$ is the dual metric of $h_N$. It is well-known, see \cite{klingenberg},
the Hamiltonian flow of $H$ is nothing but the geodesic flow which is given by
\begin{align*}\label{eq:geo-flow}
(t,(q,p)) \mapsto &\  \big(\exp_q(t p^\sharp), (d\exp_q(t p^\sharp)(p^\sharp))^\flat\big)
\end{align*}

\begin{defn}\label{defn:chord-cord} We denote
\beastar
{\rm Chord}(\nu^*T) & : = & \mathfrak X_{< 0}(\nu^*T,\nu^*T;H_h) \\
& = & \{\gamma:[0,1] \to T^*N \mid \dot \gamma(t) = X_{H_h}(\gamma(t)), \, \gamma(0), \, \gamma(1) \in \nu^*T, \, \text{non-constant}\}.
\eeastar
We call an element of $\text{Chord}(\nu^*T)$ a (non-constant) \emph{Hamiltonian chord} of $\nu^*T$.
\end{defn}

We recall from Definition \ref{defn:cord} the definition
$$
{\rm Cord}(T)  =  \{ c:[0,1] \to N \mid \nabla_t \dot c = 0,\,c(i) \in T, \,\, \dot c(i) \perp T, \,\text{for } i=0, 1\}
$$
of geodesic cords of $T$. We have the following one-one correspondence

\begin{lem}\label{lem:cord_chord}
There is a natural one-one correspondence between $\text{\rm Chord}(\nu^*T)$ and
$\text{\rm Cord}(T)$.
\end{lem}
\begin{proof} Let $c \in \text{Cord}(T)$, then $\gamma_c:=(c,\dot c\, ^\flat)$ defines a
Hamiltonian trajectory of $H$. Furthermore since $\dot c(0), \, \dot c(1) \perp T$ and
$c(1) = d\exp_q( p^\sharp)(p^\sharp)$,  $(d\exp_q(p^\sharp)(p^\sharp))^\flat \in N_{c(1)}^*T$. Therefore the assignment
\be\label{eq:Phi}
\Phi: c \mapsto \gamma_c; \quad \gamma_c(t): = \left(\exp_q(t p^\sharp), (d(\exp_q(t p^\sharp))(p^\sharp))^\flat\right)
= (c_\gamma(t), \dot c_\gamma(t)^\flat)
\ee
defines a map $\text{Cord}(T) \to \text{Chord}(\nu^*T)$. Conversely, we check that any
$\gamma \in \text{Chord}(\nu^*T)$ can be written as
$$
\gamma = \Phi(c_\gamma); \quad c_\gamma = \pi \circ \gamma.
$$
This finishes the proof.
\end{proof}

To utilize the presence of hyperbolic structure on $N$, we lift
the Hamiltonian flow to the universal covering $T^*\H^3$ of $T^*N$.

In order to lift the Hamiltonian function to $T^*\H^3$, we first recall a symplectomorphism
$T^*g:T^*\H^3\to T^*\H^3$ induced by a diffeomorphism $g:\H^3\to \H^3$ which is an element
of the deck transformation group $\Gamma$. Here the map $T^*g$ is defined by
$$
T^*g(q,p)=\left(g(q),((d_qg)^{-1})^*p\right)
$$
for any $v\in T_{g(q)}\H^3$.
By the choice of metric on the universal cover,
$g:\H^3\to\H^3$ is an isometry for any $g\in \Gamma$ and hence we have
$
|(g^{-1})^*p|_{h^\flat}=|p|_{h^\flat},
$
where $h^\flat= h_{\H^3}^\flat$ is the dual metric of $h_{\H^3}$. For the simplicity of notations, we
will suppress $\flat$ from its notation and just denote by $h$ for both of them.
As a consequence, an induced Hamiltonian function $\widetilde H:T^*\H^3\to\R$;
\[
\widetilde H(q,p)=\frac{1}{2}|p|^2_{h}
\]
satisfies $\widetilde H=\widetilde H\circ g$ for any $g \in \Gamma$ which implies $\widetilde H=H\circ (T^*p)$, where $T^*p: T^*\H^3 \to T^*N$ is the map defined by
$$
T^*p(q,p)=\left(p(q),((d_q p)^{-1})^*p\right)
$$
over the covering map $p:\H^3 \to N$. It defines a well-defined map because $p$ is a covering map and so $d_q p$ is invertible for all $q \in \H^3$. The map $T^*p$ itself defines a covering map $T^*\H^3\to T^*N$ which
covers $p: \H^3 \to N$.
In local coordinates $(x,y,z,p_x,p_y,p_z)$ for $T^*\H^3$, $\widetilde H$ and its Hamiltonian vector field
$X_{\widetilde H}$ are expressed by
\begin{align*}
\widetilde H &=\frac{1}{2}z^2(p_x^2+p_y^2+p_z^2);\\
X_{\widetilde H} &=z^2(p_x\partial_x+p_y\partial_y+p_z\partial_z)-z(p_x^2+p_y^2+p_z^2)\partial_{p_z}\\
&=z^2(p_x H_1+p_y H_2+p_z H_3)
\end{align*}
where $X_{\widetilde H}$ is defined to satisfy $\omega_{0}(X_{\widetilde H},\ \cdot\ )=d{\widetilde H}$.
Since $\widetilde H$ is given by a kinetic energy, it is well-known that the flow of $X_{\widetilde H}$ recovers the geodesic flow on the cotangent bundle.

A direct computation shows that
\begin{align*}
\mathrm{d}\widetilde H&=z(p_x^2+p_y^2+p_z^2)\mathrm{d}z+z^2(p_x\mathrm{d}p_x+p_y\mathrm{d}p_y+p_z\mathrm{d}p_z)\\
&=z^2(p_xV^1+p_yV^2+p_zV^3)\\
d\widetilde H\circ {\widetilde J}&=p_xdx+p_ydy+p_zdz=\widetilde\theta.
\end{align*}
Here we recall that our canonical symplectic structure is $\omega_0=-d\theta$.

We will consistently denote by $\widetilde \gamma$ (resp. {$\widetilde c$}) the lifting of
$\gamma$ (resp. $c$) to $T^*\H^3$ (resp. $\H^3$) in this paper.

\section{Second variation and Morse index calculation}
\label{sec:index}

The purpose of the present section is to provide the proofs of two results concerning the relationship
between the Maslov index of a Hamiltonian chord
$\gamma$ attached to the conormal $\nu^*T$ and the Morse index of
the associated geodesic cord $c = \pi \circ \gamma$.

For this we first need to explain what we mean by the `Mores index' of the geodesic cord $c$ attached to
the base manifold $T \subset N$. We recall the path space on which the Floer theory applied is
$$
\Omega(\nu^*T) = \Omega(T^*N,\nu^*T) = \{ \gamma:[0,1] \to T^*N \mid \gamma(0), \, \gamma(1) \in \nu^*T\}.
$$
The projection $c = \pi \circ \gamma$ of each element $\gamma$ automatically satisfies
the free boundary condition
$$
c(0),c(1) \in T,
$$
which leads us to considering the space of paths
$$
\CP(T) = \CP(N,T) = \{ c: [0,1] \to N \mid c(0),c(1) \in T\}
$$
and the energy functional
$$
E(c) = \frac12 \int_0^1 |\dot c|_h^2\, dt
$$
restricted thereto. The general first variation formula is given by
\be\label{eq:1stvariation}
dE(c)(V) = - \int_0^1 \left\langle V(t), \frac{ D}{d t}\frac{dc}{dt} \right \rangle\, dt -
\left\langle V(0), \frac{dc}{dt}(0) \right\rangle + \left \langle V(1), \frac{dc}{dt}(1) \right\rangle.
\ee
In particular, standard variational analysis shows that
if $c \in \CP(T)$ is a critical point, then it satisfies
$$
\begin{dcases}
\frac{D}{dt}\frac{dc}{dt} = 0;\\
c(0), \, c(1) \in T;\\
\dot c(0) \in \nu_{c(0)}T, \quad \dot c(1) \in \nu_{c(1)}T.
\end{dcases}
$$
Next we recall the general second variation formula. (See \cite[p.199]{doCarmo} for
such a derivation. We warn the readers that \emph{the definition of the curvature operator
	$R(X,Y)$ in \cite{doCarmo} is the negative of the one used here or in \cite{spivakIV,KN2}.})

\begin{lem}[Second variation formula]\label{lem:second} Let $c:[0,1] \to (N,g)$ be a
	geodesic on a Riemannian manifold.
	Let $C:(-\epsilon,\epsilon) \times [0,1] \to N$ be a variation of $c$ in $\CP(T)$
	i.e, a map satisfying
	$$
	C(0,t) = c(t),\quad C(s,\cdot) \in \CP(T).
	$$
	Denote $\displaystyle V(s,t) = \frac{\del C}{\del s}(s,t)$and $c_s: = C(s,\cdot)$.
	The second variational formula \emph{restricted to the path space $\CP(N,T)$} at a critical point $c$ is given by
	\begin{align}
	d^2E(c)(V,V)
	 =& -\int_0^1 \left \langle \frac{D^2V}{\del t^2}(0,t)+ R(V(0,t),\dot c)\dot c, V\right\rangle\, dt\nonumber \\
	& - \left\langle \frac{DV}{\del t}(0,0), V(0,0)\right\rangle + \left\langle \frac{DV}{\del t}(0,1), V(0,1)\right\rangle
	\nonumber\\
	 =&  \int_0^1 \left \langle \frac{DV}{\del t}(0,t), \frac{DV}{\del t}(0,t)\right\rangle
	- \int_0^1\left\langle R(V(0,t),\dot c)\dot c, V(0,t)\right\rangle\, dt\nonumber\\
	& - \left\langle \frac{D V}{\del s}(0,0), \dot c(0)\right\rangle
	+ \left \langle \frac{DV}{\del s}(0,1), \dot c(1)\right\rangle
	\label{eq:2ndvariation-2}
	\end{align}
\end{lem}
We would like to emphasize that the general second variation formula
\emph{allowing the end points to move} as in the free boundary problem on $T$
contains the boundary terms appearing in the last line of \eqref{eq:2ndvariation-2} which is
the same as
$$
- \left\langle \frac{D}{\del s}\frac{\del C}{\del s}, \frac{\del C}{\del t}\right \rangle(0,0)+
\left\langle \frac{D}{\del s}\frac{\del C}{\del s}, \frac{\del C}{\del t}\right \rangle(0,1).
$$
(These terms will not appear in \eqref{eq:2ndvariation-2}
\emph{when the variation is fixed at the end $t = 0,\, 1$}. See e.g., \cite[p.303]{spivakIV} or
\cite[Remark 2,10]{doCarmo} for such a discussion.)

We remark that positivity of the mean-curvature in Proposition~\ref{prop:meancurvature} is consistent with the first variation of area formula
since the area of horo-torus decreases along the outward normal direction.

\begin{prop}\label{prop:formula-2nd} Let $N$ be a hyperbolic manifold and $g$ be a hyperbolic metric on it.
	Consider a horo-torus $T = b^{-1}(r_0)$.
	Let $V$ be a variational vector field along $c$ on $\CP(T)$. Then
	\be\label{eq:formula-2nd}
	d^2E(c)(V,V) = \int_0^1 \left(\left|\frac{DV}{dt}\right|^2 +|\dot c|^2 |V|^2\right)\,  dt + |\dot c(0)||V(0)|^2 + |\dot c(1)||V(1)|^2
	\ee
	for all $V \in T_c\CP(T)$ where $\ell$ is the length of the geodesic $c$.
\end{prop}
\begin{proof}
	We first examine the boundary terms of \eqref{eq:2ndvariation-2}. Denote by ${\bf N}$ the outward unit normal
	to $T$. Then we have ${\bf N}(c(0)) = - V/|V|(0,0)$ and ${\bf N}(c(1)) = V/|V|(0,1)$.
	Therefore by the definition of the shape operator $S_{\bf N}$, we have
	$$
	-  \left\langle \frac{D V}{\del s}(0,0), \dot c(0)\right\rangle =
	- |V(0,0)| \langle S_{\bf N}(V), V \rangle (0,0) = |\dot c(0)||V(0,0)|^2
	$$
	where the last equality comes from Proposition \ref{prop:meancurvature}.
	Similarly we obtain
	$$
	 \left\langle \frac{D V}{\del s}(0,1), \dot c(1)\right\rangle
	= |\dot c(1)||V(0,1)|^2.
	$$
	Therefore we have derived
	$$
	-  \left\langle \frac{D V}{\del s}(0,0), \dot c(0)\right\rangle
    +  \left\langle \frac{D V}{\del s}(0,1), \dot c(1)\right\rangle
	= |\dot c_s(0))||V(s,0)|^2 + |\dot c_s(1)||V(s,1)|^2.
	$$
	
	For the second term of \eqref{eq:2ndvariation-2},
	since $g$ has constant curvature $-1$ and so
\be\label{eq:curvature}
R(X,Y)Z = \langle Z, X \rangle Y - \langle Z,Y\rangle X
\ee
and $V \perp \dot c = 0$, it becomes
	$$
	- \int_0^1\left\langle R(V,\dot c)\dot c, V\right\rangle\, dt = \int_0^1 |\dot c|^2 |V|^2\, dt.
	$$
	Combining all these, we have derived
	$$
	d^2E(c)(V,V) =
	\int_0^1 \left(\left|\frac{DV}{dt}\right|^2 +|\dot c|^2 |V|^2\right)\,  dt +
	|\dot c_s(0)||V(s,0)|^2 + |\dot c_s(1)||V(s,1)|^2.
	$$
	Evaluation of this at $s=0$ gives rise to \eqref{eq:formula-2nd}.
\end{proof}

An immediate corollary of Proposition \ref{prop:formula-2nd} is the following vanishing result.

\begin{thm}\label{thm:index-nullity} Let $N$ and $T$ be as in Proposition \ref{prop:formula-2nd}.
	For any geodesic cord $c \in \text{\rm Cord}(T)$, both nullity and
	Morse index of $c$ vanish.
\end{thm}
\begin{proof} Vanishing of index follows from the non-negativity of the second variation \eqref{eq:formula-2nd}.
	
	Next we examine the nullity of $c$ i.e., the dimension of the set of the Jacobi fields $V$ satisfying
	\be\label{eq:Jacobi}
	\frac{D^2V}{dt^2} - \ell^2 V = 0, \quad V(0), \, V(1) \in TT
	\ee
	where $\ell = |\dot c|$ is the length of  $c$.  There is a one-one correspondence between
	this set and the null space of $d^2E(c)$. In particular each such $V$ satisfies
	$$
	\int_0^1 \left(\left|\frac{DV}{dt}\right|^2 + \ell^2 |V|^2\right) \, dt = 0.
	$$
Therefore we have $\frac{DV}{dt} = 0 = \ell V$. Therefore any kernel element $V$ is
covariant constant and so its initial value $V(0) \in T_{c(0)} T$ uniquely
determines $V$. Since $\ell \neq 0$ in addition, $\ell V = 0$ implies $V = 0$.

Summarizing the above discussion, we have finished the proof.
\end{proof}

\section{Perturbed Cauchy-Riemann equation and $A_\infty$ structure}
\label{sec:perturbedCR}

We start with the following general property of Sasakian almost complex structure $J_h$
associated to the Riemannian metric $h$.

\begin{lem} The almost complex structure $J_h$ is of contact type on $T^*N$, i.e.,
$$
(-\theta) \circ J_h = dH_h
$$
for the kinetic energy Hamiltonian $H_h = \frac{1}{2} |p|_h^2$.
\end{lem}

This is a general fact for any Sasakian almost complex structure
on the cotangent bundle ${T^*N}$ associated to the Riemannian metric $g$ on $N$
and its kinetic energy $H_h = \frac{1}{2} |p|_h^2$.
For the simplicity of notation, we will just denote $J = J_h$ if there is no confusion.

Recall the conormals $\nu^* T$ are the main object in the current paper and
satisfy the admissible Lagrangian conditions given in \cite[Definition 4.2]{BKO}.
Noting the vanishing $\theta|_{\nu^*T} = 0$ and from our convention of the action functional,
the relevant action functional is given by
\begin{align}\label{eq:action}
\mathcal A_{H_h}(x)= - \int_{0}^{1}x^*\theta  + \int_{0}^{1} H_h(x(t))\, dt
\end{align}
on the path space
$$
\Omega(\nu^*T,T^*N) = \{x \in [0,1] \to T^*N \mid x(0), \, x(1) \in \nu^*T\}.
$$
We recall the set of time-one Hamiltonian chords
$$
\frak{X}(H_h;\nu^*T,\nu^*T) = \{\gamma:[0,1] \to T^*N \mid \dot \gamma(t) = X_{H_h}(\gamma(t)),
\text{ and } \, \gamma(0), \, \gamma(1) \in \nu^*T\}
$$
and denoted $\text{\rm Chord}(\nu^*T)$ the subset of non-constant Hamiltonian chords in
Definition \ref{defn:chord-cord}.
In addition to Theorem \ref{thm:index-nullity}, it follows from
Proposition \ref{prop:Bott-clean} that the action functional $\CA_{H_h}: \Omega(T^*N, \nu^*T) \to \R$ associated
to the hyperbolic metric $h$ is nondegenerate in the sense of Bott.

The $L^2$-gradient vector field of $\CA_H$ is given by
$$
\text{\rm grad} \CA_H(\gamma) = -J (\dot \gamma - X_H(\gamma))
$$
and hence the relevant (positive) gradient flow equation is given by \eqref{eq:CRXH}.

The relevant energy $E(u)$ is defined by
\begin{align*}
\int_{\Sigma}\frac{1}{2}|du-X_{\bf H}(u)\otimes \beta|_{\bf J}^2
\end{align*}
for arbitrary smooth map $u: \Sigma \to T^*N$.
For a solution $u:\Sigma\to T^*N$ of \eqref{eq:duXHJ}, we have the following
\begin{align}
E(u)&=\int_{\Sigma}u^*\omega-u^*d{\bf H} \wedge \beta{\nonumber}\\
&\leq\int_{\Sigma}u^*\omega-u^*d{\bf H} \wedge \beta-\int_{\Sigma}u^*{\bf H}\cdot d\beta\label{eq:energy_estimate_1}\\
&=\int_{\Sigma}u^*\omega-d(u^*{\bf H}\cdot \beta)\label{eq:energy_estimate_2},
\end{align}
where the inequality in (\ref{eq:energy_estimate_1}) comes from $\mathbf H \geq 0$ and sub-closedness of $\beta$.
By the Stoke's theorem with the fixed Lagrangian boundary condition in (\ref{eq:duXHJ}) and $\beta|_{\partial\Sigma}=0$ implies that (\ref{eq:energy_estimate_2}) becomes
\begin{align*}
\mathcal A(x^0)- \sum_{j=1}^{k}\mathcal A(x^j).
\end{align*}

Recalling Proposition \ref{prop:Bott-clean}, we consider  $L = \nu_{k,\rho}^*T$ defined in \eqref{eq:nu-krho-T}.
Then we consider the reduced wrapped Floer complex
$$
C = \widetilde{CW}_{g}^d(T, M\setminus K): = \widetilde{CW}^d(L,L; T^*(M\setminus K); H_{g_0}).
$$
Then the construction in \cite{BKO} associates an $A_\infty$ algebra to
$CW_{g_0}(T, M\setminus K)$, whose construction we recall below in the context of
$A_\infty$ algebra. We denote the associated cohomology by
\be\label{eq:HWHh}
HW_g(T,M \setminus K) =  HW(L,L; T^*(M\setminus K); H_{g_0}).
\ee
To define the $A_\infty$-maps
$$
\frak{m}^k: C^{\otimes k}\to C[2-k],
$$
we need a moduli space of perturbed $J$-holomorphic curves with respect to the following Floer data.

Now we adopt the Floer data in \cite{BKO} to the current setup.
We first briefly recall the construction of one-form $\beta$ on the domain to be used in writing
down the perturbed Cauchy-Riemann equation.

For each given $k \geq 1$, let us consider a Riemann surface $(\Sigma,j)$ of genus zero with $(k+1)$-ends.
Each end admits a holomorphic embedding
\begin{align*}
\begin{cases}
\epsilon^0: Z_- :=\{\tau \leq 0\} \times [0,1] \to \Sigma;\\
\epsilon^\ell: Z_+ :=\{\tau \geq 0\} \times [0,1] \to \Sigma, \quad\text{ for }i=1,\dots, k
\end{cases}
\end{align*}
This is isomorphic to the closed unit disk $\D^2$ minus $k+1$ boundary points ${\bf z}=\{z^0, \dots, z^k\}$ in the counterclockwise direction. Suppose that the ends are decorated by a weight datum ${\bf w}=\{w^0,\dots , w^k\}$ satisfying the balancing condition
\begin{align}\label{eqn:balancing}
w^0=w^1+\dots + w^k.
\end{align}
%It follows from Proposition \ref{prop:Bott-clean} (3) that we can choose the perturbed conormal
%$\nu^*_{k,\rho}T$ so that the weights $w^i$'s avoid the set $\{k \pi \mid k \in \Z\} \subset \R$
%so that we can apply Proposition \ref{prop:Bott-clean} (2) to deform $\nu^*T$ to
%$\nu^*_{k,\rho}T$ so that the pair $(\nu^*_{k,\rho}, H_h)$ is nondegenerate and satisfies
%Lemma \ref{lem:perturbed-cord}.

Then the one form $\beta$ on $\Sigma$ satisfies the following conditions:
\be\label{eq:beta}
\begin{cases}
d\beta = 0  \\
d(\beta \circ j) = 0 \\
i^*\beta = 0 \quad & \text{for the inclusion }\, i: \del \Sigma \to \Sigma\\
(\epsilon^j)^*\beta = w^j dt \quad & \text{on a subset of $Z_\pm$ where $\pm \tau \gg 0$.}
\end{cases}
\ee

We utilize the slit domain for the construction of $\beta$.
Let us consider domains
\begin{align*}
Z^1 &=\{ \tau+\sqrt{-1}\, t\in \C \mid \tau\in\R,\ t \in [0, w^1] \};\\
Z^2 &=\{ \tau+\sqrt{-1}\, t\in \C \mid \tau\in\R,\ t \in [w^1, w^1 + w^2] \};\\
&\vdots\\
Z^k &=\{ \tau+\sqrt{-1}\, t\in \C \mid \tau\in\R,\ t \in [w^1+\cdots+w^k, w^0] \},
\end{align*}
and its gluing along the inclusions of the following rays
\begin{align*}
R^{\ell}=\{\tau+\sqrt{-1}\, t\in \C \mid \tau\leq s^{\ell},\ t=w^1+\cdots+ w^\ell \};\\
j_-^\ell:R^\ell\hookrightarrow Z^\ell,\qquad j_+^\ell:R^\ell\hookrightarrow Z^{\ell+1},
\end{align*}
for some $s^\ell\in \R$ and for $\ell=1,\dots,k-1$.
In other words, the glued domain becomes
\[
Z^{\bf w}({\bf s})=Z^1 \sqcup Z^2 \sqcup \cdots \sqcup Z^k/ \sim,
\]
where ${\bf s}=\{s^1,\dots, s^{k-1}\}$, and $\zeta\sim \zeta'$ means $(\zeta,\zeta')\in Z^{\ell}\times Z^{\ell+1}$ and $\zeta=\zeta'\in R^\ell$ for some $\ell=1,\dots,k-1$.
We may regards
\begin{align*}
Z^0\subset \{\tau+\sqrt{-1}\, t\in \C \mid \tau\in\R,\  t\in[0,w^0]\},
\end{align*}
with $(k-1)$-slits
\begin{align*}
S^{\ell}=\{\tau + \sqrt{-1}\, t_\ell \in \C \mid \tau\geq s^\ell,\ t_i=w^1+\cdots+w^{\ell} \}
\end{align*}
where $\ell=1,\dots,k-1$. Then for $k\geq 2$ there is a unique conformal mapping
\[
\varphi:\Sigma=\D^2\setminus \{z^0,\dots,z^k\}\to Z^{\bf w}({\bf s})
\]
for some rays $\{R^\ell\}_{\ell=1,\dots,k-1}$ with respect to ${\bf s}$ satisfying
the following asymptotic conditions
\begin{align*}
\begin{cases}
\lim_{\tau\to -\infty}\mathring\varphi^{-1}(\{\tau\}\times(\sum_{j=1}^{\ell-1}w^j,\sum_{j=1}^{\ell}w^j))=z^\ell \quad \text{ for } \ell=1,\dots,k;\\
\lim_{\tau\to \infty}\mathring\varphi^{-1}(\{\tau\}\times(0,w^0))=z^0.
\end{cases}
\end{align*}
For $k=1$, the conformal mapping is unique modulo $\tau$-translations.
Let us denote such a slit domain by $Z^{\bf w}$ or simply $Z$ if there is no confusion.
Then the one form $\beta$ is defined to be
$$
\beta= \varphi^*dt\in \Omega^1(\Sigma)
$$
with $dt\in \Omega^1(Z)$.

In the rest of this article, we often identify $\Sigma$ with the slit domain $Z= Z^{\bf w}$
via the conformal map $\varphi$.

We first recall the standard Floer data used in the construction of wrapped Floer
homology in general. We especially refer to \cite[Defnition 5.1 \& Remark 5.2]{BKO}
for some adjustment relate to Statement (4) below.

\begin{defn}[Floer data for $A_\infty$ map]\label{def:Floer data}
A {\em Floer datum} $\mathcal D_{\frak m}=\mathcal D_{\frak m}{(\Sigma,j)}$ on a stable disk $(\Sigma,j)\in \overline{\mathcal M}{}{}^{k+1}$ is the following:
\begin{enumerate}
\item {\bf Weights}: $\mathbf w=(w^0,\dots,w^k)$ a $(k+1)$-tuple of positive real numbers which is assigned to $\mathbf z=(z^0,\dots,z^k)$ satisfying the balancing condition (\ref{eqn:balancing}).

\item {\bf One-form}: $\beta\in\Omega^1(\Sigma)$ constructed above
 satisfying $\epsilon^{j*}\beta$ agrees with $w^jdt$.

\item {\bf Almost complex structure}:
A map $\mathbf J:\Sigma\to\mathcal J(T^*N)$ whose pull-back under $\epsilon^j$ uniformly converges to $\psi_{w^j}^*J_t$ near each $z^j$ for some $J_t\in\mathcal J_i$.

\item {\bf Vertical moving boundary}: $\eta\in C^\infty(\Sigma, [1,+\infty))$ which converges to $w^j$
near each $z^j$.
\end{enumerate}
\end{defn}

In general we need a domain dependent admissible Hamiltonian datum $\mathbf H$ which is uniformly converges to $\frac{H}{(w^{j})^2}\circ \psi_{w^j}$ near each $z^j$ for some admissible Hamiltonian $H$. Since the kinetic Hamiltonian $H$ is quadratic, it satisfies  $\frac{H}{(w^{j})^2}\circ \psi_{w^j}=H$ for any positive $w^j$, and we take domain {\em independent} Hamiltonian $H$.

Following \cite{BKO}, we make the following specific choice of
domain {\em dependent} almost complex structures ${\bf J}$ defined by
$$
{\bf J}(\sigma) = \psi_{\eta(\sigma)}^*J_h
$$
at each $\sigma\in \Sigma$.

Then for a given $(k+1)$-tuple $(\gamma^0;\vec{\gamma})$ of Hamiltonian chords in $\frak{X}_i$,
we consider the moduli space $\CM(\gamma^0;\vec \gamma;\mathcal D_{\frak m}(\Sigma,j))$ of
maps $u:\Sigma\to T^*N$ satisfying
\be\label{eq:duXHJ}
\begin{cases}
(du- X_{\mathbf H}\otimes \beta)_{\mathbf J}^{(0,1)}=0;&\\
u(z)\in\psi_{\eta(z)}(L), &\text{ for } z\in\partial \Sigma;\\
u\circ \epsilon^j(-\infty,t)=\psi_{w^j}\circ \gamma^j(t), &\text{ for }j=1,\dots,k;\\
u\circ \epsilon^0(\infty,t)=\psi_{w^0}\circ \gamma^0(t).
\end{cases}
\ee

Now we consider a parameterized moduli space
\begin{align*}
\mathcal M(\gamma^0;\vec{\gamma})=\bigcup_{(\Sigma,j)\in
\mathcal M{}^{k+1}} \CM(\gamma^0;\vec{\gamma};\mathcal D_{\frak m}(\Sigma,j)).
\end{align*}
In order to use the moduli space in the definition on $A_\infty$ structure maps $\frak m^k$, we need to overcome compactness and the transversality issue. These are dealt in the coming sections.

Here we recall from \cite[Part 2]{BKO} that the transformation
\be\label{eq:utov}
u \mapsto \psi_\eta^{-1}\circ u=: v
\ee
transforms \eqref{eq:duXHJ} into the
\emph{autonomous} equation
\be\label{eq:dvXHJ}
\begin{cases}
(dv - X_H \otimes \beta)_{J_h}^{(0,1)}=0,\\
\text{$v(z)\in L$, for $z\in\partial \Sigma$} \\
v\circ \epsilon^j(-\infty,t)= x^j(t), \text{ for }j=1,\dots,k.\\
v\circ \epsilon^0(\infty,t)= x^0(t).
\end{cases}
\ee

\section{The maximum principle and $C^0$-estimates}
\label{sec:C0estimates}

\subsection{Vertical $C^0$-estimates}
\label{subsec:vertical}

For the study of compactness properties of the moduli space of
\eqref{eq:duXHJ}, the following \emph{vertical} $C^0$-bound is an essential step in the case of
noncompact Lagrangian such as the conormal Lagrangian $L = \nu^*T$.
For given $\gamma^j \in \Chord(H;L)$ with $j = 0, \cdots, k$, we define
\be\label{eq:CHLgamma}
C(H,L;\{\gamma^j\}_{0 \leq j\leq k}): = \max_{0 \leq j \leq k} \|p \circ \gamma^j\|_{C^0}
\ee
Proof of the following proposition is a consequence of  the strong maximum principle
based on the combination of the following
\begin{enumerate}
\item $\rho = r \circ u = e^s \circ u$ with $r = |p|_h$,
\item the conormal bundle property of $L$ and
\item the special form of the Hamiltonian $H = \frac{1}{2} r^2$, which is a radial function.
\end{enumerate}

We refer to \cite[Proposition 5.3]{BKO} for the proof.

\begin{prop}[Proposition 5.3 \cite{BKO}]\label{prop:C0vertical}
 Let $(\gamma^0;\vec{\gamma})$ be a $(k+1)$-tuple of Hamiltonian chords in $\frak X_i$. Then
\be\label{eq:bound-z}
\max_{z \in \Sigma}|p\circ u(z)| \leq C(H,L;(\gamma^0;\vec{\gamma}))
\ee
for any solution $u:\Sigma\to T^*N$ of \eqref{eq:duXHJ}.
\end{prop}

\subsection{Horizontal $C^0$ estimates}
\label{subsec:z-coordinate}
Because the base {$N$} is non-compact, we also need to
study the horizontal behavior of solutions $u$ of \eqref{eq:duXHJ}
for the study of compactness property of its moduli space.

We recall from \cite{BKO} that the way how Proposition \ref{prop:C0vertical} was proved is
to exploit the transformation \eqref{eq:utov} and the autonomous equation
\eqref{eq:duXHJ} for $v$. From now on, we will work with $v$ instead of $u$
in the rest of the paper, unless otherwise said.
By considering the lift of $v$ to a map, still denoted by $v$,
$(\Sigma, \del \Sigma) \to (\H^3, \widetilde L)$, it is
enough to consider maps to $\H^3$ whose study is now in order.

Now we consider $\rho=z^{-1}\circ v:\Sigma\to\R$ and its (classical) Laplacian $\Delta\rho$, i.e.,
$$
\Delta \rho = \frac{\del^2 \rho}{\del \tau^2} + \frac{\del^2 \rho}{\del t^2}
$$
in terms of the flat coordinates $(\tau,t)$ of $\Sigma$.
Recall $* d\rho = d\rho \circ j$. Therefore the geometric Laplacian $-\Delta \rho$ satisfies
$$
-\Delta \rho \, dA = \delta (d\rho) = - * d* d\rho = -* d(d\rho \circ j)
$$
and so $d(d\rho \circ j) = \Delta \rho \, dA$ with $dA = dA_h$ is the area form
associated to the metric $h$.

\begin{prop}\label{prop:Deltarho}
We have
$$
\Delta \rho \, dA \geq d\rho \wedge (\beta-v^*\theta)
$$
\end{prop}
\begin{proof}
We recall the identity
\begin{align*}
V^3 & =  dz^{-1}\circ J\\
X_H & =  z^2(p_x H_1+p_y H_2+p_z H_3)\\
JX_H & = -p_x V_1-p_yV_2-p_zV_3.
\end{align*}
Using \eqref{eq:duXHJ}, we compute
\begin{align*}
-d\rho \circ j&= - d(z^{-1}) \circ dv\circ j \\
& = -d(z^{-1})(Jdu + \beta\circ j \cdot X_H(v) - \beta\cdot JX_H(v))\\
& = -V^3 \circ dv - \beta\circ j\cdot d(z^{-1})(X_H) \\
&=  -v^*V^3+p_z(v)\cdot\beta\circ j.
\end{align*}
The co-closed property of $\beta$, i.e., $d(\beta \circ j) = 0$ implies
$$
dd^j\rho =-v^*dV^3+d(p_z(v))\wedge\beta\circ j.
$$
Using the formula $V^3= dp_z + z^{-1} \theta$, we compute
$$
-dV^3 = -d(z^{-1}) \wedge \theta + z^{-1}\omega.
$$
Therefore
\be\label{Eq:v*dV3}
-v^*dV^3 = \rho v^*\omega - d\rho \wedge v^*\theta.
\ee
Next we note
\[
d(p_z(v)) \wedge \beta \circ j = -d(p_z(v))\circ j \wedge \beta.\]
Then using \eqref{eq:duXHJ}
on $\Sigma$, we
compute
\begin{align*}
d(p_z(v))\circ j & =  dp_z \circ dv \circ j\\
& = dp_z (J dv + \beta \circ j\cdot X_H - \beta\cdot JX_H)\\
& = dp_z(Jdu) + dp_z(X_H(v)) \beta \circ j - dp_z(JX_H(v)) \beta.
\end{align*}
But using $\theta \circ J = -d H$, we have
$$
dp_z(Jdu) = (dp_z \circ J)(dv) = (-d(z^{-1}) - z^{-1} \theta \circ J)(dv) =   d\rho - \rho v^*dH
$$
 and
$$
dp_z(X_H(v)) = -z(v)(p_x^2 + p_y^2 + p_z^2)(v) = - 2 \rho H(v).
$$
Combining the above, we have derived
\begin{align*}
\Delta \rho \, dA = -d(d^\rho \circ j) & = \rho v^*\omega -d\rho\wedge v^*\theta + d\rho\wedge \beta
- \rho v^*dH \wedge \beta + 2\rho H(v) (\beta \circ j) \wedge \beta \\
& =  \rho( v^*\omega - v^*dH \wedge \beta) + 2\rho H(v)\cdot (\beta \circ j) \wedge \beta
+ d\rho \wedge (-v^*\theta + \beta).
\end{align*}
We note that the second form is clearly nonnegative since
$$
(\beta\circ j)\wedge \beta=(\beta_\tau^2+\beta_t^2)ds\wedge dt
$$
We now prove the first form is also nonnegative.
\begin{lem}\label{lem:positive}
$$
(v^*\omega - v^*dH \wedge \beta)\left(\frac{\partial}{\partial \tau}, \frac{\partial}{\partial t}\right)
= \left|\frac{\partial v}{\partial \tau} - \beta_\tau X_H(v)\right|^2 \geq 0.
$$
\end{lem}
\begin{proof}
We evaluate the form against the pair $\big(\frac{\partial}{\partial \tau}, \frac{\partial}{\partial t}\big)$.
A straightforward calculation using \eqref{eq:duXHJ} then gives rise to the inequality.
\end{proof}
Therefore combining the above, we have finished the proof of
$$
\Delta \rho\, dA \geq d\rho \wedge (\beta-v^*\theta).
$$
\end{proof}

Now we fix an exhaustion sequence of $N = M \setminus K$
\be\label{eq:exhaustion}
N_1 \subset N_2 \subset \cdots \subset N_i \subset \cdots
\ee
with $\del N_i = T_i$ a horo-torus.

We are now ready to establish the following uniform horizontal $C^0$ estimates.

\begin{thm}\label{thm:z-coord} Let $L = \nu^*T$ be a horo-torus and $\ell > 0$ be given real
number. Suppose $\{\gamma^a\}_{a=0}^k$ is a $k$-tuple of Hamiltonian chords of $L$ with
$$
-\CA(\gamma^a) = E(c^a) < \frac{\ell^2}{2}.
$$
Then for any solution $u$ of \eqref{eq:duXHJ}, there exists constant $j = j(T,\ell)$ independent
of $u$ depending only on $L$ and $\ell$ such that
\be\label{eq:z-bound}
\Image \pi \circ u \subset N_j
\ee
\end{thm}
\begin{proof} Consider the lifts $\widetilde L, \, \widetilde \gamma^a$ of $L$ and of $\gamma^a$
respectively.
Using the inequality $\Delta \rho \, dA \geq d\rho \wedge (\beta-v^*\theta)$
we will apply the maximum principle for the function $\rho$. First of all this inequality
enables us to apply the interior maximum principle and so its supremum must occur
either on $\del \Sigma$ or along the asymptotic chords $\gamma^a$ for some $a = 0, \ldots, k$.

We introduce
\be\label{eq:C}
C: = \max \left\{z^{-1}(x) \Big\vert x \in \widetilde L \cup \bigcup_{a =0}^k {\widetilde \gamma^a}\right\}.
\ee
Along the boundary $\del \Sigma$, we recall $\beta|_{\del \Sigma} = 0$.
Furthermore we have
$$
v^*\theta(\frac{\del}{\del \tau}) = \theta\left(\frac{\del v}{\del \tau}\right)
$$
which vanishes because $\frac{\del v}{\del \tau}$ is tangent to $\nu^*T$ and $\theta|_{\nu^*T} \equiv 0$.
Therefore we can apply the strong maximum principle and so the maximum $\rho$ cannot be
achieved on $\del \Sigma$ either. Therefore the supremum of $\rho$ must be
achieved at some point of $\cup_{a=0}^k \widetilde \gamma^a$. This proves $\rho \leq C$.

On the other hand, we also derive the uniform upper bound of $C$ from Lemma \ref{lem:length}
$$
C \leq \max_{t \in [0,1]} |f(t)| \leq a_0 \cosh \ell + |b_0| \sinh \ell
$$
where $a_0$ depends only on $L$ and $|b_0|$ depends only on $(T,\ell)$.
By noting $\pi \circ u = \pi \circ v$ and translating this bound on $\widetilde u$ to that of $u$,
we have proved that there exists $j = j(T,\ell) > i$ such that
$$
\Image \pi \circ u = \Image \pi \circ v \subset N_j
$$
for all finite energy solution $u$ with $\{\gamma^a\}_{a=0}^k$ as its
asymptotic chords and with boundary condition on $L$.
\end{proof}

It is easy to see that all these $C^0$ estimates can be established for the Lagrangian boundary
conditions given by the $k+1$ tuple of admissible test Lagrangians
$$
(L^0, L^1, \cdots, L^k).
$$
(See Remark \ref{rem:length} of this paper.)

\section{Formality of $A_\infty$ algebra associated to hyperbolic knot}

 The $C^0$ estimates established in the previous section
enables us to directly construct a version of wrapped Fukaya category,
denoted by $\CW\CF(M\setminus K, H_h)$ without taking a cylindrical adjustment of $h$
unlike in \cite{BKO}.

\subsection{$A_\infty$ algebra associated to hyperbolic knot}

Let $L = \nu^*T$ as before and $H = H_h$ be the kinetic energy Hamiltonian associated to
the hyperbolic metric $h$. We now consider the conormal $\nu^*T$ of the horo-torus $T$ in $M \setminus K$
as an object in this category.

Then the definition in Section \ref{sec:knot-algebra} applied to the metric $h$ instead of $g_0$ associates an $A_\infty$ algebra
$$
CW_h(T, M\setminus K) := C^*(T) \oplus \Z\langle\mathfrak X_{< - \epsilon_0}(H_h;\nu^*T,\nu^*T)\rangle.
$$
We take the perturbed conormal $L = \nu^*_{k\rho}T$ and denote
$$
CW^d(L;H_h)=\bigoplus_{x \in \Chord^d(L;H_h)}\Z \cdot x,
$$
Here the grading $d$ is given by the grading of the Hamiltonian chords $|x|$. We denote its wrapped
Floer cohomology by
\be\label{eq:HWHh}
HW^d (L;H_h).
\ee

We can also define the reduced Floer chain complex as in Section \ref{sec:knot-algebra}, which we denote
by $\widetilde{CW}(L;H_h)$.
This is the complex generated by the set $\Chord^d_{< 0}(\nu^*T;H_h)$ consisting of non-constant Hamiltonian chords.
With this mentioned, we will directly work with $\nu^*T$ without taking its perturbation.

In this subsection, we establish a formality result for the complex $\widetilde{CW}(\nu^*T;H_h)$.

For given asymptotic data
\begin{align*}
\begin{cases}
\mathbf x=x^1\otimes\cdots\otimes x^k \in CW(\nu^*T;H_h)^{\otimes k}\\
x^0\in CW(L;H_h),
\end{cases}
\end{align*}
consider the moduli space $\mathcal M^{k+1}(x^0;\mathbf x)$ of maps
\[
u:(\Sigma,\partial\Sigma)\to(T^*N,\nu^*T)
\]
satisfying the condition \eqref{eq:duXHJ}.

Then the map $\widetilde{\mathfrak m}^k : \widetilde{CW}(\nu^*T;H_h)^{\otimes k}\to \widetilde{CW}(\nu^*T;H_h)[2-k]$ is defined by
\begin{align*}
\widetilde{\mathfrak m}^k(\mathbf x)=\sum_{x^0}|\mathcal M^{k+1}(x^0;\mathbf x)| \cdot x^0,
\end{align*}
where the sum runs over $x^0\in\Chord(H;\nu^*T)$ satisfying
\begin{align*}
|x^0|=\sum_{j=1}^{k}|x^j|+2-k,
\end{align*}
and $|\mathcal M^{k+1}(x^0;\mathbf x)|$ denotes the algebraic count of points in the oriented compact
$0$-dimensional manifold $\mathcal M^{k+1}(x^0;\mathbf x)$.
As usual, $[d]$ means the grading shifting of a graded module {\em down} by $d\in\Z$.

Now the $C^0$ estimates established in Section \ref{sec:perturbedCR} enables us to study
compactified moduli space $\mathcal M^{k+1}(x^0;\mathbf x)$ and the standard Fredholm theory
proves that the compactified moduli is a smooth manifold with boundary and corners
(after making a $C^\infty$-small perturbation of $J_h$, if needed).

Then Theorem \ref{thm:index-nullity} implies that all the degrees of generators $x^i$ are 0 and
so the above dimension formula reduces to $2-k$. Since all the nontrivial matrix coefficients
are given by zero dimensional moduli space, only the case of $k=2$, i.e. $\widetilde{\mathfrak m}^k = 0$ for all
$k \neq 2$.

Combining the above discussion, we have proved the following theorem.
{By taking an arbitrarily small tubular neighborhood $N(K)$ of $K$ and setting $T = \del N(K)$,
we may regard it as  the `ideal boundary' that appears in the title of the present article.

\begin{thm}\label{thm:formal} Suppose that $K \subset M$ is a hyperbolic knot and $h$ be
its associated hyperbolic metric $h$.
Let $\nu^*T$ be the conormal of any horo-torus $T \subset M \setminus K$.
The $A_\infty$ structure of $(\widetilde{CW}(\nu^*T;H_h), \{\widetilde{\mathfrak m}^k\}_{k=1}^\infty)$ is reduced to
an associative algebra whose product is given by $\mathfrak m^2$, i.e., it satisfies $\frak m^k =0$
unless $k =2$. In particular $\widetilde{HW}^d(\nu^*T;H_h) = 0$ for all $d > 0$ and
$\widetilde{HW}^0(\nu^*T;H_h)$ is a free abelian group generated by $\mathscr G_{M\setminus K}$.
In particular, the rank of $\widetilde{HW}^0(\nu^*T;H_h)$ is infinity.
\end{thm}

We would like to describe the product structure of this algebra in terms of the
hyperbolic geometry of the complement ${M} \setminus K$. In Section \ref{sec:reduction-2dim} of
Appendix,  we will prove some reduction theorem as the first step towards this goal.
In relation to this goal, Conjecture \ref{conj:main} is crucial for the full study of which
will be postponed elsewhere.

\subsection{A step towards comparison with Knot Floer Algebra}

Let $L = \nu^*_{k,\rho}T$ be the perturbed conormal given in Section \ref{sec:knot-algebra}.
We assume
\be\label{eq:suppkN0}
\supp k \subset N_0 \subset N
\ee
so that $L = \nu^*_{k,\rho}T$ for the region given by $(p,q)$ satisfying
$$
q \in N \setminus N_0, \quad |p|\geq 3 \|dk\|_{C^0}.
$$

As the first step towards a comparison result between the wrapped Floer homology $HW^d(L;H_h)$
and the Knot Floer algebra
$HW(\del_\infty(M \setminus K))$, we consider a sequence of cylindrical adjustments $h_j$ of the given metric $h$ associated to the exhaustion \eqref{eq:exhaustion}: $h_j$ is defined by
$$
h_j = \begin{cases} h \quad & \text{\rm on } N_j' \\
da^2 \oplus h|_{N_j} \quad & \text{\rm on } (M\setminus K) \setminus N_j
\end{cases}
$$
where $N_j'$ is another subdomain of $N_j$ such that $ \overline N_j' \subset N_j$.
(See Section \ref{sec:comparison} for precise details on this definition.)

Utilizing the $C^0$ bound given in Proposition \ref{prop:Deltarho} and similar bound
for $h_i$ obtained in \cite{BKO}, we obtain the $A_\infty$ algebras $CW(\nu^*T, h_i)$.
Then we consider any pair $i, \, j$ with $i\leq j$. Note the
\be\label{eq:Hs-monotone}
H_{h_j} \geq H_{h_i}
\ee
for $j \geq i$ since $h_j \leq  h_i$.
We consider a homotopy $s \mapsto  H^s$ associated to the metrics
$$
h^s_{ij} = (1-s) h_i + s h_j
$$
with $H^s = H_{h^s_{ij}}$. By the monotonicity \eqref{eq:Hs-monotone}, we have
an $A_\infty$ homomorphism
$$
\iota_{ij}: CW^d(L, H_{h_i}) \to CW^d(L, H_{h_j})
$$
for $i \leq j$ which induces a (homotopy) direct system
\be\label{eq:limCWhi}
CW^d(L, H_{h_1}) \to CW^d(L, H_{h_2}) \to \cdots \to CW^d(L, H_{h_i}) \to  \cdots.
\ee
(See \cite[Section 7.2.12]{fooo:book}, \cite{seidel:biased} for a detailed explanation of such
a procedure of taking the limit $A_\infty$ structures and homomorphisms.)

Next we prove the following existence result on the continuation map.
We would like to emphasize that while the vertical $C^0$-estimate
for the homotopy of Hamiltonians in the direction of monotonically
increasing direction, e.g., for the homotopy $s \mapsto (1-s)H_{h_i} + sH_h$
can be established and well-known (see \cite{floer-hofer}, \cite{abou-seidel} for example),
the horizontal $C^0$-estimate is new.

\begin{rem}
The $C^0$-estimates obtained in Subsection \ref{subsec:z-coordinate}
and in \cite[Section 11]{BKO} dealt with the autonomous cases for the hyperbolic metric
and for the cylindrical metric respectively. However none of them apply to the current case
since we need to establish the horizontal $C^0$-estimate
for the \emph{non-autonomous} equation. The proof will clearly exhibits that existence of
such a continuation morphism strongly relies on the direction of the homotopy in terms of
the relevant Hamiltonians. This horizontal $C^0$-bound will follows from the
elliptic estimates (e.g., \cite[Theorem 3.7]{gil-trud}) and the vertical $C^0$-bound
established in Proposition \ref{prop:C0vertical}.
\end{rem}

\begin{prop}\label{prop:continuation} There exists a natural monotonicity $A_\infty$ morphism
$$
\iota_{jh}: CW^d(L, H_{h_j}) \to CW^d(L,H_{h})
$$
induced by the linear homotopy $s \mapsto (1-s)h_j + s h$.
\end{prop}
\begin{proof} We consider the non-autonomous version of \eqref{eq:dvXHJ} associated to
by the linear homotopy $s \mapsto (1-s)h_j + s h$ which becomes
\be\label{eq:dvXHJ-non}
\begin{cases}
\frac{\del v}{\del \tau} + J^{\chi} \left(\frac{\del v}{\del t} - X_{H^{\chi}}(v)\right) = 0\\
\text{$v(z)\in L$, for $z\in\partial \Sigma$} \\
v\circ \epsilon^j(-\infty,t)= x^j(t), \text{ for }j=1,\dots,k.\\
v\circ \epsilon^0(\infty,t)= x^0(t).
\end{cases}
\ee
Again it remains to ensure the $C^0$-estimate hold. As mentioned above, the vertical $C^0$-estimate
is standard and so omitted. We will focus on the horizontal $C^0$-estimate.

First, we mention that by \eqref{eq:suppkN0} and the remark right afterwards we may just work with
the conormal $\nu^*T$ instead of $L$.
Then, to establish the horizontal $C^0$-estimate, we will follow the approach taken in \cite{BKO} by
decomposing the equation into the vertical and the horizontal components in terms of
the cylindrical coordinates $(a,x,y)$ and its associated canonical coordinates
$(a,x,y,p_a, p_x,p_y)$ on $T^*(M\setminus K$) on $N_i^{\text{\rm end}} \cong [0,\infty) \times T$.

Recall from \eqref{eq:bdelta}
that the function $-a$ is the lift of the Busemann function on $N(K) \setminus K$ with $z=e^a$.
The hyperbolic metrics $h$ on $N'(K)$ is written in terms of $(z,x,y)$ and $(a,x,y)$ as follows,
\be\label{eq:hypmetric}
\begin{aligned}
h &~=~ z^{-2}(dz^2 +dx^2 +dy^2 )    \\
&~=~  da^2 + e^{-2a}dx^2 + e^{-2a} dy^2
\end{aligned}
\ee
From this, we consider the exhaustion $N_i$ such that $N\setminus N_i = a^{-1}((i,\infty))$ and
their cylindrical adjustments $h_i$'s for $i\geq 1$ of $h$
\be\label{eq:hypcyl}
h_i =
\begin{cases} h \quad & \text{on } N_{i-1/2}\\
da^2+ e^{-2i}(dx^2+ dy^2) & \text{on } N\setminus N_{i}.
\end{cases}
\ee
Then
$$
H_h = \frac12\left(p_a^2 + e^{2a}(p_x^2 + p_y^2)\right), \quad H_{h_i} = \frac12\left(p_a^2 + e^{2i}(p_x^2 + p_y^2)\right).
$$
Therefore we have
\beastar
\pi_{T^*[0,\infty)}(X_{H^{\chi(\tau)}}(a,x,y)) & = & p_a \frac{\del}{\del a}, \\
\pi_{T^*[0,\infty)}(J^{\chi(\tau)} X_{H^{\chi(\tau)}}(a,x,y))
& = &\left( e^{-2i} + \chi(\tau)(e^{-2a} - e^{-2i})\right) \frac{\del}{\del p_a}.
\eeastar
Recalling $\beta = dt$, we compute the $(a,p_a)$-component of \eqref{eq:dvXHJ-non}
\begin{align}\label{eq:a-component}
\begin{dcases}
\frac{\del a(v)}{\del \tau} - \frac{\del p_a(v)}{\del t} = 0 \\
\frac{\del p_a(v)}{\del \tau} + \frac{\del a(v)}{\del t}
- \left( e^{-2i} + \chi(\tau)(e^{-2a} - e^{-2i})\right) p_a(v)= 0
\end{dcases}
\end{align}

A straightforward calculation using these identities proves

\begin{lem}
Let $v$ be a solution of \eqref{eq:a-component}. Then
\be\label{eq:Deltaau}
\Delta (a(v))+ 2 \chi(\tau) e^{-2a} p_a(v)\frac{\del a}{\del t} = \left(e^{-2i} + \chi(\tau)(e^{-2a(v)} - e^{-2i})\right) \frac{\del p_a}{\del t}
\ee
on $a^{-1}((-\infty,a_0]) \subset \Sigma$ for any solution $v$ of \eqref{eq:dvXHJ-non}.
\end{lem}
\begin{proof} By differentiating the first equation of \eqref{eq:a-component} by $\frac{\del}{\del \tau}$ and
the second by $\frac{\del}{\del t}$ and adding them up, we get
\beastar
\Delta(a(v)) & = & \frac{\del}{\del t}\left(\left( e^{-2i} + \chi(\tau)(e^{-2a(v)} - e^{-2i})\right) p_a(v)\right)\\
& = & -2 \chi(\tau)  e^{-2a(v)} p_a(v) \frac{\del a(v)}{\del t} + \left( e^{-2i} + \chi(\tau)(e^{-2a(v)} - e^{-2i})\right)
\frac{\del p_a(v)}{\del t}
\eeastar
\end{proof}
Now we define a function $f: a^{-1}((-\infty,a_0])\cap \Sigma \to \R$ by
$$
f(\tau,t) = \left(e^{-2i} + \chi(\tau)(e^{-2a} - e^{-2i})\right) \frac{\del p_a}{\del t},
$$
and rewrite the equation \eqref{eq:Deltaau} into
\be\label{eq:Lvf}
Lv = f, \, L = \Delta + 2 \chi(\tau) e^{-2a(v)}  p_a(v)\frac{\del}{\del t}.
\ee
We get
$$
\|\chi(\tau) e^{-2a(v)}  p_a(v)\|_{C^0} \leq \|\chi(\tau) e^{-2a(v)}\|_{C^0} \| p_a(v)\|_{C^0} \leq e^{-2a_0} \|p_a(v)\|_{C^0}
$$
on $a^{-1}((-\infty,a_0]) \subset \Sigma$.  We also note
$$
\|f\|_{C^0} \leq  \left|e^{-2i} + \chi(\tau)(e^{-2a} - e^{-2i})\right\|_{C^0}
\left\|\frac{\del p_a}{\del t}\right\|_{C^0}.
$$
Here we derive a bound for $\left\|\frac{\del p_a}{\del t}\right\|_{C^0}$ which is a consequence of
the $C^1$-bound of $v$,
which follows from the energy bound and the vertical bound from Proposition \ref{prop:C0vertical}.
Therefore $L$ is a uniformly elliptic second-order
partial differential operator and there exists a constant $C_0> 0$ independent of $v$ such that
$$
\|f\|_{C^0} \leq C_0
$$
where $C_0$ depends only on $\{N_i\}$. Then by applying the classical elliptic estimate
(see \cite[Theorem 3.7]{gil-trud} for example) to $\pm a(v)$ on each connected component of
$a^{-1}((-\infty,a_0])\cap \Sigma$ separately, we prove
$$
\|a(v)\|_{C^0} \leq \|a(v)|_{\del \Sigma}\|_{C^0} + C_1 \|f\|_{C^0} \leq a_0 + C_1C_0
$$
applied on the domain $\Omega = v^{-1}([a_0,\infty)) \subset \Sigma$.

Therefore we can find $n=n(\ell)$ such that
the image of $u$ is contained in $W_{n(\ell)}$ since
its asymptotic chords are also assumed to be contained in $W_\ell$.

Once we have this uniform  $C^0$-estimate established, construction of
the continuation map proceeds as usual. This finishes the proof of the proposition.

\end{proof}

The following lemma then immediately follows from the standard construction of the Floer
theory once we have Proposition \ref{prop:continuation} in our disposal.

\begin{lem}
The homomorphism $\iota_{jh}: CW^d(L, H_{h_j}) \to CW^d(L, H_h)$
is compatible with the above direct system, i.e., that satisfies
$$
\iota_{ih} \sim \iota_{ij} \circ \iota_{jh}
$$
for all $i \leq j$, where $\sim$ denotes `being homotopic relative to the ends'.
\end{lem}

This induces a natural $A_\infty$ map
\be\label{eq:CWhih}
\iota_\infty: \lim_{\longrightarrow} CW^d(L, H_{h_j}) \to CW^d(L,H_h)
\ee
which in turn induces a homomorphism
$$
(\iota_\infty)_*: \lim_{\longrightarrow} HW^d(L, H_{h_j}) \to HW^d(L,H_h).
$$
In fact we have the following extension of Theorem \ref{thm:index-nullity} to the
cylindrical adjustment $h_i$ of $h$.

\begin{prop}\label{prop:vanishing-hi}
	Let $T \subset M \setminus K$ as above. Then we can find $h_i$ so that it has
	non-positive curvature, i.e., all sectional curvature $K(X,Y): = \langle R(X,Y)Y,X\rangle \leq 0$.
	In particular for any geodesic cord $c \in \text{\rm Cord}_{h_i}(T)$, both Morse index and nullity of $c$
	vanish.
\end{prop}
\begin{proof}
Let us consider a cylindrical adjustment $h_i$ of  $h$ of the form
\ben
h_i =  da^2 + \rho_i^2(dx^2 +  dy^2)
\een
with an interpolated function
\be\label{eq:interpolatedrho}
\rho_i(a) :=
\begin{cases}
	e^{-a} \quad & \text{ for } a <i-\frac{1}{2}\\
	e^{-i} & \text{ for } a \geq i.
\end{cases}
\ee

A straightforward calculation gives rise to the following covariant derivatives of the Levi-Civita connection of $h_i$:
\ben
\begin{aligned}
	&\nabla_{\del_a}\del_a=\nabla_{\del_x}\del_y=\nabla_{\del_y}\del_x=0, \\
	&\nabla_{\del_x}\del_x =\nabla_{\del_y}\del_y= -\rho_i \rho_i' \del_a, \\
	&\nabla_{\del_a}\del_* =\nabla_{\del_*}\del_a= \frac{\rho_i'}{\rho_i}\del_* ~~\text{ for }~ *=x,y.
\end{aligned}
\een
By direct computation, we obtain
\be\label{eq:curvature h_i}
K(\del_x,\del_y)=-\left(\dfrac{\rho'_i}{\rho_i}\right)^2, \quad K(\del_a,\del_x)=K(\del_a,\del_y)=-\dfrac{\rho''_i}{\rho_i}. 	
\ee

Now we are going to construct $\rho_i$ satisfying (\ref{eq:interpolatedrho}) explicitly.
First, let us consider a smooth cut-off function
$$
\tau_{0,1}(t):= \begin{cases} 1  \quad & \text{for } t \leq 0 \\
\frac{e^{1/t}}{e^{{1}/(1-t)}+e^{1/t}} \quad & \text{for } 0 < t < 1\\
0 \quad & \text{for } t \geq 1
\end{cases}
$$
and define $\tau_{a,b}:=\tau_{0,1}(\frac{x-a}{b-a})$ for a given $a<b$. Then
$$
\tau_{a,b}(t)=\begin{cases}
1 \quad\text{ for } t \leq a \quad & \text{for } 0 < t < 1\\
0 \quad\text{ for } t \geq b.
\end{cases}
$$
In particular,  we have the following inequality when $0<b-a<1$,
\be\label{eq:intercondition}
0 \geq \tau_{a,b}^2-\tau_{a,b} \geq \tau_{a,b}'.
\ee
Next we consider a smooth function
$$E(t)= 1+e^{-1/t} \quad\text{ for } t>0$$ and check
\be\label{eq:inequalityE}
E^2 \geq E \geq E' \quad\text{ for } t>0.
\ee
Now we define a smooth function
$$
A_{i,\varepsilon}(t):=\begin{cases}
1 &\quad\text{ for } t\leq i-\frac{1}{2}, \\
E(t-i+\frac{1}{2}) &\quad\text{ for }  i-\frac{1}{2} < t \leq i-\varepsilon,\\
E(t-i+\frac{1}{2})\tau_{i-\varepsilon,i}(t), &\quad\text{ for }    i-\varepsilon < t.
\end{cases}
$$
Then we can find $0<\varepsilon_0 <1$ such that  $\int_0^i A_{i,\varepsilon_0}(t) dt = i$ because $\int_0^i A_{i,1}(t) dt <i$  and $\int_0^i A_{i,0}(t) dt >i$.

Using this function, we define a smooth function
$$
B_i(a) := \int_0^a A_{i,\varepsilon_0} dt,
$$
which satisfies
$$
B_i(a)= \begin{cases}
a \quad\text{ for } a\leq i-\frac{1}{2},\\
i \quad\text{ for } a \geq i.
\end{cases}
$$
Finally we obtain an interpolated function
$$
\rho_i(a) := e^{-B_i(a)}
$$
satisfying (\ref{eq:interpolatedrho}) where the second derivative is
$$
\rho''_i(a)= (B'_i(a)^2 - B''_i(a) )e^{-B_i(a)}.
$$
We can check
$B'^2_i-B''_i=A_{i,\varepsilon_0}^2 - A'_{i,\varepsilon_0} \geq 0$ by the inequality (\ref{eq:intercondition}) and (\ref{eq:inequalityE}). The non-negativity of the second derivative of $\rho_i$ completes the proof with (\ref{eq:curvature h_i}).
\end{proof}

One immediate consequence of this proposition is the following formality
\begin{cor} The boundary map $m^1: CW(\nu^*T,H_{h_i}) \to CW(\nu^*T,H_{h_i})$
is zero. Therefore $CW(\nu^*T,H_{h_i}) = \ker m^1$ and  the natural map
$$
CW(\nu^*T,H_{h_i}) \to HW(\nu^*T,H_{h_i})
$$
is an isomorphism.
\end{cor}

Therefore we will freely regard the chain map \eqref{eq:CWhih} as its homological version for the discussion below.

\begin{thm}\label{thm:inftyrank} The homomorphism $(\iota_\infty)_*$
is an isomorphism for all integer $d \geq 0$.
\end{thm}
\begin{proof} We first prove surjectivity of the map. Let $a \in HW^d(L, H_h)$ be given.
By the formality of $CF(\nu^*T;H_h)$, there exists a \emph{unique cycle}
$$
\alpha = \sum_{m=1}^k  n_m \gamma_{\delta_m} \in CF(\nu^*T;H_h)
$$
representing $a$ with $\delta_m \in \mathcal G_{M\setminus K}$. Here $\gamma_{\delta_m} \in \text{Chord}(\nu^*T)$
is the Hamiltonian chord of $\nu^*T$ associated to the tame geodesic $\delta_m$ via
the one-one correspondence established by Proposition \ref{prop:cordG} and Lemma \ref{lem:cord_chord}.

We denote by $c_m = (\delta_m)_{(T)}$
the unique geodesic cord of $T$ corresponding to $\delta_m$ given by Proposition \ref{prop:cordG}.
Denote
$$
N_0 = \max_{m=1}^k \{\leng(c_m)\}.
$$
Then it follows from the $C^0$ estimate, Theorem \ref{thm:z-coord}, that there exists a sufficiently large
$i_0 = i_0(N_0)$ such that
$$
\supp c_m \subset N_{i_0}, \quad m=1, \ldots, k
$$
and so the chain
$$
\sum_{m=1}^k n_m \gamma_{c_m}
$$
becomes a cycle in $CF(\nu^*T;H_{h_i})$ for all $i \geq i_0$.
We denote the resulting cycle by $\alpha_i$ for each $i \geq i_0$.
Furthermore since $h_i \equiv h$ on $N_{i_0}$, we also
have $\iota_{ih}(\alpha_i) = \alpha$, for all $i \geq i_0$ by the
Floer theory construction of $\iota_{ih}$ and the definition of $\alpha_i$.
Then obviously they satisfy the compatibility relation
$$
\alpha_{i+1} = \iota_{i(i+1)}(\alpha_i)
$$
for all $i \geq i_0$. We extend the sequence $\alpha_i$ all the way up to $i =1$
by choosing any cycle $\alpha_i$ that satisfies
$$
\alpha_{i_0} = \iota_{ii_0}(\alpha_i)
$$
for $1 \leq i \leq i_0$. Such a cycle $\alpha_i$ always exists because
each $\iota_{ji_0}$ is a quasi-isomorphism by the following general
lemma (or rather from its proof in Appendix).

\begin{lem}\label{lem:twolimits} Let $g, \, g'$ be any metric on $M \setminus K$ with cylindrical ends.
Then we have a quasi-isomorphism
$$
CW(L,H_g) \cong CW(L, H_{g'})
$$
\end{lem}

Furthermore such cycle $\alpha_i$ is unique by the formality given in Proposition \ref{prop:vanishing-hi}.

Then by construction the sequence $\{\alpha_i\}_{i=1}^\infty$
is a compatible sequence of cycles and hence defines an element in
$\displaystyle{\lim_{\longrightarrow} HW^d(L, H_{h_j})}$. By construction, we have
$$
(\iota_{\infty})_*([\alpha_1 \to \alpha_2 \to \cdots]) = a
$$
which proves surjectivity.

For the proof of injectivity, we recall the energy identity for the continuation equation:
\bea\label{eq:lemma4.1}
\int \left|\frac{\del u}{\del \tau}\right|_{J_{\chi(\tau)}}^2 \, dt\, d\tau
& = & \mathcal A_{H^+}(z^+) - \mathcal A_{H^-}(z^-) \nonumber \\
&{}& - \quad
\int_{-\infty}^\infty \chi'(\tau) \left(\int_0^1
\frac{\del H_s}{\del s}\Big|_{s = \chi(\tau)}(u(\tau,t))dt\right)  d\tau.
\eea
(See \cite[(7.15)]{BKO} for the energy formula
for the non-autonomous equation in the convention used in the present paper.) Due to the
fact $\frac{\del H_s}{\del s} \geq 0$, we have
$$
\mathcal A_{H^+}(\gamma_{c^+}) \geq \int \left|\frac{\del u}{\del \tau}\right|_{J_{\chi(\tau)}}^2 \, dt\, d\tau
+ \mathcal A_{H^-}(\gamma_{c^-})
$$
and so $\mathcal A_{H^+}(\gamma_{c^+}) \geq \mathcal A_{H^-}(\gamma_{c^-})$. The last inequality is equivalent to
\be\label{eq:Ec+-}
E(c^+) \leq E(c^-)
\ee
by \eqref{eq:CA-E}. This bound for $\mathcal A_{H^+}(\gamma_{c^+})$ in particular implies that there are
only finitely many such $c^+$ for each given $c^-$.

With this preparation, we proceed with the proof of injectivity.
Suppose $(\iota_{\infty})_*(b_\infty) = 0$.
Represent $\displaystyle{b_\infty \in \lim_{\longrightarrow} HW^d(L, H_{h_j})}$ by a sequence of $b_j \in  HW^d(L, H_{h_j})$ so that
$$
b_\infty = [b_1 \to b_2 \to b_3 \to \cdots \to b_j \to \cdots].
$$
Using Proposition \ref{prop:vanishing-hi}, we will also regard $b_i$ as a chain (cycle)
abusing notation. We denote the unique cycle representing $b_i$ by $\beta_i$.

Then by the compatibility of the sequence and the formality of $CW^d(L,H_h)$, we have
$$
0 = \iota_{1h}(\beta_1) = \iota_{2h}(\beta_2) = \cdots
$$
as a chain.
Out of this, we will derive $\beta_i = 0$ for all sufficiently large $i$.

We first note that each geodesic  $c_\ell$ of $H_h$ is also a geodesic cord of $h_j$
whenever $\supp c_\ell \subset N_{j-1}$ since $h_j = h$ on $N_j$.
We denote by $I_j$ the finite subset of $\Z_+$ for such $\ell$'s
and by $\delta_{c_\ell}'$ the Hamiltonian chords of $H_j$ corresponding to
such geodesic cord of $h_j$. In particular we mention that the set
$
\{\delta_{c_\ell}'\}_{\ell \in I_j}
$
is linearly independent as chains in $CW(L;H_{h_j})$ for each $j \geq i_1$.

\begin{lem}\label{lem:i0hj}
Let $i_0$ be given. Express
\be\label{eq:iotai0h}
\iota_{i_0h}(\beta_{i_0}) = \sum_{\ell \in I} a_\ell\, \delta_{c_\ell} \in CW(L;H_h)
\ee
for a finite index set $I \subset \Z_+$ with $a_\ell \neq 0$ for $\ell \in I$.
Then there exists $i_1 > i_0$  such that
$$
\iota_{i_0j}(\beta_{i_0}) = \sum_{\ell \in I} a_\ell\, \delta_{c_\ell}'
$$
for all $j \geq i_1$.
\end{lem}
\begin{proof}
Since $h_{i_0}$ is cylindrical outside $N_{i_0}$, any Hamiltonian chords of $\nu^*T$ cannot
touch the cylindrical region $N \setminus N_{i_0}$. (See \cite[Proposition 4.3]{BKO}.)
Therefore we can choose a cycle $\beta_{i_0} \in CW(L,H_{h_{i_0}})$
representing $b_{i_0}$ with $\supp \beta_{i_0} \subset \Int W_{i_0}$.

On $W_{i_0}$, we have $H_h = H_{h_{i_0+1}}$. It follows from \eqref{eq:Ec+-} that for all $j \geq i_0$
we have the uniform bound
$$
E(c^+) \leq \ell(\beta_{i_0}): = \max \left\{ E(c'_\ell) \mid \beta_{i_0}
= \sum_{\ell =1}^n b_\ell \, \delta_{c_\ell'}, \, b_\ell \neq 0 \right\}
$$
for any $c^+ \in \text{\rm Cord}_{h_j}(T)$ appearing at $\tau = \infty$ for some solution $u$
of \eqref{eq:CRXH} satisfying $u(-\infty) \in \beta_{i_0}$. Since $\supp \beta_{i_0} \subset W_{i_0}$,
this energy bound for such $c^+$ forces $\supp c^+$ to be contained in $W_{i_0'}$ for a sufficiently large $i_0'$
\emph{for any $j \geq i_0'+1$}: This is because $\nu^*T \subset W_1$ and
$\text{\rm dist}(T, M \setminus N_j) \to \infty$ as $j \to \infty$.

We recall $H_h = H_{h_j}$ on $W_{i_0'}$ for any $j \geq i_0' +1$ and so
$\delta_{c_\ell}$ for $\ell \in I$ appearing in \eqref{eq:iotai0h} can be also regarded
as a cycle of $CW(L;H_{h_i})$.
Then it follows from the uniform energy bound \eqref{eq:lemma4.1} that
the image of any solution $u$ contributing to the continuation maps $\iota_{i_0h}$, $\iota_{i_0j}$
is contained in $W_{i_1}$ for sufficiently large $i_1$ depending only on
$i_0, \, i_0'$ but independent of $j \geq i_1+1$.

Since $H_h = H_{h_j}$ on $H_{i_1}$ and so the Cauchy-Riemann equations to solve for
the construction of the continuation maps $\iota_{i_0h}$, $\iota_{i_0j}$ on $W_{i_1}$ are
exactly the same equation, we obtain
$$
\iota_{i_0j}(\beta_{i_0})
= \sum_{\ell \in I} a_\ell\, \delta_{c_\ell}'
$$
by the construction of the continuation maps $\iota_{i_0h}$, $\iota_{i_0j}$
established in Proposition \ref{prop:continuation}.
This finishes the proof.
\end{proof}

Therefore using this lemma, the formality and the hypothesis $(\iota_\infty)_*(b_\infty) = 0$, we derive
$$
0 = \iota_{i_0h}(\beta_i)= \iota_{i_0i_1}(\beta_i)
$$
and so $(\iota_{i_0 i_1})_*(b_{i_0}) = 0$.
Since $(\iota_{i_1i_0})_*$ is an isomorphism by Lemma \ref{lem:twolimits}, this implies $b_{i_0} = 0$.
Since this holds for $b_{i_0}$ for all given $i_0$, this finishes the proof of injectivity. Hence we have
proved that $(\iota_\infty)_*$ is an  isomorphism.
\end{proof}

\section{Comparison with Knot Floer cohomology}
\label{sec:comparison}

In this section we would like to compare the Knot Floer algebra constructed in \cite{BKO} with
the algebra constructed in Theorem \ref{thm:inftyrank} directly on
$T^*({M} \setminus K)$ for $L = \nu^* T$ for a horo-torus $T$ using the hyperbolic metric.

\subsection{Comparison between $HW(\nu^*T;H_h)$ and $HW(\del_\infty(M\setminus K))$}

The main goal of this section is to study the relationship between the algebra
$HW(\nu^*T;H_h)$ constructed in the previous sections via the hyperbolic metric $h$
and the knot Floer algebra $HW(\del_\infty(M\setminus K))$.

For this purpose, we first choose the exhaustion sequence
\be\label{eq:h-Ni}
N_1 \subset N_2 \subset \cdots \subset N_i \subset \cdots
\ee
so that its boundary $\del N_i$ is a horo-torus $\del N_i= T_i$ with respect to
the given hyperbolic metric $h$ on $M \setminus K$. Then we take $W_i = T^*N_i$.

We now express the metric $h$ on $M \setminus K$ in the cylindrical representation
\be\label{eq:cusp-nbd}
N_{(0,\epsilon]} \cong [0,\infty) \times T_{\varepsilon}
\ee
with coordinates $(a, q)$ with $q \in T_{\varepsilon}$.

Then we prove the following that each $h_i$ is also a cylindrical adjustment of a smooth metric on $M$
restricted to $M \setminus K$. More precisely, we have

\begin{prop}\label{prop:hi-smoothadjust} We can construct a sequence of smooth metrics $g[i]$
defined on $M$ such that the followings hold:
\begin{enumerate}
\item For all  $i \geq 1$,
\be\label{eq:hi-gij}
h_i = g[i]_0
\ee
where $g[i]_0$ is the cylindrical adjustment outside $N_i$.
\item  For all $1 \leq i \leq k$,
\be\label{eq:gil-gkj}
g[i]_0 \geq g[k]_0.
\ee
\end{enumerate}
\end{prop}
\begin{proof}
Let $N=M\setminus K$ be a hyperbolic knot complement and $N(K)=N_\varepsilon(K)$ be a sufficient small tubular neighborhood. Let $N_\varepsilon:=N\setminus N(K)$ be the compact thick part and $N'(K)$ be the deleted neighborhood $N(K)\setminus K$. The statement of the proposition is a (uniform) pointwise statement. Therefore we will lift the relevant metrics on $N$ to its universal covering space $\H^3$.

Denote by $p:\H^3 \to
 M \setminus K$ a universal covering map and put the lifting of the horo-torus  $T=\del N(K)$ to be the horo-sphere of $\{z=1\}$. Consider a holonomy representation $\rho$ for the  hyperbolic structure of $N$.
Since $N$ is a knot complement, there is a canonical choice of two generators $m$ and $l$ of $H_1(T)$, called  $\emph{meridian}$ and  $longitude$. Note that $l$ is a closed curve parallel to $K$ and $m$ bounds a disk in the tubular neighborhood $N(K)$.

 We can choose an explicit holonomy representative $\rho$ in the conjugacy class $[\rho]$ such that $\rho(m)$ and $\rho(l)$ preserve the horosphere $\{z=1\}$ and $\rho(m)$ is a translation of only  $x$-direction, as follows.
\be\label{eq:holonomy}
\begin{aligned}
\rho(m) &= (x,y)\mapsto(x+m_1,y+m_2)  &\text{ with }\quad m_2=0 \\
\rho(l) &= (x,y)\mapsto(x+l_1,y+l_2)  &
\end{aligned}
\ee
   Here $(m_1,m_2)$ and $(l_1,l_2)$ are linearly independent vectors in the $xy$-plane of $\H^3$, which are determined by the complete hyperbolic structure of $M\setminus K$, the choice of a horo-torus $T$ and the choice of a holonomy $\rho$.\footnote{ The complex number of $\frac{1}{m_1}(l_1+  \sqrt{-1}\,l_2 )$ is a hyperbolic knot invariant, called \emph{cusp shape}. For the details, there are many texts on hyperbolic knots. For instance, see \cite[ Section 11,13,14]{Mart}. }

Now we have an explicit coordinate of $p^{-1}(N'(K))$ in $\H^3$ as follows.
\be\label{eq:coordN}
\{(z,x,y) \mid z\geq 1 \text{ and }  x= \mu m_1 + \lambda l_1,\, y= \lambda l_2  \text{ for } 0\leq \mu, \lambda <1\}
\ee
Recall that $N'(K)$ is obtained from $p^{-1}(N'(K))$ identified by the holonomies $\rho(m)$ and $\rho(l)$ and hence the coordinate of (\ref{eq:coordN}) itself can be regarded as a global coordinate of $N'(K)$.

Now we consider cylindrical adjustments. Recall the hyperbolic metrics $h$ on $N'(K)$ is written as
\be\label{eq:hypmetric}
\begin{aligned}
h &~=~ z^{-2}(dz^2 +dx^2 +dy^2 )    \\
&~=~  da^2 + e^{-2a}dx^2 + e^{-2a} dy^2
\end{aligned}
\ee
in terms of $(z,x,y)$ and $(a,x,y)$, and
 their cylindrical adjustments $h_i$'s for $i\geq 1$ of $h$ are given by
\be\label{eq:hypcyl}
h_i =
\begin{cases} h \quad & \text{on } N_{i-1/2}\\
da^2+ e^{-2i}(dx^2+ dy^2) & \text{on } N\setminus N_{i}
\end{cases}
\ee
with respect to the exhaustion $N_i$ such that $N\setminus N_i = a^{-1}((i,\infty))$.
We remark that
\be
h \leq h_i \quad \text{ for all } i\geq 1.
\ee
Now we consider a metric $g[i]$ smooth on $M$ whose cylindrical adjustment coincides with $h_i$ as follows.

\be\label{eq:cappedhyp}
g[i] = \begin{cases} h_i \quad & \text{on } N_{i}\\
e^{-2a} da^2 + e^{-2a} dx^2 + e^{-2i} dy^2 \quad & \text{on } N \setminus N_{i+1/4}
\end{cases}
\ee
\begin{lem}
	The above metric $g[i]$ extends smoothly to $N(K)$, i.e. it is a smooth metric of the ambient manifold $M$.
\end{lem}
\begin{proof}
	To consider smoothly extended metrics on $N'(K)$,
	let us look at a standard solid torus $ST\subset \R^3$,
	$$ST:=D^2\times S^1=\{( \cos\varphi, \sin\varphi,0)+ (r\cos\theta \cos\varphi , r\cos\theta \sin\varphi, r \sin \theta) \} \subset  \R^3 $$
	with $ 0\leq r\leq r_0,~ 0\leq \theta,\varphi < 2\pi$
	and  a standard polar coordinate
	as follows,
	$$ST = \{(r,\theta,\varphi) \mid 0\leq r\leq r_0,~ 0\leq \theta,\varphi < 2\pi \}.$$
Here, we also consider a deleted solid torus $ST':=ST\setminus \{r=0\}$.
	We construct an explicit diffeomorphism $\Psi$ between $(z,x,y)$ of $N'(K)$ and $(r,\theta,\varphi)$ of $ST'$,
	\be\label{eq:zxypolar}
	(z,x,y)=\Phi(r,\theta,\varphi) := \left(1/r,~\frac{m_1}{2\pi}\theta + \frac{l_1}{2\pi}\varphi ,~ \frac{l_2}{2\pi}\varphi \right)
	\ee
Next, we consider a pullback of the metric $g[i]$ by $\Phi$ on $ST'$,
$$
\Phi^*g[i] = \begin{cases} h_i \quad & \text{on } N_{i}\\
dr^2 + \dfrac{r^2 (m_1^2\, d\theta^2 + l_1^2 \, d\varphi^2)+ e^{-2i} l_2^2\, d\varphi^2 }{4\pi^2} \quad & \text{on } N \setminus N_{i+1/4}
\end{cases}
$$	
 Recall that the standard smooth metric for $ST$ is given by
 $dr^2 + r^2 d\theta^2+d\varphi^2$.
 We can directly verify that $\Phi^* g[i]$ is also smooth at $r=0$ on $ST$.	

Note that the construction of (\ref{eq:coordN}) shows that
$\Phi$ between $N(K)'$ and $ST'$ extends continuously to a homeomorphism between $N(K)$ and $ST$. Therefore the metric completion of  $N'(K)$ with $g[i]$ also recovers the original $N(K)$ and it proves the claim.
\end{proof}
Now, we can say that a smooth metric $g[i]$ on $M$ and a complete hyperbolic metric $h$ of $M\setminus K$ has a common cylindrical adjustment $h_i$.

Moreover, as comparing (\ref{eq:hypcyl}) and (\ref{eq:cappedhyp}) we have an inequality
\be
h_i \geq g[i] \qquad\text{ for all } i\geq 1.
\ee

Next, let us define cylindrical adjustments $g[i]_j$ of $g[i]$ for $j\geq 0$ as follows
\be\label{eq:capcyl}
g[i]_j = \begin{cases} g[i] \quad & \text{on } N_{i+j-1/2} \\
	da^2 + e^{-2(i+j)}dx^2 + e^{-2(i+j)}dy^2 \quad & \text{on } N \setminus N_{(i+j)}
\end{cases}
\ee

Note that we can put $g[i]_0=h_i$.
By the direct comparing  (\ref{eq:cappedhyp}) and (\ref{eq:capcyl}), we obtain
	
	\be\label{eq:compmetrc}
	g[i]_0  \geq g[k]_0
	\ee
for all $1\leq i\leq k$.  This finishes the proof.
\end{proof}

Now we are ready to prove the following comparison theorem.

\begin{thm}\label{thm:comparison} Suppose $K$ is a hyperbolic knot on $M$. Then
we have an (algebra) isomorphism
$$
HW^d(\nu^*T;H_h) \cong HW^d(\del_\infty(M \setminus K))
$$
for all integer $d \geq 0$. In particular $HW^d(\del_\infty(M \setminus K)) = 0$ for all $d > 0$ and
$HW^0(\del_\infty(M \setminus K))$ is a free abelian group generated by $\mathscr G_{M\setminus K}$.
\end{thm}

\begin{proof} Let $h_i$ be a cylindrical adjustment of $h$ given in the proof of Theorem \ref{thm:inftyrank}.
We first observe
\be\label{eq:hj<ji}
h_j \leq h_i \quad \text{for all }\, j > i.
\ee
Denote by $g[i]$ the smooth metric constructed in the proof of Proposition \ref{prop:hi-smoothadjust}
associated to $h_i$, and $g[i]_j$ its cylindrical adjustment with $g[i]_0 = h_i$ associated to
$$
N_i \subset N_{i+1} \subset \cdots.
$$

Using the monotonicity inequality \eqref{eq:gil-gkj}, we can find a sequence of
metrics
$$
g[1]_0 \geq g[2]_0 \cdots \geq g[i]_0 \geq \cdots.
$$
Then
$$
CW(\nu^*T; H_{g[i]_0}) = CW(\nu^*T; H_{h_i})
$$
whose cohomology satisfies
\be\label{eq:HWg[i]=}
HW(\nu^*T; H_{g[i]_0}) \cong HW(\del_\infty(M\setminus K))
\ee
for all $i$ by the definition of the latter in \cite{BKO}. Furthermore we have the monotonicity $A_\infty$ homomorphism
\be\label{eq:CWg[i][j]}
CW(\nu^*T;;H_{g[i]_0}) \to CW(\nu^*T; H_{g[j]_0})
\ee
and obtain an $A_\infty$ homomorphism
(see \cite[Section 7.2.12]{fooo:book} as before), and hence
a homomorphism
$$
(\varphi_{ij})_*: HW(\nu^*T; H_{g[i]_0}) \to HW(\nu^*T; H_{g[j]_0}).
$$
(In fact, this is an algebra homomorphism but the property will not be used in the present paper.)

We recall from Proposition \ref{prop:hi-smoothadjust} that $h_i = g[i]_0$ and so
$H(g[i]_0) = H(h_i)$  for all $i$. Therefore we have the monotonicity $A_\infty$ homomorphism
\eqref{eq:CWg[i][j]} coincides with
$$
\iota_{ij} :  CW(\nu^*T;H_{h_i}) \to CW(\nu^*T;H_{h_j})
$$
since $g[i]_0 = h_i$ by construction of $g[i]$ in Proposition \ref{prop:hi-smoothadjust}.

On the other hand, \cite[Proposition 10.5]{BKO} and Lemma \ref{lem:twolimits} respectively
imply that the maps
$$
(\iota_{ij})_*  = (\varphi_{ij})_*:  HW(\nu^*T;H_{g[i]_0}) \to HW(\nu^*T;H_{g[j]_0})
$$
are  isomorphisms for all $i \leq j$. Furthermore
$$
\lim_{\longrightarrow} HW(\nu^*T;H_{h_i}) \cong HW(\nu^*T; H_h)
$$
by Theorem \ref{thm:inftyrank}. Therefore the following lemma will finish the proof of the theorem.
\begin{lem}\label{lem:gi-gilim}
The canonical map
$$
HW(\nu^*T;H_{h_k}) \to \lim_{\longrightarrow} HW(\nu^*T;H_{h_i})
$$
is an isomorphism for all $k$.
\end{lem}
\begin{proof} Let $k =i_0$ be fixed.
We start with the proof of injectivity.
Let $a_{i_0} \in HW(\nu^*T;H_{h_i})$
be an element satisfying $\iota_i(a_{i_0}) = 0$. In other words, $\iota_i(a_{i_0})$ can be
represented by a sequence $a_i \in HW(\nu^*T;H_{h_i})$ satisfying
$$
(\iota_{i_0j})_*(a_{i_0}) = 0
$$
for sufficiently large $j \geq i_0$. Since $(\iota_{i_0j})_*$ is an isomorphism, we have $a_{i_0} = 0$.
This proves injectivity.

For the surjectivity, let $\displaystyle{a \in \lim_{\longrightarrow} HW(\nu^*T;H_{h_i})}$ be given
and let $\{a_i\}$ be a representative thereof.
It is enough to show
$
(\iota_{i_0})_*(a_{i_0}) = a
$
i.e., we have
$$
a_j = (\varphi_{i_0j})_*(a_{i_0})
$$
 for all sufficiently large $j \geq i_0$.
By the energy estimate \eqref{eq:Ec+-}, $(\varphi_{i_0j})_*(a_{i_0})$ is eventually stable, i.e.,
there exists some $i_1 > i_0$ such that
$$
(\varphi_{i_0j})_*(a_{i_0}) = (\varphi_{i_0i_1})_*(a_{i_0})
$$
for all $j \geq i_1$ similarly as in Lemma \ref{lem:i0hj}.

On the other hand, by compatibility of $\{a_i\}$, we also have
$$
a_j \sim (\iota_{i_0j})_*(a_{i_0})
$$
and in turn $a_j = (\iota_{i_1j})_*(a_{i_0})$ for all $j \geq i_0$ by formality.
Combining the above discussion, we have shown $a = (\iota_{i_0})_*(a_{i_0})$.
Since $k = i_0$ is arbitrarily given, the surjectivity for folds for all $k$ as required.
This finished the proof.
\end{proof}

Combining the above, we have finished the proof of Theorem \ref{thm:comparison}.
\end{proof}

Combination of Theorem \ref{thm:inftyrank} and \ref{thm:comparison} gives rise to
the following

\begin{cor}\label{cor:finite-rank} For any hyperbolic knot $K \subset M$, $\widetilde{HW}^d(\del_\infty(M\setminus K)) = 0$
for all $d \geq 1$ and the rank of $\widetilde{HW}^0(\del_\infty(M \setminus K))$  is
infinity. In particular if $\widetilde{HW}^d(\del_\infty(M \setminus K)) \neq 0$  for some $d \geq 1$,
the knot cannot be hyperbolic.
\end{cor}

We would like to point out that the knot Floer algebra $HW(\del_\infty(M \setminus K))$ is
defined for arbitrary knots, while $HW(L;H_h)$ in the comparison result of Theorem \ref{thm:comparison}
should have
a description in terms of the hyperbolic geometry of ${M} \setminus K$. The upshot of
Theorem \ref{thm:comparison} is that it enables us to compute the topological invariant
$HW(\del_\infty(M \setminus K))$ in terms of the hyperbolic geometry of the complement $M \setminus K$
for the case of hyperbolic knots.

\section{The case of torus knots in $S^3$ and in $S^1 \times S^2$}
%\subsection{$\H^2\times \R$ geometry}

Let $K$ be a $(p,q)$-torus knot in $S^3$, where $p,q \geq 2$ and $\gcd(p,q)=1$.
Following \cite{Yi}, we first describe the $\hr$ geometry of $S^3 \setminus K$.
We will equip $S^3 \setminus K$ with $\hr$ geometry by
constructing the fundamental domain $D$ in $\hr$.
First, we consider a hyperbolic polygon $P$ in $\H^2 = \H^2\times\{0\}\subset\hr$ with $2p$ vertices of $v_1,\dots v_{2p}$ such that
\begin{enumerate}[(i)]
	\item each vertex $v_{2i}$ is ideal vertex for $i = 1,\dots,p$,
	\item each vertex $v_{2i-1}$ has the cone angle  $\frac{2\pi}{p}$ for $i = 1,\dots,p$,
	\item $P$ is  symmetric under every reflection in a geodesic from $v_i$ to $v_{i+p}$ for $i=1,\dots,p$.
\end{enumerate}
Then, there is a unique orientation preserving isometry  $\phi_i$ of $\H^2$ from $\overline{v_{2i-1}v_{2i}}$ to $\overline{v_{2i+1}v_{2i}}$ where $\overline{v_iv_{j}}$ denotes the oriented edge from $v_i$ to $v_{j}$. Note that each $\phi_i$ is a parabolic isometry of $\H^2$ and all of them are conjugate to  one another.
Moreover, denote by $\tau_k$ for  $k=1,\dots, p$ an isometry of $\H^2$ which maps $v_i$ to $v_{i+2k}$ for all $i$. Note that $\tau_k$ is a rotational isometry of $2 \pi k/ p$. It is convenient to use Poincar\'e disk model of $\H^2$ for visualizing a symmetric picture. See Figure \ref{fig:PinPoincaredisk}.

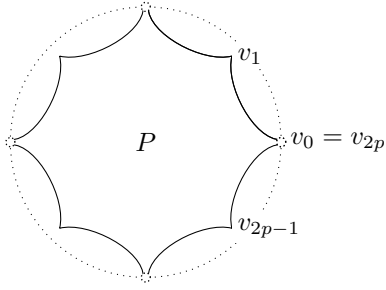
\begin{figure}[h!]
	\begin{tikzpicture}
	
	\pgfmathsetmacro\s{0.6}  %The whole picture scaling \s=1
	\pgfmathsetmacro\n{4}  % n-gon
	
	\pgfmathsetmacro\Angle{360 / \n}
	\pgfmathsetmacro\PoincareRadius{\s*3}
	\pgfmathsetmacro\Iradius{\PoincareRadius*0.9}
	\pgfmathsetmacro\curfactor{0.8}
	\pgfmathsetmacro\initangle{0}
	\pgfmathsetmacro\sfactor{1}
	
	\pgfmathsetmacro\Height{\s*6.0}
	\pgfmathsetmacro\Width{\s*4.0}
	\pgfmathsetmacro\LhoroRadius{\s*0.9}
	\pgfmathsetmacro\RhoroRadius{\s*1}
	\pgfmathsetmacro\geoRadius{\Width/2}
	\pgfmathsetmacro\LhoroAngle{2*atan(\LhoroRadius/\geoRadius)}
	\pgfmathsetmacro\RhoroAngle{2*atan(\RhoroRadius/\geoRadius)}
	
	\pgfmathsetmacro\domainRadius{\s*1.7}
	\pgfmathsetmacro\domainX{-\Width*2}
	\pgfmathsetmacro\domainY{\Height*0.6}

	\draw [ dotted](0,0) circle ( \PoincareRadius );
	
	\foreach \x in { 0,\Angle,...,360}
	\draw
	(\x+\initangle+\sfactor:\PoincareRadius) .. controls (\x+\initangle+\sfactor:\PoincareRadius*\curfactor) and (\x+\initangle+\Angle*0.4 :\Iradius*\curfactor) .. (\x+\initangle+\Angle/2:\Iradius)
	(\x+\initangle+\Angle*0.5:\Iradius) .. controls (\x+\initangle+\Angle*0.6:\Iradius*\curfactor) and (\x+\initangle+\Angle -\sfactor:\PoincareRadius*\curfactor) .. (\x+\initangle+\Angle-\sfactor:\PoincareRadius);
	
	\foreach \x in { 0,\Angle,...,360}
	\draw [line width=0.5pt,fill=white,densely dotted] (\x+\initangle:\PoincareRadius) circle (\s*0.1);	
	
	\draw  (0,0) node {$P$};
	\draw  (\initangle:\PoincareRadius) node[right] {$v_0=v_{2p}$};
	\draw  (\initangle+\Angle*1/2:\Iradius) node[right,fill=white,inner sep=2pt] {$v_1$};
	\draw  (\initangle-\Angle*1/2:\Iradius) node[right,fill=white,inner sep=2pt] {$v_{2p-1}$};
	\end{tikzpicture}
	\caption{A polygon $P$ in the Poincar\'e disk model of $\H^2$}
	\label{fig:PinPoincaredisk}	
\end{figure}

Let  $D:=P\times [0,q] \subset \hr$ which would be a fundamental domain of $S^3\setminus K$.
The face-paring maps in $\Is(\H^2\times \R)$  are given by
\begin{align*}
\phi_i'&:=\phi_i\times (x\mapsto x+1) \text{ for } i=1,\dots p, \text{ and }\\
\tau_q'&:=\tau_{q}\times(x\mapsto x+q).
\end{align*}
Now we have the following proposition.
\begin{prop}\label{prop:S3comple}
	The domain $D$ glued by $\{\phi_1',\dots,\phi_p',\tau_q'\}$  is homeomorphic to $S^3\setminus K$.  	
\end{prop}
\begin{proof}
	Consider a regular hyperbolic compact $p$-gon $P'$ in $P$ whose vertices are $v_1,v_3,...,v_{2p-1}$. Let us look at  $P'':= P\setminus P'$. The closure of each connected component of $P''$ is a triangle of $v_{2i-1}, v_{2i}, v_{2i+1}$ for $i=1,\dots,p$.
	We consider a decomposition of $D=D' \cup D''$ where $D'=P'\times [0,q]$ and $D''= P''\times [0,q]$. Then $D'$ glued by $\tau'_q$ becomes a solid torus $S'$ and the closure of $D''$ glued by $\{\phi_1',\dots,\phi_p',\tau_q'\}$ also become a solid torus $S''$ in which the core curve is $v_{2i} \times [0,q]$. The common boundary of $S'\cap S''$ is a 2-torus $T$ given  by $\partial P'\times [0,q]$ glued by $\tau'_q$ since the decomposition of $P$ is invariant under $\tau_q$. Now we can see that the longitude curve of $S'$ and the meridian curve of $S''$ are in the same homotopy class in $T$. It completes the proof.
\end{proof}

We remark that the holonomy of a meridian $\mu$ of $K$ is conjugate to any $\phi_i'$. Moreover $(\tau_q')^p$ is the holonomy  of a longitude $\lambda$ of $K$ where the writhe of $\lambda$  is  $(p-1)q$.
Note that the sum of all cone angle of $v_1,v_3,\dots,v_{2p-1}$ is $2\pi$.
By construction, the resulting $3$-manifold obtained by the face pairings from $D$ is homeomorphic to $N:=S^3\setminus K$ and its universal cover $\w N$ with $\pi: \w N \to N$ is identified with $\hr$, where the deck transformation group $\Gamma$ is generated by $\{\phi_1',\dots,\phi_p',\tau_q'\}$, which acts properly discontinuously on $\hr$. This finishes a description of $\hr$ geometry of $S^3 \setminus K$.

\subsection{Semi-horo-torus}
From now on, we think of the upper half space model of $\H^2=\{(x,y) \mid y>0\}$. Then a standard metric $h$ for $\hr$ is given by
$$h=\frac{dx^2 +dy^2}{y^2}+dz^2.$$
Without loss of generality, we can assume that $v_0(=v_{2p})=\infty \in \H^2\subset \H^2\times\{0\}\subset \hr$.
Let us consider  $\w N_{y_0}:=\{(x,y,z) \mid y > y_0\}\subset \H^2\times\R$.  Then there exists  a sufficient large $y_0>0$ such that
$$ g\cdot \w N_{y_0} \cap \w N_{y_0} = \varnothing$$ for any $g\in \Gamma$, and $N_{y_0}:=p(\w N_{y_0})$ is a small tubular neighborhood of $K$ in $N$.
We remark that $T:=\partial N_{y_0}$ carries a natural Euclidean metric inherited from $h$.
%   where the meridian cycle in $T$ is lifted to a horocycle in $\H^2$.

Note that a connected component of the lifting of $T$ looks like a cylinder of a horocycle times $\R \subset \hr $. Motivated by this, we call $T$ a \emph{semi-horo-torus}.
If the center of the horocycle is $\infty \in \H^2$ then the connected component of  $\w T$ is a horizontal plane of $\{(x,y,z) \mid y=y_0\}$.

\subsection{Geodesic cords}\label{sec:GeodesicCords}

Consider a hyperbolic surface $S_t$ obtained from the set of polygons
$$P_{t+k}:= P\times \{t+k\} \subset \H^2 \times \{t\} \subset \hr \quad\text{ for } k=0,1,\dots,q-1$$
by gluing them by $\phi_i'$'s. Then $S_t$ become a $p$-punctured surface which is embedded in $N$ and carries a complete hyperbolic metric inherited from $\H^2$. Here, $\chi( \widehat S_t) = 2p+q-pq$ is
the Euler characteristic of $\widehat S_t$ where $\widehat S_t$ is obtained by capping off the $p$ punctures.
%Remark that if one take $p$ is an odd integer then $S_t$ is an orientable surface.
Let $C_t:= T \cap S_t$. Then $C_t$ is the boundary of a tubular neighborhood of the punctures of $S_t$.
% is a disjoint union of horo-cycles and

Let us recall the definition of geodesic cords,
$$
\cord(N,T)=\{ c:[0,1] \to N \mid \nabla_t \dot c = 0,\,c(i) \in T, \,\, \dot c(i) \perp T, \,\text{for } i=0, 1\}.
$$
Note that $S_t$ is a totally geodesic surface in $N$.
\begin{lem}\label{lem:toruscord}
	A geodesic cord $c\in \cord(N,T)$ is contained in a $S_t$. Moreover
	$$ \cord(S_t,C_t) \subset \cord(N,T)$$
for all $t$.
\end{lem}
\begin{proof}
	Consider a lifting $\w c$ to $\hr$ and take a lifting of $T$ containing $c(0)$ such that  $$\w T=\{(x,y,z)\mid y=y_0\} \quad\text{for some } y_0>0.$$
	Then  $\dot c(0) \perp T$ implies that $\w c$ is a vertical line. Hence $\w c$ is contained in a $\H^2 \times \{t\}$ for some $t$. The fact that $\pi(\H^2\times \{t\})=S_t$ implies $c \subset S_t$. Moreover, any geodesic in $S_t$ is also a geodesic in $N$ and any normal vector of a horo-cycle $C_t$ is also perpendicular to $T$. It proves the second assertion.   	
\end{proof}

\begin{prop}\label{prop:toruscord}
	For each $c\in \cord(N,T)$, there is an $S^1$-family of geodesic cords $\{c_\theta\mid \theta \in S^1\}$ containing $c$, where the parameter of $S^1$ is identified to
a longitude $\lambda$.
 Moreover, any geodesic cord $c'$ in the homotopy class $[c]$ for $[(I,\partial I),(N,T)]$ is contained in $\{c_\theta\}$.
\end{prop} 	

\begin{proof}
	
	Let us consider a lifting $\w c$ such that
	$$\w c(0)=(x_0,y_0,z_0) \in \w T_0:= \{(x,y,z) \mid y=y_0\}.$$
	Take another lifting $ \w T_1$ at $\w c(1)$. Then $\w T_1$ and $\w T_0$ in $\H^2\times\R$ are invariant under a translation $\ell_s$ along the $z$-axis, i.e., $\ell_s: (x,y,z)\mapsto (x,y,z+s)$. Therefore $\ell_s \circ \w c$ is also perpendicular to $\w T_1$ and $\w T_0$, and hence $\pi(\ell_s\circ\w c) \in \cord(N,T)$. Moreover, $\pi(\ell_s\circ \w c) = \pi(\ell_{s'}\circ \w c)$ if and only if $s-s' \in p\Z$ because  $\ell_p = (\tau_q)^p$ is the holonomy of the longitude $\lambda$.
	In particular, $\w T_1$ is determined only by the homotopy class of $[c]$ for $ [(I,\partial I),(N,T)]$. For any geodesic cord $c' \in [c]$ should be a vertical line between the same $\w T_0$ and $\w T_1$ and hence $c' \in \{c_\theta\}$.
\end{proof}
%
% Let $\cord'(N,T)$ denotes the set of non-constant cords in $\cord(N,T)$
% \begin{thm}
% 	$$\cord'(S_t,C_t)\times K \cong \cord'(N,T)$$.
% \end{thm}

In summary,
we have the following classification of geodesic cords for a semi-horo-torus $T$ in a torus knot complement $S^3\setminus K$.

\begin{enumerate}[(a)]
	\item $\T^2(\cong T)$ many constant cords.
	\item  $S^1(\cong \lambda \cong K)$ many oriented non-constant geodesic cords in each free homotopy class.
\end{enumerate}

\subsection{$(p,q)$-Torus knots in $S^2\times S^1$}
%\subsection{$\H^2\times \R$ geometry}
We now consider  a $(p,q)$-torus knot $K$ in $S^2\times S^1$ where $p>|q| \geq 2$ and $\gcd(p,q)=1$. It also admits $\hr$ geometry which is easier to see the geometric structure than the case of $S^3$.

The polygon $P \subset H^2\times\{0\}$ and
$D:=P\times [0,p] \subset \hr$ are the same as in the case of $S^3$, but
here, $D$ would be a fundamental domain of $S^2\times S^1\setminus K$. The face-paring maps in $\Is(H^2\times \R)$  are  given by
\begin{align*}
\phi_i''&:=\phi_i\times (x\mapsto x) \text{ for } i=1,\dots p, \text{ and }\\
\tau_q''&:=\tau_{q}\times(x\mapsto x+p).
\end{align*}
Then we have the following proposition as like Proposition \ref{prop:S3comple}.
\begin{prop}
	The domain $D$ glued by $\{\phi_1'',\dots,\phi_p'',\tau_q''\}$  is homeomorphic to $S^2\times S^1\setminus K$.  	
\end{prop}
The proof is the same as before, but two meridians in $S'$ and $S''$ are the same homotopy class in $T$ and hence the ambient manifold are $S^2\times S^1$.
In summary,   $D$ with gluing by $\Gamma:=\{\phi_1'',\dots,\phi_p'',\tau_q''\}$ becomes a knot $K$ in $S^2\times S^1$ and $K$ is a closed curve of slope $(p,q)$ on an embedded torus in $S^2\times S^1$. Moreover $\Gamma$ generates a discrete subgroup of  $\Is(\hr)$ and $S^2\times S^1 \setminus K \cong \hr /\Gamma$.
The holonomy of a meridian $\mu$ or a longitude $\lambda$ of $K$ is also $\phi_i''$ or $(\tau_q'')^p$ respectively.

Lemma \ref{lem:toruscord} and Proposition \ref{prop:toruscord} also hold without any modification, but the embedded surface $S^t$ is simpler and  always becomes a $p$-puncture sphere
obtained by the polygon
$$P_t:= P\times \{t\} \subset \H^2 \times \{t\} \subset \hr$$
by gluing of $\phi_i\times\{t\}$'s.
Note that each connected component of $T\cap S_t$ is exactly a meridian $\mu$ unlike the case of $S^3\setminus K$.

By the exactly same argument,
we obtain the same classification of geodesic cords for a semi-horo-torus $T$ in a $(p,q)$ torus knot complement $S^2\times S^1 \setminus K$.

\begin{enumerate}[(a)]
	\item $\T^2(\cong T)$ many constant cords.
	\item  $S^1(\cong \lambda \cong K)$ many oriented non-constant geodesic cords in each free homotopy class.
\end{enumerate}

\subsection{Computation of wrapped Floer homology}
Now we consider the reduced wrapped Floer chain complex $\widetilde{CW}^*(\partial_\infty(S^3\setminus K))=CW^*(\partial_\infty(S^3\setminus K))/C^*(K)$, where $C^*(K)$ is the subchain complex of the constant chords. For the computation of the chain complex, we need to take a cylindrical adjustment of a restricted metric $g=g_{S^3}|_N$ on $N=S^3 \setminus K$.

By the observation in Subsection \ref{sec:GeodesicCords}, without loss of generality, we take a cylindrical adjustment $g_0$ such that its lifting $\widetilde g_0$ to the universal cover is invariant under the $\R$-translation induced by the longitude coordinate $\lambda$.
Under such adjustment $g_0$, Proposition~\ref{prop:toruscord} continues to hold .
By the Morse-Bott technique, we additionally consider a Morse function $f_c$ on each $S^1$-parameterized geodesic cords of homotopy type $c\in {\rm Cord}(N,T)$.
Especially we choose $f_c$ so that it has two critical points $c_{\rm max}$, $c_{\rm min}$ with $\mu_{\rm Morse}(c_{\rm max})=1$, $\mu_{\rm Morse}(c_{\rm min})=0$ and the zero differential with respect to the gradient flow line of $f_c$.

Now we have
\begin{align*}
\widetilde{CW}^*(\partial_\infty(S^3\setminus K)) \cong \widetilde{CW}^*(\partial_\infty S_t) \oplus \widetilde{CW}^*(\partial_\infty S_t)[1],
\end{align*}
where $[1]$ denotes the degree shift by $+1$.
Note that $S_t$ in Subsection \ref{sec:GeodesicCords} is a $p$-punctured hyperbolic surface.
Let us recall the comparison argument from \cite[Section 10]{BKO}, this time applied to the pair of
a cylindrical adjustment of the restricted metric and the complete hyperbolic metric.
Even though we only deal with the 3-dimensional case in {BKO}, the comparison argument does not depend on the dimension of the base manifold. In the 2-dimensional case, we obtain
\begin{align*}
\widetilde{CW}^*(\partial_\infty S_t) \cong \widetilde{CW}^*(\nu^* C_t, T^*S_t; H(h)),
\end{align*}
where $h$ is a metric on $S_t$ satisfying $\widetilde{h}=\frac{1}{y^2}(dx^2 + dy^2)$ on $\H^2$.

Since the metric $h$ is complete and hyperbolic, there exists a unique Hamiltonian chord $\gamma_x$ with $\mu(\gamma_x)=0$ for each relative homotopy class $x \in {\rm Chord}(\nu^* C_t, T^* S_t)$. In conclusion, we have
\begin{align*}
\widetilde{CW}^*(\partial_\infty(S^3\setminus K))=
\begin{cases}
\bigoplus_{x \in {\rm Chord}(\nu^* C_t, T^* S_t)}\Z\cdot \gamma_x &\text{if}\quad *=0,1;\\
0&\text{otherwise}.
\end{cases}
\end{align*}
Since the differential $\widetilde{\frak m}^1$ on $\widetilde{CW}^*(\partial_\infty S_t)$ and the differential of Morse-Bott function vanish, we have obtained
\begin{thm}\label{thm:torusKnot} Let $K \subset S^3$ be any torus knot in $S^3$. Then
\begin{align*}
\widetilde{HW}^*(\partial_\infty(S^3\setminus K))=
\begin{cases}
\bigoplus_{x \in {\rm Chord}(\nu^* C_t, T^* S_t)}\Z\cdot [\gamma_x] &\text{if}\quad *=0,1;\\
0&\text{otherwise}.
\end{cases}
\end{align*}
\end{thm}

\appendix

\section{Proof of Lemma \ref{lem:twolimits}}
\label{sec:lemma-twolimits}

We consider two metrics $g, \, g'$, which are both
cylindrical outside $N_i$ for some $i$, and their
associated kinetic energy Hamiltonian. Consider their convex combination
$$
g_s = (1-s) g + s g'
$$
whose associated Hamiltonian satisfies similar identity
$$
H_s: = (1-s) H_{g} + s H_{g'}.
$$
We note that the metric $g_s$ is cylindrical on $N_j$ which has the form
$$
g_s = da^2 \oplus g_s|_{\del N_j} \quad \text{on } [0,\infty) \times \del N_j
$$
with $g_s|_{\del N_j} = (1-s) g|_{\del N_j} + s g'|_{\del N_j}$.

For the construction of the quasi-isomorphism
$$
\Phi : CW(\nu^*T ;H_{g_0}) \to CW(\nu^*T ;H_{g_0'}),
$$
for given pair $i, \, j$ with $i \leq j$,
we follow the construction given in \cite[Section 9]{BKO} over some \emph{monotone} homotopy $s \mapsto H_s$.
For the readers' convenience, we repeat verbatim with indication of the small change to
be made arising from the fact that \emph{the hyperbolic metric $h$} does not extend smoothly
to $M$.

For notational convenience, we denote the kinetic energy Hamiltonian $H_g$ also by $H(g)$ in this section.
We compare $CW(T^*N,H(g_0'))$ and $CW(T^*N,H(g_0))$ similarly as in \cite[Section 7]{BKO}
where the categorical version was constructed in terms of
two Riemannian metric $g_0'$ and $g_0$ on $M \setminus K$ that arise from smooth metrics $g, \, g'$ on $M$
as a pair of cylindrical adjustments of their restrictions to $M \setminus K$. They satisfy the inequality
 \be\label{eq:LipschitzCondition}
 \frac{1}{C} g_0 \leq g_0' \leq C g_0,
 \ee
for some constant $C = C(g,g') > 1$.

We can verbatim follow the construction of \cite{BKO}
in a simpler form to make comparison
between the case of metric $g$ coming from a smooth metric on $M$ and the hyperbolic metric $h$
on $M \setminus K$ in the current case of our interest.

Since the Hamiltonian is given by the {\em dual} metric, we have
\begin{align}\label{eq:confinal}
\begin{cases}
H(g_0') \leq H(\frac{1}{C}g_0);\\
H(g_0) \leq H(\frac{1}{C}g_0').
\end{cases}
\end{align}
Then there are two $A_\infty$ homomorphisms
\begin{align}\label{eqn:two vertical sequences}
\Phi  &:  CW(\nu^*T ,H(g_0')) \to CW(\nu^*T ,H(\tfrac{1}{C}g_0));\nonumber \\
\Psi  &: CW(\nu^*T ,H(g_0)) \to CW(\nu^*T ,H(\tfrac{1}{C}g_0')
\end{align}
which are defined by the standard $C^0$-estimates for the monotone homotopies.

Now consider the composition of the functors
\begin{align*}
\Psi \circ \Phi :  CW(\nu^*T ,H(g_0)) \to CW(\nu^*T ,H(\tfrac{1}{C^2}g_0));\\
\Phi \circ \Psi : CW(\nu^*T ,H(g_0')) \to CW(\nu^*T ,H(\tfrac{1}{C^2}g_0')).
\end{align*}
These are homotopic to natural isomorphisms induced by the rescaling of metrics
\begin{align*}
\rho_{C^2}&: CW(\nu^*T ,H(g_0)) \to CW(\nu^*T ,H(\tfrac{1}{C^2}g_0));\nonumber\\
\eta_{C^2}&: CW(\nu^*T ,H(g_0')) \to CW(\nu^*T ,H(\tfrac{1}{C^2}g_0')),
\end{align*}
respectively. This finishes the proof.

\section{Reduction of the classification problem to 2 dimension}
\label{sec:reduction-2dim}

We note that when all the asymptotic chords are non-constant,
solutions for (\ref{eq:duXHJ}) exist only for $k+1 = 3$ by
the degree reason: the degree of the $A_\infty$ maps $\mathfrak m^k$,
which is given by $2-k$, must be zero by Theorem \ref{thm:index-nullity}.
Therefore it remains to examine the case $k+1 = 3$.

We recall the transformation $u \mapsto \psi_\eta^{-1} \circ u =: v$ for which
$v$ satisfies the autonomous equation \eqref{eq:dvXHJ}.

\subsection{Rewriting of the Cauchy-Riemann equation lifted to $\H^3$}

We consider the lifting $\tilde{v}:\Sigma \to T^*\H^3$ of $v:\Sigma \to T^*N$
to $\H^3$. With slight abuse of notation, we also denote $\tilde v$ just by $v$
as long as there is no danger of confusion.
Then we compute the coordinate expression of the equation of \eqref{eq:dvXHJ}
lifted to $\H^3$.

Now let $(q,p)$ be the canonical coordinates of $T^*\H^3$ induced by the standard
coordinates $q = (x,y,z)$ of $\H^3 \subset \C^3$, and by $p = (p_x,p_y,p_z)$ the associated
fiber coordinates.  We decompose the derivative $dv$, a $T(T^*\H^3)$-valued one-form,
into the horizontal and vertical components
$$
dv = d^H v + d^V v.
$$
Here
\begin{align*}
d^H v & = d(x \circ v) \otimes H_1(v) + d(y  \circ v) \otimes H_2(v) + d(z  \circ v) \otimes H_3(v);\\
d^V v & = \nabla (p_x \circ v) \otimes V_1(v) + \nabla (p_y \circ v) \otimes V_2(v) + \nabla (p_z \circ v) \otimes V_3(v).
\end{align*}
More explicitly for each $i = 1, \,2, \, 3$, we have
$$
\nabla p_i = dp_i - \sum_{j=1}^3 \sum_{k=1}^3 p_k \Gamma^k_{ij} dq^j
$$
where $q^1 = x$, $q^2 = y, \, q^3 = z$.

Recall that the Levi-Civita connections of $g$ induces the splitting
\begin{align}\label{eqn:splitting}
T_{(q,p)}(T^*\H^3) \cong T_q \H^3 \oplus T_q^*\H^3
\end{align}
at each point $(q,p) \in T^*\H^3$. Also recall the lowering and raising operators
with respect to $g$ by
$$
\flat: T\H^3 \to T^*\H^3, \quad \sharp: T^*\H^3 \to T\H^3
$$
where $\flat(X) = \langle X, \cdot \rangle_h$ and $\sharp$ is its inverse. We may regard $\flat,\sharp$ as operations on $T(T^*\H^3)$ with respect to (\ref{eqn:splitting}). The Sasakian almost complex structure $J_h$ is given by
$$
J_h(X) = X^\flat, \quad J_h(\alpha) = - \alpha^\sharp
$$
for $X \in T \H^3$ and $\alpha \in T^* \H^3$, respectively, under the identification (\ref{eqn:splitting}).

Here $\nabla p_x,\, \nabla p_y,\,  \nabla p_z$ are nothing but the coefficients
of the covariant derivative of the one-form
$$
\alpha = p \circ v = p_x(v)\, dx|_f + p_y(v) \, dy|_f + p_z(v) \, dz|_f.
$$
Here $\alpha$ is considered as a section of $f^*(T^*\H^3) \to \Sigma$ along the map $f:=\pi \circ v : \Sigma \to \H^3$.
In other words, we have
$$
\nabla \alpha=  \nabla p_x(v) \, dx|_f + \nabla p_y(v)\,  dy|_f + \nabla p_z(v)\, dz|_f.
$$

We now derive the following coordinate expression of \eqref{eq:dvXHJ} for the map $v$.

\begin{lem}\label{lem:coord-expession} Let ${\bf J} = \psi^*_{\eta(\tau,t)} J_g$ and let
$u$ be a solution to  the equation $(du- X_{H}\otimes \beta)_{\mathbf J}^{(0,1)}=0$.
Let $v$ be the map associated to $u$ as above.
Then the coordinate expression of \eqref{eq:dvXHJ}
with respect to the frame fields $\{\partial_\tau,\partial_t\}$ and $\{H_i,V_j\}_{i,j=1,2,3}$
is given by
\begin{align*}
\begin{dcases}
dq^i(\partial_\tau v) -  z^2\nabla p_i(\partial_t v)=0\quad \text{ for } i=1,2,3;\\
\nabla p_i(\partial_\tau v)+\frac{1}{z^2}dq^i(\partial_t v)- p_i\circ v =0\quad \text{ for } i=1,2,3.
\end{dcases}
\end{align*}
\end{lem}
\begin{proof} Using the above given frame fields, we compute the coordinate expression of
$(dv- X_{H}\otimes \beta)_{\mathbf J}^{(0,1)}=0$ as
\begin{eqnarray*}
&{}& (dx(\partial_\tau v),dy(\partial_\tau v),dz(\partial_\tau v),\nabla p_x(\partial_\tau v),\nabla p_y(\partial_\tau v),\nabla p_z(\partial_\tau v))\\
&{}& +J_g(dx(\partial_t v),dy(\partial_t v),dz(\partial_t v),\nabla p_x(\partial_t v),\nabla p_y(\partial_t v),\nabla p_z(\partial_t v))\\
&{}& - J_g ( z^2 p_x, z^2 p_y,  z^2 p_z,0,0,0)=0
\end{eqnarray*}
and deduce the following coordinate expression of \eqref{eq:dvXHJ}
\begin{equation}\label{eqn:J-hol_coord}
\begin{dcases}
dq^i(\partial_\tau v) - z^2\nabla p_i(\partial_t v)=0\quad \text{ for } i=1,2,3;\\
\nabla p_i(\partial_\tau v)+\frac{1}{ z^2}dq^i(\partial_t v)-  p_i\circ v =0\quad \text{ for } i=1,2,3,
\end{dcases}
\end{equation}
where $q^1=x,\ q^2=y,\ q^3=z$ and $p_1=p_x,\ p_2=p_y,\ p_3=p_z$.
\end{proof}

The coordinate expression \eqref{eqn:J-hol_coord} then
admits the following coordinate-free expression with boundary condition:
\be\label{eq:CRcoordfree}
\begin{dcases}
df -  (\nabla \alpha \circ j + \alpha \cdot dt)^\sharp = 0\\
f(z) \in S, \quad \alpha \in \nu^*_{f(z)}S \quad \text{ for } z \in \del \Sigma,
\end{dcases}
\ee
where $S$ is the union of horo-spheres in $\H^3$ which is the lift of the torus $T$ in $N$ to the universal cover. We emphasize that this equation is written purely in terms of the data $(f,\eta)$
which defined in terms of the data of the pull-back bundle $f^*(T^*N)$ over the base map $f:\Sigma \to N$.

\subsection{Reduction to 2 dimensional hyperbolic plane $\H^2$}

In this section we will provide a complete description of
the set of solutions of \eqref{eq:duXHJ} exploiting the following theorem for the
equation \eqref{eq:CRcoordfree} on the thrice punctured discs.
We will do this first by applying some elementary hyperbolic geometry on $\H^3$ and reducing the study to
the 2 dimensional case $\H^2$.
	
\begin{lem}\label{lem:3d-classify}
	Let $N$ be a complete hyperbolic 3-manifold with one cusp and $T$ is a horo-torus near the cusp.
	Consider three geodesic cords $c^0, \,  c^1, \, c^2$ perpendicular to $T$.
	Let $X$ be a hexagonal domain with edges labelled by $a^0, b^0, a^1, b^1, a^2, b^2$ counterclockwise.
	Suppose that there is a continuous map
	$$f : X \to N,$$
	satisfying $f(a^j)=c^j$ and $f(b^j) \subset T$ for $j=0,1,2$,
	  then the lifted three geodesic cords ${\widetilde c^j} \subset \H^3$  are coplanar, i.e. there is an action $g \in \psl$ sending all ${\widetilde c^j}$ into a hyperbolic plane $\{x=0\}$.
\end{lem}
\begin{proof}
	Consider a lifting of $f$ into the universal cover $\H^3$,
	 $${\widetilde f} : X \to \H^3$$ with the lifted boundary curves denoted by ${\widetilde f( a^j) =: \widetilde A^j}$ and ${\widetilde f(b^j)=: \widetilde B^j}$.  Each lifted peripheral curve ${\widetilde B^j}$ is contained in a horo-sphere, let say, $S^j \in \H^3$ which is a connected component of the lift of $T$. Then {$\widetilde A^j$} and ${\widetilde A^{j+1}}$ are geodesics perpendicular to $S^j$ for $j=0,1,2$ $\mod 3$.
	Because of the fact that all inward geodesics perpendicular to a horosphere goes to the same ideal point called \emph{the center of the horosphere}, we have an ideal triangle $\Delta\subset \H^3$ where each ideal vertex is the center of a horosphere $S^j$. Therefore all {$\widetilde c^j$} are contained in the boundary of a triangle $\Delta$ and hence coplanar.
\end{proof}

The following is again derived using the strong maximum principle applied to the function
$\varphi = \phi \circ v$ with $\phi(x,y,z) = x/z$.

\begin{thm}\label{thm:inx=0}
 Let $(\widetilde \gamma^0,\widetilde \gamma^1,\widetilde \gamma^2)$ be a triple of three
	Hamiltonian chords that admit a solution $v: (\Sigma,\partial \Sigma) \to (T^*\H^3,S)$
	to the perturbed $J$-holomorphic equation (\ref{eq:duXHJ})	where ${\widetilde \gamma^i}$ is the Hamiltonian chord whose projection ${\widetilde c^i: = \pi \circ \widetilde \gamma^i}$
	is the geodesic whose image is contained in the plane $\{x = 0\}$. Then the image of $v$ is also contained in
	the plane $\{x = 0\}$.
\end{thm}
\begin{proof}
	Let us first recall the construction of the Lagrangian $L$ which is a conormal lifting of $T$.
	Since $T$ is a level set of the Busemann function $-\log z$,
	$$
	L:=\nu^*T=\{(x,y,z_0,0,0,p_z)\}\subset T^*\H^3
	$$
	for some $z_0\in\R^+$. Also note that
	\be\label{eqn:tangent_lagrangian}
	T_{(q,p)}L=\langle H_1,H_2,V_3 \rangle_{(q,p)}\subset T_{(q,p)}T^*\H^3.
	\ee
	
For the purpose of classification problem of holomorphic triangles,
we will control the behavior of $w: = \pi \circ v: \Sigma \to \H^3$ in the $x$-direction on $\H^3$.
As the first try, it is natural to attempt to apply it to the coordinate function $x$ itself.
It turns out that this obvious choice of $x$ does not lead to a favorable formula for the Laplacian of $x$
in an application of the maximum principle unlike the case of coordinate function $z$ (or rather $1/z$).
What turns out to be the right choice is the quotient $x/z$. With the $C^0$ bound of
the $z$-coordinate away from $z = 0$ or $z = \infty$ already established,
the bound of $x$-coordinate is equivalent to a
bound of $\frac{x}{z}$ which turns out to be the right quantity to look at for the
application of maximum (or strong maximum) principle for the proof of the theorem.

\begin{prop}\label{prop:max-principle-z/x} Let $v$ be any solution of \eqref{eq:duXHJ}.
Define $\varphi = \phi \circ v$ with $\phi = x/z$ on $\H^2$. Then
\be\label{eq:finalddjvarphi}
\Delta \varphi\, dA = \varphi(v^*\omega - v^*dH \wedge \beta) + 2\varphi H(v)  (\beta \circ j) \wedge \beta
+ d \varphi \wedge (\beta- v^*\theta).
\ee
\end{prop}
Since the precise calculation of $\Delta(\varphi)$ is rather involved but straightforward, we
postpone its derivation till Appendix \ref{sec:Deltaxz}.

We recall from Lemma \ref{lem:positive} that the two form $v^*\omega - v^*dH \wedge \beta$ is a nonnegative form.
Therefore the \emph{maximum principle (resp. the minimum principle), especially the strong maximum principle,
at critical points of positive value (resp. of negative value) applies by the similar arguments used in the proof of Theorem \ref{thm:z-coord}.}

	Because of the asymptotic condition and thanks to the bound on the $z$-coordinate \eqref{eq:bound-z},
	the maximum (or the minimum) of the function $(x/z)\circ v:\Sigma\to\R$ is achieved at a point in
	$\Sigma$.
	The interior maximum (and minimum) principle is already done by the equation (\ref{eq:finalddjvarphi}),
	and hence the maximum (or minimum) is achieved on a boundary point in $\del \Sigma$.
	Again, thanks to the bound on the $z$-coordinate \eqref{eq:bound-z}, this also implies that
	the maximum (or the minimum) of the function $x\circ v:\Sigma\to\R$ is achieved.

	Now suppose the maximum of $z/x$ is achieved at $m_0\in\partial \Sigma$ and the maximum value is positive.
	We choose a complex coordinate $s+it$ on a neighborhood  $U$ of $m_0$ such that
	\beastar
	\Sigma\cap U\subset\R+i\R^{\geq0}\subset \C\\
	\partial \Sigma\cap U\subset\R+i\cdot0\subset \C
	\eeastar
	By the Lagrangian boundary condition $\tau\mapsto v(\tau+i\cdot0)$ defines a curve on $L=\nu^*T$ and
	\beastar
	x\circ v|_{\partial\Sigma\cap U}:\R&\to & \R\\
	s&\mapsto&  x(v(\tau+i\cdot0))
	\eeastar
	has a maximum at $\tau_0$ where $m_0=\tau_0+i\cdot 0$. In particular, we have
	$$
	dx(\partial_\tau v(m_0))=0
	$$
	i.e. $\partial_\tau v(m_0)\in\ker H^1_{v(m_0)}\cap T_{v(m_0)}L$.
	Then (\ref{eqn:tangent_lagrangian}) implies $\partial_\tau v(m_0)\in\langle H_2,V_3\rangle$
	and hence
	$$
	J\partial_\tau (m_0)\subset\langle V_2,H_3\rangle_{v(m_0)}\subset \ker dx_{v(m_0)}.
	$$
	On the other hand,
	$J$-holomorphic equation (\ref{eq:duXHJ}) with $\beta_\tau =0$ on $\partial\Sigma$ implies
	$$
	\partial_\tau v(m_0)+J\partial_t v(m_0)= \beta_t(m_0) J X_H(v(m_0)).
	$$
	and so
	$$
	\partial_t v(m_0) = J \partial_\tau v(m_0) + \beta_t(m_0) X_H(v(m_0)).
	$$
	Here we recall $X_H=z^2(p_x H_1+p_y H_2+p_z H_3)$ and $v(m_0)\in L=\{(x,y,z_0,0,0,p_z)\}$ which implies $X_H(v(m_0))\in\langle H_3 \rangle$. Since $J\partial_s v(m_0)\in\langle JH_2,JV_3 \rangle=\langle V_2,H_3 \rangle$, we have $$\partial_t v(m_0)\in\langle V_2,H_3\rangle.$$
	
	Now we claim that $z\circ v(\sigma)\leq z_0$ for any $\sigma\in\Sigma\cap U$. Suppose not, then we may assume that a (local) maximum is achieved at $\sigma_0\in\mathring U$. This assumption is possible because of $C^0$-estimate for the base coordinates $x,y,z$.

Let us consider an isometry
	$$
\psi:\H^3\to\H^3
$$
that restricts to
	 $(x_0,y_0,z) \mapsto (x_0,y_0,z^{-1})$ on the line $\ell = \{(x,y,z) \mid x=x_0, y=y_0\}$
with $\sigma_0=(x_0,y_0,z_0)$ and the induced symplectomorphism
	$$T^*(\psi)^{-1}:T^*\H^3\to T^*\H^3.$$
	Then $w:=T^*(\psi^{-1})\circ v:\Sigma\to T^*\H^3$ again satisfies the $J$-holomorphic equation with shifted asymptotic and boundary conditions with respect to $T^*(\psi)^{-1})$. It is easy to see that $\psi\circ w:\Sigma\cap U\to \R$ attains its (local) maximum at the interior point $\sigma_0\in \mathring U$.
	This is not possible by the estimate in  Subsection \ref{subsec:z-coordinate}.
	Since $\partial_t v \in T_{m_0}(\Sigma\cap U)$ is an inner normal direction, the above claim
	$$
	z\circ v(\sigma)\leq z_0\quad\text{for any }\sigma\in \Sigma\cap U
	$$
	implies that $dz(\partial_t v(m_0))\leq 0$.

	On the other hand, $\partial_t v(m_0)\in\langle V_2,H_3\rangle$ and $dz(\partial_t v(m_0))\leq 0$ implies
	\beastar
	d\phi(\partial_t v(m_0))&= & \frac{zdx-xdz}{z^2}(\partial_t v(m_0))\\
	&=& -\frac{x}{z^2}dz(\partial_t v(m_0))\quad\text{ (because $\partial_t v(m_0)$ has no $H_1$-factor)} \geq 0.
	\eeastar
	However this contradicts to the strong maximum principle $\frac{\del \varphi}{\del r} > 0$
	where  $\frac{\del}{\del r}$ is the (outward) normal derivative, since
	$\frac{\del}{\del t} = - \frac{\del}{\del r}$ noting that $\{\frac{\del}{\del s}, \frac{\del}{\del t}\}$ and
	$\{\frac{\del}{\del r}, \frac{\del}{\del s}\}$ has the same orientation on $\Sigma$.
	
	This proves that $\varphi$ cannot have a boundary maximum point either unless
	$\varphi \equiv 0$ which implies $\Im v \subset \text{Zero}(x/z) = \text{Zero}(x)$.
	Since $\lim_{\zeta \to z_i} x(v(\zeta)) = 0$, the latter must be the case which finishes the proof.
\end{proof}

Based on Theorem \ref{thm:inx=0}  we examine the case of $\H^2$. An immediate
consequence thereof is that for the given triple of geodesics prescribed here with the given
boundary horo-circles there is a parametrization $w: \Sigma \to \H^2$
whose image is the given domain bounded by the triple of geodesics and three
horo-circles. We denote this latter domain by $\Theta$. We also denote by
$\dot \Theta$ the subset of $\Theta$ with the three boundary geodesics removed.

\begin{figure}[h!]
	\centering
	\includegraphics{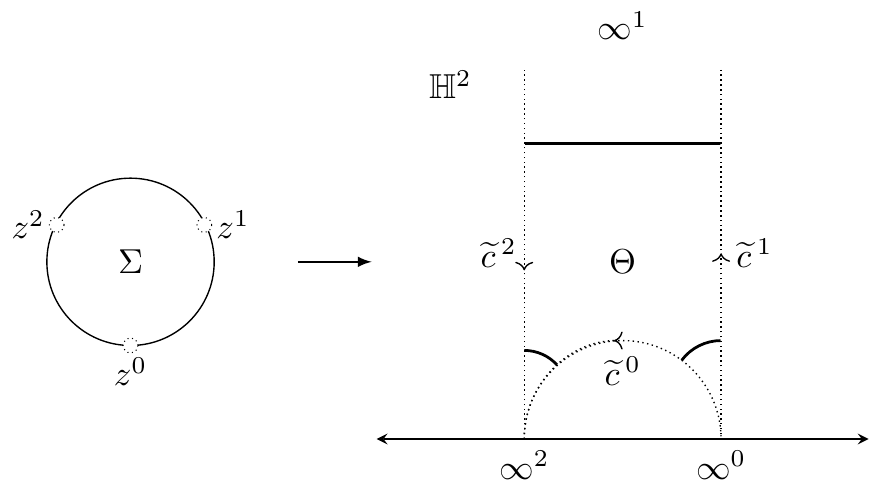}
	\caption{Truncated triangle in $\H^2$}
\end{figure}

In summary, we have proved
\begin{thm}\label{thm:hexagon} Let $v$ be a lifting of any solution of
\eqref{eq:duXHJ}. Then the following holds:
\begin{enumerate}
\item The image of $v: (\Sigma,\partial \Sigma) \to (T^*\H^3,S)$  is contained
in the totally geodesic plane containing the three given asymptotic cords.
\item Identify the totally geodesic plane with the plane $\{x = 0\}$ by applying an isometry of $\H^3$.
Then its image is projected $\H^2 \cong \{0\} \times \H^2$ is a truncated geodesic triangle $\dot\Theta$
as pictured in Figure 2.
\end{enumerate}
\end{thm}
\begin{proof} This is an immediate consequence of maximum principle which we verified can
be applied to the two functions $\frac1z$ and $\frac{x}{z}$ in the proofs of
Proposition \ref{prop:Deltarho} and Proposition \ref{prop:max-principle-z/x},
respectively.
\end{proof}

Conjecture \ref{conj:main} would be the converse to this theorem, i.e., there should
exist a unique solution $v$ for each given truncated geodesic triangle $\dot\Theta$.
We leave the study of this conjecture elsewhere but we mention that
it is enough to consider the problem in 2-dimensional hyperbolic space $\H^2$:
Let $\H^2$ be the hyperbolic two plane with coordinates $(y,z)$ equipped with
the metric
$$
h_{\H^2} = \frac{dy^2 + dz^2}{z^2}.
$$
It is isometrically embedded into $\H^3$ via the map $(y,z) \to (0,y,z)$
whose image is totally geodesic with respect to the metric $h_{\H^3}$.
This in turn induces a totally geodesic almost K\"ahler embedding $T^*\H^2 \hookrightarrow T^*\H^3$
with respect to the Sasakian almost complex structure.
These make the perturbed Cauchy-Riemann
equations on $\H^2$ whose solutions are described in Theorem \ref{thm:inx=0}, canonically provide a solution on $\H^3$.

\section{Computation of $\Delta(x/z)$: Proof of Proposition \ref{prop:max-principle-z/x}}
\label{sec:Deltaxz}

In this appendix, we prove Proposition \ref{prop:max-principle-z/x}. Let $\Delta \rho$ be the
classical Laplacian.

We first note
\be\label{eq:dJs}
-dq^i\circ J=z^2 V^i, \quad i = 1,\, 2,\, 3
\ee
and $- \Delta \varphi\, dA = - d(d\varphi \circ j)$.
Define the function $\phi$ on $\H^3$ by $\phi(x,y,z) = \frac{x}{z}$ and denote $\varphi = \phi \circ v$.
We apply $d\phi$ to \eqref{eq:duXHJ} and derive
$$
d\phi \big(dv + J dv \circ j - \beta X_H(v)- (\beta \circ j)\cdot JX_H(v)\big) = 0.
$$
By the fact that $JX_H$ is tangent to the fiber of $T^*\H^3$, this becomes
$$
v^*d\phi + d\phi(J dv \circ j) - \beta\cdot d\phi (X_H(v)) = 0
$$
which is equivalent to
\beastar
0 & = & v^*d\phi \circ j - d\phi\circ J\circ dv- v^*d\phi(X_H) \cdot(\beta \circ j)\\
& = &  v^*d\phi \circ j - v^*(d\phi\circ J)- v^*d\phi(X_H) \cdot(\beta \circ j)
\eeastar
By using $\varphi = \phi \circ v$ and taking the differential, we get
\be\label{eq:ddjvarphi}
- d(d\varphi \circ j) = - v^*(d(d\phi\circ J)) - d\big(v^*d\phi(X_H)\big) \wedge (\beta \circ j)
\ee
Now we compute the two terms in the right hand side separately. We first compute
\bea \label{eq:dJphi}
- d\phi \circ J & = &-d(\frac{x}{z})\circ J = -\left(\frac{zdx - xdz}{z^2}\right) \circ J \nonumber \\
& = & -\left(\frac{zdx\circ J - xdz\circ J}{z^2}\right)\nonumber \\
& = & zV^1 - xV^3 \nonumber\\
& = &  (zp_x - p_zdx + p_x dz - x dp_z) - \frac{x}{z}\theta\nonumber\\
& = & - d(x p_z - z p_x)- \frac{x}{z} \theta .
\eea
Therefore we obtain
\be\label{eq:ddJphi}
-d(d\phi\circ J) = - d(\frac{x}{z}) \theta - \frac{x}{z}d\theta =\phi\,\omega - d\phi \wedge \theta
\ee
and hence
\be\label{eq:v*ddJphi}
v^*(dd^J\phi) =\varphi v^*\omega - d\varphi \wedge v^*\theta.
\ee
For the second, we compute
$$
d\phi(X_H) = \frac{zdx - xdz}{z^2}(X_H) = \frac{z^3p_x - z^2xp_z}{z^2} = zp_x - xp_z.
$$
Therefore
$$
d\big(v^*(d\phi(X_H)\big) \wedge (\beta \circ j) = v^*d(zp_x - xp_z) \wedge \beta\circ j.
$$
Using \eqref{eq:dJphi}, we evaluate
\bea\label{eq:2ndterm}
 &{}&v^*d(zp_x - xp_z) \wedge \beta\circ j\nonumber\\
 & = & - v^*d(zp_x - xp_z) \circ j  \wedge \beta \nonumber\\
 &=& (d\phi \circ J - \phi\, \theta) (dv \circ j) \wedge \beta
 \nonumber\\
& = & (d\phi \circ J - \phi\, \theta)\big(Jdu + (\beta \circ j)\cdot X_H(v)- \beta \cdot  JX_H(v)\big) \wedge \beta\nonumber\\
& = & (d\phi \circ J - \phi\, \theta)(Jdu) \wedge \beta + (d\phi \circ J - \phi\, \theta)(X_H(v))\cdot (\beta\circ j) \wedge \beta
\nonumber\\
& = & (-d(v^*\phi) + \phi\, v^*dH) \wedge \beta  + \big(d\phi(JX_H(v)) - \phi(v) \theta(X_H(v))\big)\cdot (\beta \circ j) \wedge \beta
\nonumber\\
& = & \varphi v^*dH \wedge \beta - d\varphi \wedge \beta - 2\varphi H(v)  (\beta \circ j) \wedge \beta
\eea
Subtracting this from \eqref{eq:v*ddJphi}, we have derived the following key identity from \eqref{eq:ddjvarphi}
This finishes the proof of Proposition \ref{prop:max-principle-z/x}.

\end{document}